\documentclass[a4paper, reqno, 11pt]{amsart}
\usepackage[utf8]{inputenc}
\usepackage{tabularx}
\usepackage{exscale}
\usepackage[mathscr]{euscript}
\usepackage{url}
\usepackage[all,ps,tips,tpic]{xy}
\usepackage{newtxtext}
\usepackage{mathrsfs}
\DeclareMathAlphabet{\mathcal}{OMS}{cmsy}{m}{n}

\RequirePackage{xspace}
\RequirePackage{etoolbox}
\RequirePackage{varwidth}
\RequirePackage{enumitem}
\RequirePackage{tensor}
\RequirePackage{mathtools}
\RequirePackage{longtable}
\RequirePackage{multirow}

\usepackage[T1]{fontenc}
\usepackage{CFmath}
\usepackage{tikz}
\usepackage{tikz-cd}
\usepackage{amsmath}
\usepackage{amsfonts}
\usepackage{amssymb}
\usepackage{color}
\usepackage{array}
\usepackage{cancel}
\usepackage{extarrows}
\usepackage{graphicx}
\usepackage{parallel}
\usepackage{yhmath}
\usetikzlibrary {quotes,calc,positioning} 
\usepackage[normalem]{ulem}
\makeatletter
\patchcmd{\@maketitle}{\LARGE}{\Large}{}{}
\makeatother

{\raise0.5ex\hbox{$#1$}\! \left/ \! \lower0.5ex\hbox{$#2$}\right.}

\newtheorem{mainthm}{Theorem}

\begin{document}
\title{Monodromy Groups of Supersingular Abelian Surfaces over \ensuremath{\qp}}
%Moduli of Supersingular Abelian Surfaces
\author{\small{MOQING CHEN}}
\date{\today}
\maketitle
\begin{abstract}
For primes $p\ge 7$, we give a parametrization of the filtered $\varphi$-modules attached to the $p$-adic Tate modules of abelian surfaces over $\qp$ with supersingular good reduction. We use this classification to determine the neutral components of the monodromy groups of the associated $p$-adic representations up to $\bqp$-isomorphism. Furthermore, we analyze the $p$-adic distribution of these groups in the moduli space of filtered $\varphi$-modules. In particular, we prove that the neutral components are generically isomorphic to $\GL_2\times_{\det}\GL_2$.
\end{abstract}
\tableofcontents

\section{Introduction}
\subsection{Main classification result}
In this paper we study the monodromy groups of the $p$-adic crystalline representations arising as the $p$-adic Tate modules of abelian surfaces over $\qp$ with supersingular good reduction. Let $\mathbf{AbSurf}^{\sss}_{\qp}$ be the full subcategory of the category of abelian varieties over $\qp$ consisting of such abelian surfaces. Via Fontaine's functor $D_{\cris}$, we work in the category $\mf_{\qp}^{\f}$ of admissible filtered $\varphi$-modules: 
$$% https://tikzcd.yichuanshen.de/#N4Igdg9gJgpgziAXAbVABwnAlgFyxMJZABgBpiBdUkANwEMAbAVxiRAB12BbOnACwBGAM2ABBAQGUmAJyEBfAHrA4cOQH1gnAI5o5IOaXSZc+QigCM5KrUYs2nAEoBjaVjib2OuQApOAcTouHg1tXQBKfUMQDGw8AiIAJitqemZWRA5uIRDPXSVOeX1rGCgAc3giUCFpCC4kMhAcCCRLGzS2ADU1NG8AWgiDKpq6xFampCS2uwyAERyXNz05CjkgA
\begin{tikzcd}
\mathbf{AbSurf}^{\sss}_{\qp} \arrow[r, "V_p"] & \Rcris{\qp}(\Gamma_{\qp}) \arrow[r, "D_{\cris}"] & \mf_{\qp}^{\f}
\end{tikzcd}.$$
Denote the essential image of the functor $D_{\cris}\circ V_p$ by $\mfs$. We give an explicit classification of objects in $\mfs$ and their associated monodromy groups.

\begin{mainthm}[Theorem \ref{mainthm0}, Theorem \ref{mainclass}]\label{introthm1}
    Let $p\ge 7$. The objects in $\mfs$ are precisely the filtered $\varphi$-modules $D^{prod}_{\epsilon'}$, $D^{\epsilon,iso}_{\epsilon'}$, $D_{a'}^{\epsilon,\nu}$ or $D_{(a,b)}^{\epsilon,\mu}$ constructed in Definitions \ref{prodcase} and \ref{canfam}, with parameters satisfying the explicit arithmetic conditions given in Theorem \ref{mainclass}.
\end{mainthm}
Based on this classification, we determine the $p$-adic algebraic monodromy groups of supersingular abelian surfaces over $\qp$. Denote by $H_D$ the monodromy group of $D\in\mfs$ and by $H_D^\circ$ the neutral connected component.
\begin{mainthm}[Theorem \ref{mainthm}]\label{introthm2}
    Fix a prime $p\ge 7$. For $\epsilon'\in\st{\pm1}$, $\epsilon\in\st{0,\pm1}$ and $a',a,b\in\qp$ with $ab\neq-1$, we have the following isomorphisms of algebraic groups over $\bqp$.
    \begin{enumerate}
        \item $H_{D^{prod}_{\epsilon'}}^\circ\simeq\g_{m}^2$,
        \item $H_{D^{\epsilon,\nu}_{a'}}^\circ\simeq\begin{cases}
        \g_m^3 &\mbox{if }a'=\epsilon=0\\
        \GL_2\times_{\det}\GL_2&\mbox{otherwise},
    \end{cases}$
        \item $H_{D^{\epsilon,\mu}_{(a,b)}}^\circ\simeq\begin{cases}
        \g_m^2&\mbox{if } a=b=0\\
        \g_a^2\rtimes_g\g_m^2 &\mbox{if }c(a,b)=-\epsilon p,\mbox{ and }a\neq 0\mbox{ or }b\neq 0\\
        \GL_2\times_{\det}\GL_2&\mbox{otherwise},
    \end{cases}$
        \item $H_{D^{\epsilon,iso}_{\epsilon'}}^\circ\simeq\GL_2$
    \end{enumerate}
    
    where $c(a,b)=-\frac{a^2+\epsilon p+b^2p^2}{ab+1}$ and $g(s,t)=\br{\begin{smallmatrix}
            s & 0\\
            0 & t
        \end{smallmatrix}}$. In particular, for any object $A$ in $\mathbf{AbSurf}^{\sss}_{\qp}$, the neutral connected component $G_{V_p(A)}^\circ$ of the $p$-adic algebraic monodromy group associated with $A$ is $\bqp$-isomorphic to one of the following:
    $$\mathbf G_{m}^2,~\mathbf G_{m}^3,~\mathbf G_{a}^2\rtimes_g\mathbf G_{m}^2,~\GL_{2},~\GL_{2}\times_{\det}\GL_{2}.$$
\end{mainthm}
As a consequence, we obtain an explicit semisimplicity criterion (see Corollary \ref{semisimplicity}). The exceptional non-semisimple objects are listed there, and may be normalized as in Remark \ref{normalization}.
\subsection{A moduli space description}
The families occurring in Theorem \ref{introthm1} are parametrized by "algebraic" subvarieties in a certain coarse moduli space of filtered $\varphi$-modules. Fix $\epsilon\in\st{0,\pm1}$. We focus on the objects $D\in\mfs$ such that $\chi_{\varphi_D}(X)=X^4+\epsilon pX^2+p^2$, where $\varphi_D$ denotes the crystalline Frobenius of $D$. We denote the set of the isomorphism classes of these objects by $\mwa$. We show that $\mwa$ is in natural bijection (up to finitely many points) with the $\qp$-points of a projective GIT quotient of an open subvariety $\dm_c$ of $\Gr(2,4)$.
\begin{mainthm}[Theorem \ref{propmain}]\label{introthm3}
    There exists a commutative diagram
    $$% https://tikzcd.yichuanshen.de/#N4Igdg9gJgpgziAXAbVABwnAlgFyxMJZABgBpiBdUkANwEMAbAVxiRAE0A9YAHR6gBmAXwAUfAI5oAlCCGl0mXPkIoAjOSq1GLNnwC2AIzSdVYnpJlyF2PASJlVm+s1aIQ+gO51ZmmFADm8ESgAgBOEHpIZCA4EEgATNQAFjB0UGw4HhApaQhWIGERUdSxSOpaLmwAxrLyBeGRiIkxcYjlDFhgriBQdHAp6dTOOm58BnShAAQ1QhRCQA
\begin{tikzcd}
\dm_c(\qp) \arrow[d, two heads] \arrow[r, "c"] & \mbp^1(\qp) \\
\mwa \arrow[ru, "\bar c", dashed]                &            
\end{tikzcd}$$
where
\begin{enumerate}
    \item $c$ is algebraic and becomes a projective GIT quotient after base change to $\bqp$;
    \item if the polynomial $X^2+\epsilon pX+p^2$ has no root in $\qp$, the map $\bar c$ is injective;
    \item if the polynomial $X^2+\epsilon pX+p^2$ splits over $\qp$, the map $\bar c$ is injective outside $\bar c^{-1}([-\epsilon p\!:\!1])$ and ${\bar c}^{-1}([-\epsilon p\!:\!1])$ consists of three points.
\end{enumerate}
\end{mainthm}

For a more concrete study, we construct a closed immersion $\mbp^2\to\dm_c$ such that the induced map $\pi:\mbp^2(\qp)\to\mwa$ is surjective.
$$% https://tikzcd.yichuanshen.de/#N4Igdg9gJgpgziAXAbVABwnAlgFyxMJZABgBpiBdUkANwEMAbAVxiRAB12BbAIzQD0ATAApOARzQBKEAF9S6TLnyEUARlKqqtRizYBNfsE5QAZjNHsJ0uQux4CRQRq31mrRB259+qi1dnyIBh2ykTqgi467p5cAO50slowUADm8ESgJgBOEFxI6iA4EEgAzNQAFjB0UGw4sRCV1Qg2INm5+dRFSE7abmwAxgGZOXmIZYXFiD0MWGDRUHRwlTXUrroenDx0WQAEgy1to2QTHb3rnvg4CQcjSMddYxVVNR51Dc8Iq1FsnGhYINQGHQeDAGAAFRT2FQgLJYFLlHBDVq3RD3SY9NbRfqcfpYLLY9iXa4UGRAA
\begin{tikzcd}
\mbp^2(\qp) \arrow[rd, "\iota"] \arrow[rdd, "\pi"', two heads] \arrow[rrd, "c\circ\iota"] &                                                  &             \\
                                                                                          & \dm_c(\qp) \arrow[d, two heads] \arrow[r, "c"] & \mbp^1(\qp) \\
                                                                                          & \mwa \arrow[ru, "\bar c", dashed]                &            
\end{tikzcd}$$
This approach clarifies the structure of the map $\bar c$. When $X^2+\epsilon pX+p^2$ has two different roots $\mu_{1},\mu_2$ over $\qp$, the preimage $(c\circ\iota)^{-1}([-\epsilon p\!:\!1])$ consists of two projective lines $l_1,l_2$ in $\mbp^2(\qp)$ intersecting at a point $o$, and the preimages of the three points in ${\bar c}^{-1}([-\epsilon p\!:\!1])$ under $\pi$ are $l_1\backslash\st{o},l_2\backslash\st{o}$ and $\st{o}$. In particular, we see that under the quotient topology of the Zariski topology, the points $\pi(x_i)$ specialize to $\pi(o)$ where $x_i\in l_i\backslash\st{o}$, $i=1,2$. 

Moreover, we have natural lifts of the families $D^{\epsilon,iso}_{\epsilon'}$, $D_{a'}^{\epsilon,\nu}$ and $D_{(a,b)}^{\epsilon,\mu}$ (which are defined on $\mwa$) along $\pi$. We construct points $P_{\pm1}\in\mbp^2(\qp)$ and morphisms
    $$\mba^1\stackrel{\nu}\to\mbp^2\qaq\mba^2\backslash\st{ab=-1}\stackrel{\mu}\to\mbp^2$$
such that $[D_{\pm1}^{\epsilon,iso}]=\pi(P_{\pm1})$ and the diagrams
$$
% https://tikzcd.yichuanshen.de/#N4Igdg9gJgpgziAXAbVABwnAlgFyxMJZARgBoAGAXVJADcBDAGwFcYkQAdDgWwCM0AegCYAFFwCOaAJQgAvqXSZc+QijLFqdJq3ZduAd3pyFIDNjwEi5CpoYs2iTj170BxMR0kzZmmFADm8ESgAGYAThDcSNYgOBBIZCAAFjD0UOw4+hApaQg0djqOXGhYxqERUYhCNHEJNIz0vDCMAApKFqogYVj+STgg+doOIAAiAsBcMGjYjASkXGDMsgD6wAC0smUg4ZFI1bHxiDEFwwvMcpSyQA
\begin{tikzcd}
\mba^1(\qp) \arrow[rd, "{D^{\epsilon,\nu}_{-}}"'] \arrow[r, "\nu"] & \mbp^2(\qp) \arrow[d, "\pi", two heads] \\
                                                                   & \mwa                                   
\end{tikzcd}
\qaq
% https://tikzcd.yichuanshen.de/#N4Igdg9gJgpgziAXAbVABwnAlgFyxMJZARgBoAGAXVJADcBDAGwFcYkQAdDgWwCM0AegCYAFFwCOaAJQgAvqXSZc+QijLFqdJq3ZduAd3pyFIDNjwEi5CpoYs2iTj171hYjpJmzNMKAHN4IlAAMwAnCG4kaxAcCCQyEAALGHoodhx9CGTUhBo7HUcuNCxjEPDIxCEaWPiaRnpeGEYABSULVRBQrD9EnBA87QcQABEBYC4YNGxGAlI9ZlkAfWAAWllSkDCIpCqYuMRo-KH5uUpZIA
\begin{tikzcd}
\mba^2(\qp)\backslash\st{ab=-1} \arrow[rd, "{D^{\epsilon,\mu}_{-}}"'] \arrow[r, "\mu"] & \mbp^2(\qp) \arrow[d, "\pi", two heads] \\
                                                                   & \mwa                                   
\end{tikzcd}$$
commute. Here, $[D]\in\mwa$ denotes the isomorphism classes of $D\in\mf_{\qp}^{\f}$.

In terms of our moduli space description, we have the following corollary of Theorem \ref{introthm2}.
\begin{mainthm}[Theorem \ref{distmono}]\label{introthm4}
    For each point $P\in\mwa$, let $H_{D_P}$ denote the monodromy group of any filtered $\varphi$-module representing the isomorphism class corresponding to $P$. Then, we have $H_{D_{P}}^\circ\simeq_{\bqp}\GL_{2}\times_{\det}\GL_{2}$ for all but finitely many $P\in\mwa$. The exceptions are 
    $$H_{D_{P}}^\circ\simeq_{\bqp}\begin{cases}
        \GL_{2}&\bar c(P)=[\pm2p\!:\!1]\\
        \g_{m}^2~\mbox{or}~\g_{a}^2\rtimes\g_{m}^2&\bar c(P)=[-\epsilon p\!:\!1]\\
        \g_{m}^3 &\epsilon=0\mbox{ and }\bar c(P)=[1:0]
    \end{cases}$$
    where $g(s,t)=\br{\begin{smallmatrix}
        s & 0\\
        0 & t
    \end{smallmatrix}}$. When $\bar c(P)=[-\epsilon p\!:\!1]$, in the context of Theorem \ref{introthm3}, the case $H_{D_P}^\circ\simeq_{\bqp}\g_{m}^2$ occurs when $P=\pi(o)$, and $H_{D_P}^\circ\simeq_{\bqp}\g_{a}^2\rtimes_g\g_{m}^2$ when $P=\pi(x_i)$ where $x_i\in l_i\backslash\st{o}$, $i=1,2$.
\end{mainthm}

The distribution of monodromy group can be visualized when pulled back to $\mbp^2(\qp)$ (based on Remark \ref{diagramp2}). For example, when $\epsilon=0$ we have the following description:
\begin{center}
\tikzset{every picture/.style={line width=0.75pt}} %set default line width to 0.75pt        

\begin{tikzpicture}[x=0.75pt,y=0.75pt,yscale=-1,xscale=1]

%Shape: Arc [id:dp359561006969801] 
\draw  [draw opacity=0] (264.01,249.64) .. controls (212.32,248.74) and (170.69,206.78) .. (170.69,155.15) .. controls (170.69,103.28) and (212.7,61.18) .. (264.72,60.66) -- (265.69,155.15) -- cycle ; \draw   (264.01,249.64) .. controls (212.32,248.74) and (170.69,206.78) .. (170.69,155.15) .. controls (170.69,103.28) and (212.7,61.18) .. (264.72,60.66) ;  
%Shape: Arc [id:dp37708628630220264] 
\draw  [draw opacity=0][dash pattern={on 4.5pt off 4.5pt}] (264.68,249.65) .. controls (265.02,249.65) and (265.35,249.65) .. (265.69,249.65) .. controls (318.16,249.65) and (360.69,207.34) .. (360.69,155.15) .. controls (360.69,102.96) and (318.16,60.65) .. (265.69,60.65) .. controls (265.02,60.65) and (264.36,60.66) .. (263.69,60.67) -- (265.69,155.15) -- cycle ; \draw  [dash pattern={on 4.5pt off 4.5pt}] (264.68,249.65) .. controls (265.02,249.65) and (265.35,249.65) .. (265.69,249.65) .. controls (318.16,249.65) and (360.69,207.34) .. (360.69,155.15) .. controls (360.69,102.96) and (318.16,60.65) .. (265.69,60.65) .. controls (265.02,60.65) and (264.36,60.66) .. (263.69,60.67) ;  
%Shape: Circle [id:dp34075134049622724] 
\draw  [color={rgb, 255:red, 0; green, 0; blue, 0 }  ,draw opacity=1 ][fill={rgb, 255:red, 0; green, 0; blue, 0 }  ,fill opacity=1 ] (197.68,87.3) .. controls (197.68,85.93) and (198.8,84.81) .. (200.17,84.81) .. controls (201.55,84.81) and (202.67,85.93) .. (202.67,87.3) .. controls (202.67,88.68) and (201.55,89.8) .. (200.17,89.8) .. controls (198.8,89.8) and (197.68,88.68) .. (197.68,87.3) -- cycle ;
%Shape: Circle [id:dp07626022594571147] 
\draw  [color={rgb, 255:red, 0; green, 0; blue, 0 }  ,draw opacity=1 ][fill={rgb, 255:red, 0; green, 0; blue, 0 }  ,fill opacity=1 ] (197.68,223.3) .. controls (197.68,221.93) and (198.8,220.81) .. (200.17,220.81) .. controls (201.55,220.81) and (202.67,221.93) .. (202.67,223.3) .. controls (202.67,224.68) and (201.55,225.8) .. (200.17,225.8) .. controls (198.8,225.8) and (197.68,224.68) .. (197.68,223.3) -- cycle ;
%Straight Lines [id:da6753982195377196] 
\draw    (235.17,65.8) -- (295.67,244.67) ;
%Straight Lines [id:da5109488059224139] 
\draw    (293.67,65.67) -- (237.67,244.67) ;
%Shape: Circle [id:dp6725561032214642] 
\draw  [color={rgb, 255:red, 0; green, 0; blue, 0 }  ,draw opacity=1 ][fill={rgb, 255:red, 255; green, 255; blue, 255 }  ,fill opacity=1 ] (261.67,156.39) .. controls (261.67,154.03) and (263.57,152.13) .. (265.92,152.13) .. controls (268.27,152.13) and (270.18,154.03) .. (270.18,156.39) .. controls (270.18,158.74) and (268.27,160.64) .. (265.92,160.64) .. controls (263.57,160.64) and (261.67,158.74) .. (261.67,156.39) -- cycle ;
%Shape: Circle [id:dp3983651749340189] 
\draw  [color={rgb, 255:red, 0; green, 0; blue, 0 }  ,draw opacity=1 ][fill={rgb, 255:red, 0; green, 0; blue, 0 }  ,fill opacity=1 ] (263.43,156.39) .. controls (263.43,155.01) and (264.55,153.89) .. (265.92,153.89) .. controls (267.3,153.89) and (268.42,155.01) .. (268.42,156.39) .. controls (268.42,157.76) and (267.3,158.88) .. (265.92,158.88) .. controls (264.55,158.88) and (263.43,157.76) .. (263.43,156.39) -- cycle ;
%Curve Lines [id:da828568591639107] 
\draw    (174,130) .. controls (214,100) and (291.67,253) .. (331.67,223) ;
%Straight Lines [id:da6374058719418172] 
\draw    (257.67,33) -- (246.67,90) ;
%Straight Lines [id:da4762422915028409] 
\draw    (270.67,33) -- (283.67,90) ;
%Straight Lines [id:da10877582916235773] 
\draw    (274,160) -- (370.67,206) ;

% Text Node
\draw (235,100.4) node [anchor=north west][inner sep=0.75pt]    {$l_{1}$};
% Text Node
\draw (252,151.4) node [anchor=north west][inner sep=0.75pt]    {$o$};
% Text Node
\draw (283,100.4) node [anchor=north west][inner sep=0.75pt]    {$l_{2}$};
% Text Node
\draw (158,72.4) node [anchor=north west][inner sep=0.75pt]    {$\GL_{2}$};
% Text Node
\draw (197,147.4) node [anchor=north west][inner sep=0.75pt]    {$\g_{m}^3$};
% Text Node
\draw (373,201.4) node [anchor=north west][inner sep=0.75pt]    {$\g_{m}^2$};
% Text Node
\draw (158,217.4) node [anchor=north west][inner sep=0.75pt]    {$\GL_{2}$};
% Text Node
\draw (233,11.4) node [anchor=north west][inner sep=0.75pt]    {$\g_{a}^2\rtimes\g_{m}^2$};
% Text Node
\draw (296,129.4) node [anchor=north west][inner sep=0.75pt]    {$\GL_{2}\times_{\det}\GL_{2}$};
\draw (102,257) node [anchor=north west][inner sep=0.75pt]   [align=left] {Distribution of monodromy groups on $\mbp^2(\qp)$ when $\epsilon=0$};

\end{tikzpicture}
\end{center}
\subsection{Distribution of Wintenberger types}
In the proof, we use a refinement of Hodge decomposition in the proof,
$$D=\bigoplus_{\xi\in X_{\p}}D_{\xi}$$
which was introduced in \cite{wintenberger1984scindage}. Here, $X_{\p}$ is the abelian group of periodic functions from $\mz$ to itself, and we call the subset consisting of the $\xi$ that appear in this decomposition the \textit{Wintenberger type} of $D$, denoted $X_D$. We study the distribution of $X_{D_P}$ for $P\in\mwa$, where $D_P\in\mf_{\qp}^{\f}$ represents the isomorphism class associated to $P$.
\begin{mainthm}[Corollary \ref{distwintype}]\label{introthm5}
Let $P\in\mwa$. The Wintenberger type of $D_P$ satisfies
    $$X_{D_P}=\begin{cases}
    (A)\quad &v_p(\bar c(P))\le 0\\
    (B) &v_p(\bar c(P))\ge 1
\end{cases}$$
where $(A)$ and $(B)$ are given in Lemma \ref{listxd} and conventionally, we set $v_p([a\!:\!1]):=v_p(a)$ and $v_p([1\!:\!0]):=-\infty$.
\end{mainthm}
\subsection{Strategy of proof}
Throughout this subsection, we fix an unramified finite extension $K$ of $\qp$. The structure of any admissible filtered $\varphi$-module over $K$ is determined by its
$$\mbox{crystalline Frobenius}~\varphi \qaq \mbox{Hodge cocharacter}\footnote{The cocharacter associated to the Hodge decomposition \cite[Theorem 3.1.2]{wintenberger1984scindage}.}~ \mu:\g_{m}\to\GL_{4,K}.$$
When $K=\qp$, the crystalline Frobenius is linear and hence determined by its characteristic polynomial, which must be a supersingular $p$-Weil polynomial of degree $4$. We prove that, for $p\ge 7$, such polynomials are precisely
$$(X^2\pm p)^2\qaq X^4+\epsilon pX^2+p^2,~\epsilon\in\st{0,\pm1}.$$
Moreover, the conjugacy class of the Hodge cocharacter is fixed, as the $p$-adic Tate module of any abelian surface has Hodge numbers $h^{1,0}=h^{0,1}=2$. There are technical constraints on the relative position of $\varphi$ and $\mu$ inside $\GL_4(\qp)$ \cite[Introduction]{wintenberger1984scindage}, which translate into the arithmetic conditions in Theorem \ref{introthm1}. Finally, we complete the proof of Theorem \ref{introthm1} by verifying the conditions required in Volkov's characterization of filtered $\varphi$-modules arising from abelian varieties with good reduction over $\qp$ \cite[Corollary 5.9]{volkov2005class}.

In the proof of Theorem \ref{introthm2}, one of the main ingredients is a density result due to Pink \cite[Proposition 2.5]{pink1998ℓ}, which asserts that for any admissible filtered $\varphi$-module $D\in\mf^{\f}_K$, the conjugations of $\mu(K^\times)$ by $\varphi^i$ together with $\varphi^{[K:\qp]}$ generates a Zariski dense subgroup in $H_D$. We also use the algebraic Lie correspondence \cite[Corollary 10.17]{milne2017algebraic} in our computations.

Theorem \ref{introthm3} is established by proving that the isomorphism class of $D^{\epsilon,\nu}_{a'}$ (resp. $D^{\epsilon,\mu}_{(a,b)}$) is determined by $c(\iota(\nu(a')))$ (resp. $c(\iota(\mu(a,b)))$), except for the case $c=[-\epsilon p\!:\!1]$, which requires a special discussion.
$$
% https://tikzcd.yichuanshen.de/#N4Igdg9gJgpgziAXAbVABwnAlgFyxMJZARgBoAGAXVJADcBDAGwFcYkQAdDgWwCM0AegCYAFFwCOaAJQgAvqXSZc+QiiGli1Ok1bsAmgOBcoAM1liOkmfMXY8BIgGYNWhizaJOPfgOIWrcgogGHYqROpCrjoeXtwA7vSBtsoOKOQuNG66nlx89ML+0ly89ADGANZwjPRwABZccDjA9LwAvAC0xLJJwUr2qiQZ2u7suSW+hdZaMFAA5vBEoCYAThDcSGQgOBBIziC1MPRQ7DhxEAdHCDYgK2sbNNtI6sPZIKU9t+uIe4+Iz4xYMAxKA1A7HTLRUYcErLAAE72unyQ6S2O0QmyyMS4+BwiURqy+KN+ewux08p3OhygCAhIxyHDQWBANGqvBgjAACn0wp5llhZrUcB8CciHmjnpj2KUuKUsMtpRwcXigkjEAAWMWil5YnjMYV3dWa76014AEQA+sB2rJDFwYGhsIwCKRcsxuiyWuyuaFUiA+QKhfiDQBWI0oyX0sB6oNfUOo3YmmIWq02owce2O51cKPukCsr3c32A7CwOSUWRAA
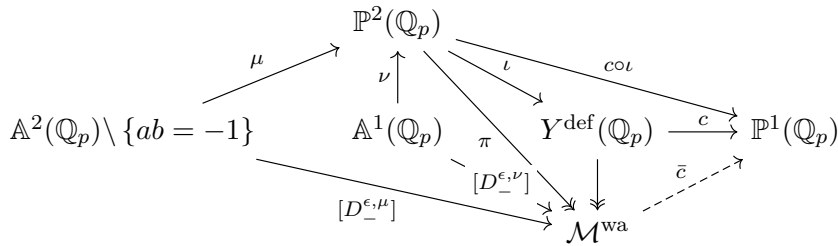
\begin{tikzcd}
                                                                                         & \mbp^2(\qp) \arrow[rd, "\iota"] \arrow[rdd, "\pi"', two heads] \arrow[rrd, "c\circ\iota"] &                                                  &             \\
\mba^2(\qp)\backslash\st{ab=-1} \arrow[ru, "\mu"] \arrow[rrd, "{[D_{-}^{\epsilon,\mu}]}"'] & \mba^1(\qp) \arrow[u, "\nu"] \arrow[rd, "{[D_{-}^{\epsilon,\nu}]}" description]             & \dm_c(\qp) \arrow[d, two heads] \arrow[r, "c"] & \mbp^1(\qp) \\
                                                                                         &                                                                                           & \mwa \arrow[ru, "\bar c", dashed]                &            
\end{tikzcd}
$$
Theorem \ref{introthm4} and \ref{introthm5} are direct corollaries of Theorem \ref{introthm2} and \ref{introthm3}.

\subsection{Relation to previous work and broader motivation}
Our study provides examples of filtered $\varphi$-modules arising from $p$-adic geometry, which is inspired by and generalizes the work of Volkov \cite{volkov2023abelian}. In \cite{volkov2023abelian} \S2 and \S3, Volkov constructed two filtered $\varphi$-modules and claimed that they are both non-semisimple, which is equivalent to saying that the associated monodromy groups are both non-reductive. We can show that these two filtered $\varphi$-modules are isomorphic to $D^{prod}_{-1}$ and $D^{1,\mu}_{(-\zeta_3^{-1}p,1)}$ respectively. The neutral components of the monodromy groups of these two filtered $\varphi$-modules are $\mathbf G_{m}^2$ and $\mathbf G_{a}^2\rtimes_g\mathbf G_{m}^2$ respectively. The second group is non-reductive, reproving Proposition 3.1 (b) in \cite{volkov2023abelian}, while the first is reductive, in contrast with Proposition 2.1 (b) in \cite{volkov2023abelian}\footnote{In fact, we found a mistake in her proof of Proposition 2.1. Using the notation in this proof, she claimed that the nontrivial $\varphi$-stable subspaces of $D$ are precisely the $D_i$. However, since $\varphi^2=-p$ in this case, for any nonzero vector $v\in D$, the $2$-dimensional subspace of $D$ spanned by $v$ and $\varphi(v)$ is $\varphi$-stable. Therefore, $D_1$ actually has a complement subobject in $D$. Indeed, setting $z:=y_1+x_2$, the subspace of $D$ spanned by $z$ and $\varphi(z)$ is such a complement. A similar argument implies that every subobject of $D$ has a complement and $D$ is semisimple.}.

Theorem \ref{introthm2} lies in the general framework of algebraic monodromy groups arising from geometry, defined as the Zariski closure of the Galois action on \'etale cohomology groups. When the base field is a number field, these groups are deeply studied. For instance, they satisfy certain $\ell$-independence properties (see \cite[\S5]{serre2000collected}, \cite{larsen1992independence}) and are large for low-dimensional abelian varieties with constraints on their endomorphism rings (see \cite[IV 2.2]{serre1997abelian}, \cite[\S5]{serre2000collected}).

The case where the base field is a mixed characteristic $p$-adic field and $\ell=p$ is more delicate. The central case concerns the monodromy group of the $p$-adic Tate module of an abelian variety with good reduction over $K$. Pink \cite[Proposition 2.9]{pink1998ℓ} proved that when the reduction is ordinary, this $p$-adic algebraic monodromy group is always solvable. When the reduction is supersingular, Serre \cite{serre1997abelian} showed that for elliptic curves, the $p$-adic monodromy group is either $\GL_2$ or a non-split Cartan subgroup of $\GL_2$. Theorem \ref{introthm2} may be viewed as an extension of Serre's result.

Another motivation for studying the distribution of $p$-adic monodromy groups is the theory of Hecke operators on Shimura varieties, since monodromy groups are stable under the Hecke action. For instance, the determination of local monodromy groups of $F$-isocrystals was one of the main ingredients for the resolution of Chai--Oort Hecke orbit conjecture in
positive characteristic \cite{d2022hecke}.

Yu Fu recently formulated a $p$-adic nowhere density conjecture \cite[Conjecture 1.3]{fu2025padicdensityconjecturehecke} for Hecke orbits of positive-dimensional subvarieties in Hodge type Shimura varieties. To provide evidence for this conjecture, she proves the $p$-adic density of the "big monodromy" locus in the formal neighborhood of certain points $x\in Sh(\overline{\mathbb F}_p)$ \cite[Theorem 1.5]{fu2025padicdensityconjecturehecke}. In our work, we study instead the variation of monodromy groups on the coarse moduli space $\mwa$ of admissible filtered $\varphi$-modules with isomorphic $\varphi$-action. These two spaces are related by Grothendieck--Messing type period maps. 

\subsection*{Outline of the article}
In \S\ref{pre} we recall the necessary background on filtered $\varphi$-modules and the Tannakian formalism. 

In \S\ref{classphimodule} we establish the classification Theorem \ref{introthm1}. 

In \S\ref{classmonodromy} we compute the associated monodromy groups and prove Theorem \ref{introthm2}. 

In \S\ref{moduli} we study the space $\mwa$, and we prove Theorem \ref{introthm3}--\ref{introthm5}.

\subsection*{Acknowledgement}
I am grateful to my advisor Marco d'Addezio, who suggested that I study $p$-adic local monodromy groups coming from geometry. I also thank him for helpful discussions on this topic and for the many constructive suggestions he gave on the draft of this article. I thank Maja Volkov, who kindly answered my questions on filtered $\varphi$-modules arising from abelian surfaces in her work. I have also greatly benefited from discussions with Yu Fu, who is likewise interested in $p$-adic monodromy groups arising from deformations of supersingular abelian varieties. This work was supported in part by the Agence Nationale de la Recherche (ANR) under the project ANR-25-CE40-7869-01 (pDefi), and by China Scholarship Council (CSC) program (File No. 202306340018). 
\section{Preliminaries}\label{pre}
Let $K$ be a finite extension of $\qp$. We denote by $\Gamma_K$ the absolute Galois group of $K$. Let $K_0$ be the maximal unramified subextension inside $K$, we have a Frobenius lift $\sigma\in\gal(K_0/\qp)$.
\subsection{Crystalline representations and filtered $\varphi$-modules}
For a proper smooth variety $X$ over $K$ and $n\ge 0$, the $p$-adic \'etale cohomology groups $H_{\Acute et}^n(X_{\overline K},\qp)$ give rise to Galois representations of $\Gamma_K$. Unlike the $\ell$-adic case, these representations are typically highly ramified. Nevertheless, they are controlled by Fontaine's period ring formalism (see \cite[\S5]{brinon2009cmi}): if $X$ has good reduction, its $p$-adic cohomology groups are \textit{crystalline representations}; in the case of semistable reduction, one obtains \textit{semistable representations}; and in general one always obtains \textit{de Rham representations}. For our purposes, we will restrict attention to crystalline representations, which form a subcategory $\Rcris{\qp}(\Gamma_K)$ of the category $\Rep{\qp}(\Gamma_K)$ of finite-dimensional continuous $\qp$-representations of $\Gamma_K$ \cite[\S5, \S9]{brinon2009cmi}.

A central aspect of classical $p$-adic Hodge theory is the passage from categories of $p$-adic representations of $\Gamma_K$ to corresponding categories of semilinear algebraic objects. As an important example, the category of crystalline representations corresponds to the category of filtered $\varphi$-modules defined as follows.
\begin{defn}[{\cite{fontaine1979modules}}]
    A \textit{filtered $\varphi$-module} over $K$ is a triple $(D,\varphi,\fil_\bullet)$ where
    \begin{enumerate}
        \item $D$ is a finite-dimensional vector space over $K_0$,
        \item $\varphi$ is a $\sigma$-semilinear automorphism of $D$,
        \item $\fil_\bullet$ is a decreasing exhaustive and separated filtration on $D\otimes_{K_0}K$.
    \end{enumerate}
    A morphism between two filtered $\varphi$-modules over $K$ is a $K_0$-linear map which commutes with the $\varphi$-actions and respects the filtration after tensoring $K$. We denote by $\mf_{K}$ the category of filtered $\varphi$-modules over $K$.
\end{defn}
The category $\mf_K$ is not abelian. However, we can define a full subcategory $\mf_K^{\f}$ of $\mf_K$ which is abelian by introducing some additional properties relating the $\varphi$-action and the filtration called \textit{admissibility} \cite[\S8]{brinon2009cmi}, which requires the existence of a $\zp$-lattice $M\subset D$ such that $\sum_{i\in\mz}\frac{1}{p^i}(\fil_iD\cap M)=M$. The objects in $\mf_K^{\f}$ are called \textit{admissible filtered $\varphi$-modules} and $M$ is called an \textit{adapted lattice}.

Based on works of many, Fontaine and Colmez \cite{colmez2000construction} proved an equivalence of categories from $\mf_{K}^{\f}$ to the category $\Rcris{\qp}(\Gamma_K)$ of crystalline $\qp$-representations of $\Gamma_K$:
$$% https://tikzcd.yichuanshen.de/#N4Igdg9gJgpgziAXAbVABwnAlgFyxMJZABgBpiBdUkANwEMAbAVxiRAFkAxEAX1PUy58hFAEZyVWoxZsASjDS9JMKAHN4RUADMAThAC2SMiBwQk4qc1aIQANV78Qug+eqmj1OAAssWnEgBaC3orNgARJR4gA
\begin{tikzcd}
\mf_{K}^{\f} \arrow[r, "V_\cris"] & \Rcris{\qp}(\Gamma_K) \arrow[l, "D_\cris", shift left]
\end{tikzcd}.$$
Moreover, the natural forgetful functor $\mf_{K}^{\f}\stackrel{\omega}\to\vc_{K}$ (resp. $\Rcris{\qp}\stackrel{\eta}\to\vc_{\qp}$) makes $\mf_{K}^{\f}$ (resp. $\Rcris{\qp}$) into a Tannakian category (resp. a neutral Tannakian category) over $\qp$ in the sense of \cite{deligne2012tannakian}, and the functors $D_{\cris}$, $V_{\cris}$ are compatible with the $\otimes$-structures. For an object $D\in\mf_{K}^{\f}$, let $V=V_{\cris}(D)$ and $\ang{D}^\otimes$ (resp. $\ang{V}^\otimes$) be the Tannakian subcategory of $\mf_{K}^{\f}$ (resp. of $\Rcris{\qp}$) generated by $D$ (resp. by $V$). By Tannakian formalism \cite{deligne2012tannakian}, the automorphism groups $\aot(\omega|_D)$ and $\aot(\eta|_V)$ are representable by algebraic subgroups $H_D$ of $\GL_D$ and $H_V$ of $\GL_V$ respectively, which are called the monodromy groups of $D$ and of $V$. Moreover, $H_D$ and $H_V$ are inner forms of each other (see \cite[Theorem 3.2]{deligne2012tannakian}), which are in particular isomorphic over $\bqp$. On the other hand, the algebraic monodromy group of a crystalline representation $V$, defined as the Zariski closure of the image of $\rho_V$ inside $\GL_V$, is isomorphic to $H_V$, which hence can be computed from the corresponding filtered $\varphi$-module.

\subsection{Hodge decomposition for admissible filtered $\varphi$-modules}
Suppose throughout this subsection that the extension $K/\qp$ is unramified. Wintenberger \cite{wintenberger1984scindage} proved that any object $D\in\mf_{K}^{\f}$ enjoys a natural splitting of its filtration, called the \textit{Hodge decomposition} as an analogue of the Hodge decomposition in complex geometry, which is a functorial decomposition $D=\bigoplus_{i\in\mz}D_i$ such that for all $i\in\mz$, $\fil_iD=\bigoplus_{i'\ge i} D_{i'}$. Via Tannakian formalism, there arises a \textit{Hodge cocharacter} $\mu:\g_{m,K}\to\GL_D$ which factors through the monodromy group $H_D$ of $D$. Moreover, the splitting satisfies and is determined by the following conditions.

\begin{thm}[{\cite[Introduction and Theorem 3.1.2]{wintenberger1984scindage}}]\label{hodgedecomp}
    Let $D=(D,\varphi,\fil_\bullet)\in\mf_{K}^{\f}$ and $X_{\p}=\st{\xi:\mz\to\mz~\mbox{periodic}}$. There exists a unique pair $\br{\br{D_i}_{i\in\mz},u}$ where $D_i$ are $K$-linear subspaces of $D$, $u\in\GL(D)$ such that
    \begin{enumerate}[label=(\alph*)]
        \item $D=\bigoplus_{i\in\mz}D_i$ and for all $i\in\mz$, $\fil_iD=\bigoplus_{i'\ge i} D_{i'}$,
        \item $(u-\id_D)\bigoplus_{i'\le i} D_{i'}\subset\bigoplus_{i'\le i-1} D_{i'}$,
        \item $D=\bigoplus_{\xi\in X_{\p}}D_\xi$, where $D_\xi=\st{x\in D\ver(\F^{\el})^j(x)\in D_{\xi(j)}}$ with $\F^{\el}=u^{-1}\circ\varphi\circ\mu(p^{-1})$,
        \item For any adapted lattice $M$ of $D$, let $M_{\xi}=M\cap D_{\xi}$, $k$ be the residue field of $K$, $\bar M=M/pM$ and $\bar M_\xi=M_\xi/pM_\xi$.
        \begin{enumerate}[label=(\roman*)]
            \item We have $M=\bigoplus_{\xi\in X} M_{\xi}$, $\F^{\el}(M)=M$ and $u(M)=M$.
            \item Let $\bar\F^{\el}$ (resp. $\bar u$) be the residue of $\F^{\el}$ (resp. of $u$) on $\bar M$. Then, there exists a descending chain $0=L_0\subsetneq\cdots\subsetneq L_r=\bar M$ of $\bar\F^{\el}$-stable $k$-linear subspaces such that for $1\le r'\le r$, $L_{r'}=\bigoplus_{\xi\in X_{\p}}(L_{r'}\cap\bar M_\xi)$ and $(\bar u-\id_{\bar M})L_{r'}\subset L_{r'-1}$.
        \end{enumerate}
    \end{enumerate}
\end{thm}

The periodic functions $\xi\in X_\p$ that appear in the decomposition $D=\bigoplus_{\xi\in X_{\p}}D_\xi$ are clearly an invariant of $D$. For simplicity, we take the following definition.
\begin{defn}\label{wintype}
    Let $D=(D,\varphi,\fil_\bullet)\in\mf_{K}^{\f}$. The set $X_D:=\{\xi\in X_\p~\!|~\!D_\xi\neq 0\}$ is called the \textit{Wintenberger type} of $D$.
\end{defn}
\subsection{Filtered $\varphi$-modules arising from geometry}
For a proper and smooth variety $X$ over $K$ with good reduction, an important invariant is the characteristic polynomial $h_{X,n}(T)\in\ql[T]$ of a 
$\mathrm{mod}$-$p$ Frobenius lift acting on $H^{n}_{\et}(X_{\bar K},\ql)$, which is independent of $\ell\neq p$. By Weil conjectures, this polynomial is defined over $\mz$ and indeed is a $q^n$-Weil polynomial where $q$ is the cardinality of the residue field of $K$. 
\begin{defn}
    Let $q$ be a power of $p$. A $q$-Weil number (resp. supersingular $q$-Weil number) is an algebraic integer $\alpha$ such that all of its embeddings in $\mc$ admit absolute value $\sqrt{q}$ (resp. $\alpha$ is a product of $\sqrt{q}$ and a complex root of unity). A monic polynomial in $\mz[T]$ is called a $q$-Weil polynomial (resp. supersingular $q$-Weil polynomial) if all its roots in $\bar\mq$ are $q$-Weil numbers (resp. supersingular $q$-Weil numbers) and its valuation at $T^2-q$ is even. Let $\mathscr P_{\w,q}$ (resp. $\mathscr P^{\sss}_{\w,q}$) denote the set of all $q$-Weil polynomials (resp. supersingular $q$-Weil polynomials).
\end{defn}

As a consequence of the Riemann Hypothesis, $h_{X,n}$ can also be recovered from $D_{\cris}(H^{n}_{\et}(X_{\bar K},\qp))$, the filtered $\varphi$-module associated with the $p$-adic cohomology. If we denote by $\varphi$ the crystalline Frobenius of this filtered $\varphi$-module, the characteristic polynomial of $\varphi^{[K_0:\qp]}$ is exactly $h_{X,n}(T)$.

Conversely, when $K=\qp$, Volkov \cite{volkov2005class} gives a series of conditions which classify the admissible filtered $\varphi$-modules arising from abelian varieties with good reduction over $\qp$.

\begin{thm}[{\cite[Corollary 5.9]{volkov2005class}}]\label{dfromav}
    Let $D\in\mf_{\qp}^{\f}$. There is an abelian variety $A/\qp$ with good reduction such that $D=D_\cris(V_p(A))$ if and only if $D$ satisfies the following conditions.
    \begin{enumerate}
        \item the Hodge numbers satisfy $h^{0}(D)=h^{1}(D)=\frac{1}{2}\dim D$,
        \item the $\varphi$-action on $D$ is semisimple and $\chi_{\varphi}$ is a $p$-Weil polynomial,
        \item there exists a nondegenerate skew form on $D$ under which $\varphi$ is a $p$-similitude and $\fil_1D$ is totally isotropic.
    \end{enumerate}
\end{thm}

\subsection{Tools for computing monodromy groups}
For $K/\qp$ unramified, any object $D\in\mf_{K}^{\f}$ is determined by its crystalline Frobenius $\varphi$ together with the Hodge decomposition. By Tannakian theory \cite[Proposition 2.21 (a)]{deligne2012tannakian}, this could be illustrated by the following.

\begin{prop}[{\cite[Proposition 2.5]{pink1998ℓ}}]\label{zardense}
    Let $D=(D,\varphi,\fil_\bullet)\in\mf_{K}^{\f}$ and $\mu$ be the Hodge cocharacter. For $i\in\mz$, let $\mu_i:\g_{m,K}\to H_D$ be the unique cocharacter such that for all $t\in K^\times$, $\mu_i(\sigma^i(t))=\varphi^i\circ\mu(t)\circ\varphi^{-i}$. The subgroup of $H_D$ generated by $\varphi^m$ and the image of $\mu_i$ for all $i\in\mz$ is Zariski dense in $H_D$.
\end{prop}
The following lemma allows us to reformulate Proposition \ref{zardense} in terms of generation of Lie algebras.
\begin{lem}[{\cite[Corollary 10.17]{milne2017algebraic}}]\label{Liealg}
    Let $H_1,\cdots,H_n$ be smooth algebraic subgroups of a connected algebraic group $G$. If the Lie algebras of the subgroups $H_i$ generate the Lie algebra of $G$, then the subgroups $H_i$ generate $G$ as an algebraic group.
\end{lem}

The following technical lemma will be used repeatedly in the later computation.

\begin{lem}\label{Liegene}
Let $K$ be a field with $\cha K\neq 2$ and $I_2\in\GL_2(K)$ be the identity matrix.
    \begin{enumerate}[label=(\roman*)]
        \item For $A,B,C,D\in\mat_{2\times 2}(K)$, the Lie subalgebra of $\gl_4(K)$ generated by the block matrices $\br{\begin{smallmatrix}
            I_2 & 0 \\
            0 & 0
        \end{smallmatrix}}$, $\br{\begin{smallmatrix}
            A & B\\
            C & D
        \end{smallmatrix}}$
        contains
        $$\br{\begin{smallmatrix}
            A & 0\\
            0 & D
        \end{smallmatrix}},~\br{\begin{smallmatrix}
            0 & B\\
            0 & 0
        \end{smallmatrix}}\qaq\br{\begin{smallmatrix}
            0 & 0\\
            C & 0
        \end{smallmatrix}}.$$
        \item For $\alpha_{i,j},\beta_{i,j}\in K$, the Lie subalgebra of $\gl_4(K)$ generated by
    $$\br{\begin{smallmatrix}
            1 & 0 & 0 & 0\\
            0 & 0 & 0 & 0\\
            0 & 0 & 1 & 0\\
            0 & 0 & 0 & 0
        \end{smallmatrix}}\qaq\br{\begin{smallmatrix}
            \alpha_{1,1} & \alpha_{1,2} & 0 & 0\\
            \alpha_{2,1} & \alpha_{2,2} & 0 & 0\\
            0 & 0 & \beta_{1,1} & \beta_{1,2}\\
            0 & 0 & \beta_{2,1} & \beta_{2,2}
        \end{smallmatrix}}$$
    contains
    $$\br{\begin{smallmatrix}
            \alpha_{1,1} & 0 & 0 & 0\\
            0 & \alpha_{2,2} & 0 & 0\\
            0 & 0 & \beta_{1,1} & 0\\
            0 & 0 & 0 & \beta_{2,2}
        \end{smallmatrix}},
        ~\br{\begin{smallmatrix}
            0 & \alpha_{1,2} & 0 & 0\\
            0 & 0 & 0 & 0\\
            0 & 0 & 0 & \beta_{1,2}\\
            0 & 0 & 0 & 0
        \end{smallmatrix}}\qaq
        \br{\begin{smallmatrix}
            0 & 0 & 0 & 0\\
            \alpha_{2,1} & 0 & 0 & 0\\
            0 & 0 & 0 & 0\\
            0 & 0 & \beta_{2,1} & 0
        \end{smallmatrix}}.$$
    \end{enumerate}
\end{lem}
\begin{proof}
    By the definition of $\gl_4(K)$, we have
    $$\sq{\br{\begin{smallmatrix}
            I_2 & 0 \\
            0 & 0
        \end{smallmatrix}},\br{\begin{smallmatrix}
            A & B\\
            C & D
        \end{smallmatrix}}}=\br{\begin{smallmatrix}
            0 & B\\
            -C & 0
        \end{smallmatrix}},~\sq{\br{\begin{smallmatrix}
            I_2 & 0 \\
            0 & 0
        \end{smallmatrix}},\br{\begin{smallmatrix}
            0 & B\\
            -C & 0
        \end{smallmatrix}}}=\br{\begin{smallmatrix}
            0 & B\\
            C & 0
        \end{smallmatrix}}.$$
        It's clear that the $K$-linear subspace of $\gl_4(K)$ spanned by these matrices and $\br{\begin{smallmatrix}
            A & B\\
            C & D
        \end{smallmatrix}}$ contains 
        $$\br{\begin{smallmatrix}
            A & 0\\
            0 & D
        \end{smallmatrix}},~\br{\begin{smallmatrix}
            0 & B\\
            0 & 0
        \end{smallmatrix}}\qaq\br{\begin{smallmatrix}
            0 & 0\\
            C & 0
        \end{smallmatrix}}.$$
    The assertion (i) follows. The assertion (ii) is a direct corollary of (i).
\end{proof}

\section{Filtered $\varphi$-modules arising from supersingular abelian surfaces over $\qp$}\label{classphimodule}
For any abelian variety $A$ over $\qp$ with good reduction, the filtered $\varphi$-module $D_{\cris}(V_p(A))\in\mf_{\qp}^{\f}$ associated with $A$ carries rich information. The aim of this section is to classify for $p\ge 7$ all the filtered $\varphi$-modules associated with supersingular abelian surfaces over $\qp$, which form a full subcategory $\mfs$ of $\mf_{\qp}^{\f}$.
\subsection{Supersingular $p$-Weil polynomials of degree $4$}
For $D=(D,\varphi,\fil_\bullet)\in\mfs$, the crystalline Frobenius $\varphi$ acts linearly on $D$, with characteristic polynomial $\chi_{\varphi}\in\mathscr P^{\sss}_{\w,p}$. The next proposition proves the rigidity of the forms of supersingular $p$-Weil polynomials of degree $4$. 
\begin{prop}\label{classssweil}
    Let $p\ge 7$ be a prime number, then all of the supersingular $p$-Weil polynomials of degree $4$ are
    \begin{equation}\label{sspolylist}
        (X^2\pm p)^2\qaq X^4+\epsilon pX^2+p^2\whr\epsilon\in\st{0,\pm1}.
    \end{equation}
\end{prop}
\begin{proof}
    It suffices to classify supersingular $p$-Weil numbers of degree\footnote{The degree of an algebraic number $\alpha\in\bar{\mq}$ is the number $[\mq(\alpha)\!:\!\mq]\in\mz$.} smaller than $4$. Let $p\ge 7$ be a prime number, $m$ a positive integer and $\zeta_m$ a fixed primitive $m$-th root of unity. Suppose that $[\mq(\zeta_m\sqrt{p}):\mq]\le 4$. We have immediately
    $$[\mq(\zeta_m):\mq]=\frac{[\mq(\zeta_m,\sqrt{p}):\mq]}{[\mq(\zeta_m,\sqrt{p}):\mq(\zeta_m)]}=\frac{[\mq(\zeta_m,\sqrt{p}):\mq(\zeta_m\sqrt{p})][\mq(\zeta_m\sqrt{p}):\mq]}{[\mq(\zeta_m,\sqrt{p}):\mq(\zeta_m)]}.$$
    Since $\mq(\zeta_m,\sqrt{p})=\mq(\zeta_m\sqrt{p},\sqrt{p})$, the numerator of the right hand side is bounded by $8$. Note that the left hand side is equal to the Euler function $\varphi(m)$. 
    
    Suppose that $\varphi(m)=8$ and hence $m\in\st{15,16,20,24,30}$. In this case, any $p\ge 7$ is unramified in $\mq(\zeta_m)$. Therefore, the denominator of the right hand side $[\mq(\zeta_m,\sqrt{p}):\mq(\zeta_m)]$ equals $2$ and the left hand side is less than or equal to $4$, contradiction.

    It follows that $\varphi(m)\le 4$ and hence $m\in\st{1,2,3,4,5,6,8,10,12}$. Note that $\zeta_5\sqrt{p}$ and $\zeta_{10}\sqrt{p}$ both have degree $8$. For each of the rest $m$, we fix a $m$-th root of unity $\zeta_m$. It follows that the conjugacy classes of these $\zeta_m\sqrt{p}$ are
    $$\st{\pm\sqrt{p}},~\st{\pm\zeta_4\sqrt{p}},~\st{\pm\zeta_3\sqrt{p},\pm\bar{\zeta}_3\sqrt{p}},~\st{\pm\zeta_8\sqrt{p},\pm\bar{\zeta}_8\sqrt{p}},~\st{\pm\zeta_{12}\sqrt{p},\pm\bar{\zeta}_{12}\sqrt{p}},$$
    which correspond to the polynomials in (\ref{sspolylist}).
\end{proof}

\subsection{Linear-algebraic constraints}
The classification of filtered $\varphi$-modules arising from supersingular abelian surfaces over $\qp$ is reduced to the analysis of the linear-algebraic conditions, as stated in the following lemma, which is a direct corollary of Theorem \ref{dfromav}.
\begin{lem}\label{ssfilphimod}
    An admissible filtered $\varphi$-module $D\in\mf_{\qp}^{\f}$ belongs to $\mfs$ if and only if $D$ satisfies the following conditions.
    \begin{enumerate}
        \item[$\mathbf{S1}$] $\dim D=4$ and the Hodge numbers satisfy $h^{0}(D)=h^{1}(D)=2$,
        \item[$\mathbf{S2}$] the $\varphi$-action on $D$ is semisimple and $\chi_{\varphi}$ is a supersingular $p$-Weil polynomial,
        \item[$\mathbf{S3}$] there exists a nondegenerate skew form on $D$ under which $\varphi$ is a $p$-similitude and $\fil_1D$ is totally isotropic.
    \end{enumerate}
\end{lem}
\begin{lem}\label{adm}
    Let $D$ be any filtered $\varphi$-module over $\qp$ (not necessarily admissible) satisfying \hyperref[ssfilphimod]{$\mathbf{S1}$} and \hyperref[ssfilphimod]{$\mathbf{S2}$}. Then, $D$ is admissible if and only if $\fil_1D$ is not stable under $\varphi$.
\end{lem}
\begin{proof}
    Let $D'\subset D$ be a proper sub-$\varphi$-module of $D$. We have $\dim D'=2$ since it corresponds to a certain $\qp$-factor of a supersingular $p$-Weil polynomial. Each eigenvalue of $\varphi$ in $\bqp$ has valuation $\frac{1}{2}$, so we have $t_N(D')=1$. On the other hand,
    $$t_H(D')=\sum_{i\in\mz}i\dim_{\qp}(D'\cap\fil_iD/D'\cap\fil
    _{i+1}D)\xlongequal{\mathbf{S1}}\dim_{\qp}(D'\cap\fil_1D).$$
    Therefore, $t_H(D')\le t_N(D')$ if and only if $\fil_1D\neq D'$. $D$ is admissible if and only if this holds for any such $D'$, that is, $\fil_1D$ itself is not $\varphi$-stable.
\end{proof}

The Hodge decomposition of $D\in\mfs$ is rigid. Theorem \ref{hodgedecomp} implies that there exists a decomposition of $\qp$-linear subspaces $D=D_1\oplus D_0$ and a $u\in\GL(D)$ such that
\begin{enumerate}[label=(\alph*)]
    \item $D_1=\fil_1D$,
    \item $(u-\id_D)D_0=0$ and $(u-\id_D)D_1\subset D_0$,
    \item $D=\bigoplus_{\xi\in X}D_{\xi}$ where $D_\xi=\st{x\in D\ver(\F^{\el})^j(x)\in D_{\xi(j)}}$ with $\F^{\el}=u^{-1}\circ\varphi\circ\mu(p^{-1})$.
\end{enumerate}
Moreover, the data $D_0$ and $u$ are unique if they satisfy the condition (d) in Theorem \ref{hodgedecomp}. We focus on the decomposition $D=\bigoplus_{\xi\in X_D}D_{\xi}$ first.

\begin{nota}
    For $r\in\mz_{\ge 1}$ and $a_0,\cdots,a_{r-1}\in\mz$, the symbol $[a_0a_1\cdots a_{r-1}]$ stands for a function $\xi\in X_{\p}$ of period dividing $r$ such that $\xi(k)=a_k$ for $0\le k\le r-1$.
\end{nota}

\begin{lem}\label{listxd}
    Suppose that $D\in\mf_{\qp}^{\f}$ satisfies the equivalent \hyperref[ssfilphimod]{$\mathbf{S1}$} and \hyperref[ssfilphimod]{$\mathbf{S2}$}. The Wintenberger type $X_D:=\{\xi\in X_\p~\!|~\!D_\xi\neq 0\}$ of $D$ is one of the following.
    \begin{enumerate}[label=(\Alph*).]
        \item $\left\{[0011],~\![0110],~\![1100],~\![1001]\right\}$,
        \item $\st{[01],~\![10]}$,
        \item $\st{[0],~\![1]}$,
        \item $\st{[001],[010],[100],[1]}$.
    \end{enumerate}
\end{lem}
\begin{proof}
    For $\xi\in X_\p$ and $k\in\mz$, denote by $P(\xi)$ the period of $\xi$ and $\xi^{k}(j)=\xi(j+k)$. It follows that $\dim D_{\xi}=\dim D_{\xi^k}$ for any $k\in\mz$. Moreover, by definition for $i=0,1$ we have $D_i=\bigoplus_{\xi\in X_D\atop\xi(0)=i}D_\xi$ and
    $$\sum_{\xi\in X_D\atop\xi(0)=i}\dim D_{\xi}=\dim D_i=2.$$
    Note that $\dim D_{[0]}\neq 1$. Otherwise for any $0\neq x\in D_{[0]}\subset D_0$, we have $\F^{\el}(x)$ proportional to $x$, and
    $$\varphi(x)=u(\F^{\el}(\mu(p)(x)))=u(\underbrace{\F^{\el}(px)}_{\in D_0})=p\F^{\el}(x).$$
    It follows that $x$ is an eigenvector of $\varphi$, contradictory to the fact that $\varphi$ has no root over $\qp$.
    We argue by cases, depending on $P_{\max}(D):=\max_{\xi\in X_D}P(\xi)\le 4$. 
    \begin{enumerate}
        \item[$\bullet$] If $P_{\max}(D)=4$, clearly $X_D=(A)$. 
        \item[$\bullet$] If $P_{\max}(D)=3$, there exists $\xi\in X_D$ such that $P(\xi)=\dim D_\xi=1$. It follows that $\xi=[1]$ and $X_D=(D)$. 
    
        \item[$\bullet$] If $P_{\max}(D)=2$, note that the only period functions with value in $\st{0,1}$ are $\xi=[01]$ and $\xi^1=[10]$. If $\dim D_\xi=\dim D_{\xi^1}=1$, then $[0]\in X_D$ and $D_{[0]}=1$ for dimensional reason, contradiction. Therefore, $\dim D_{\xi}=\dim D_{\xi^1}=2$ and $X_D=(B)$. 
        \item[$\bullet$] If $P_{\max}(D)=1$, it's clear that $X_D=(C)$.
    \end{enumerate}
\end{proof}

\subsection{The case $\chi_{\varphi}(X)=(X^2\pm p)^2$}
\begin{defn}\label{prodcase}
For $\epsilon'\in\st{\pm1}$, we denote by $D^{prod}_{\epsilon'}\in\mf_{\qp}^{\f}$ the filtered $\varphi$-modules satisfying \hyperref[ssfilphimod]{$\mathbf{S1}$} determined by the following data
\begin{enumerate}
    \item[$(prod)$] $D^{prod}_{\epsilon'}=(\qp^4,\varphi^{prod}_{\epsilon'},\fil_\bullet)$, $\mat_{\mathcal B}\br{\varphi^{prod}_{\epsilon'}}=\br{\begin{smallmatrix}
            0 & 0 & \epsilon'p & 0\\
            0 & 0 & 0 & \epsilon'p\\
            1 & 0 & 0 & 0\\
            0 & 1 & 0 & 0
        \end{smallmatrix}}$, $\fil_1D^{prod}_{\epsilon'}=\Span(e_1,e_2)$.
\end{enumerate}
\end{defn}
\begin{rem}
    The filtered $\varphi$-modules $D_{\epsilon'}^{prod}$ are products of two isomorphic subobjects of dimension $2$.
\end{rem} 
\begin{thm}\label{mainthm0}
    For $\epsilon'\in\st{\pm1}$, $D^{prod}_{\epsilon'}$ represents the unique isomorphism class of admissible filtered $\varphi$-modules over $\qp$ satisfying \hyperref[ssfilphimod]{$\mathbf{S1}$}, \hyperref[ssfilphimod]{$\mathbf{S2}$} with $\chi_{\varphi}(X)=(X^2-\epsilon' p)^2$. In this case, we have $X_{D^{prod}_{\epsilon'}}=(B)$ in Lemma \ref{listxd}.
\end{thm}
\begin{proof}
    Let $D$ be an admissible filtered $\varphi$-modules over $\qp$ satisfying \hyperref[ssfilphimod]{$\mathbf{S1}$}, \hyperref[ssfilphimod]{$\mathbf{S2}$} with $\chi_{\varphi}(X)=(X^2-\epsilon' p)^2$. We claim that $\fil_1D\cap\varphi(\fil_1D)=0$. By contradiction, suppose that for some nonzero $x\in\fil_1D$ we have $\varphi(x)\in\fil_1D$. Then $\varphi(x)$ is not proportional to $x$, since otherwise $x$ is an eigenvector of $\varphi$ on $D$, contradictory to the fact that $\chi_{\varphi}(X)=(X^2-\epsilon' p)^2$ has no root over $\qp$. Therefore, $\fil_1D$ is spanned by $x$ and $\varphi(x)$. Using that $\varphi^2(x)=\mp px$, we have immediately $\tilde D:=\fil_1D$ is a $\qp[\varphi]$-submodule of $D$. On the other hand, we have $t_N(\tilde D)=1$ by the supersingularity of $\chi_{\varphi}(X)$. Therefore, we have
    $$t_H(\tilde D)=\sum_{i\in\mz}i\dim_{\qp}(\fil_iD\cap\tilde D)= 2>1=t_N(\tilde D),$$
    contradictory to the admissibility of $D$. Let $(x,y)$ be a basis of $\fil_1D$. Since $\fil_1D\cap\varphi(\fil_1D)=0$, the 
$\qp$-linear subspaces
    $$D':=\qp x\oplus\qp\varphi(x)\qaq D'':=\qp y\oplus\qp\varphi(y)$$
    define two isomorphic subobjects of $D\in\mf_{\qp}^{\f}$ and $D=D'\oplus D''$. Since the Hodge decomposition is given functorially, it suffices to study the filtered $\varphi$-module $D'$, which lives in the unique isomorphism class of filtered $\varphi$-modules over $\qp$ with $\chi_{\varphi}(X)=X^2-\epsilon'p$. Let $D'=D'_1\oplus D'_0$ be the Hodge decomposition of $D'$ and $u'$ be the automorphism as in Theorem \ref{hodgedecomp}. As in the proof of Lemma \ref{listxd}, we have $\dim D'_{[0]}\neq 1$. It follows that $X_{D'}=(B)$, $D'_1=D'_{[10]}$ and $D'_0=D'_{[01]}$. Therefore, we have $X_D=(B)$ as well and $D$ is isomorphic to the direct sum of $2$ copies of $D'$, and one can check that it is exactly $D^{prod}_{\epsilon'}$.
\end{proof}
\subsection{The case $\chi_{\varphi}(X)=X^4+\epsilon pX^2+p^2$}
Fix a parameter $\epsilon\in\st{0,\pm1}$. In contrast to the case $\chi_{\varphi}(X)=(X^2\pm p)^2$, there are infinitely many isomorphism classes of admissible filtered $\varphi$-modules with $\chi_{\varphi}(X)=X^4+\epsilon pX^2+p^2$. We can construct concretely families of such filtered $\varphi$-modules.
\begin{defn}\label{canfam}
For $\epsilon'\in\st{\pm1}$ and $a,b,a'\in\qp$ such that $ab\neq-1$, we denote by
$$D^{\epsilon,iso}_{\epsilon'},~D_{a'}^{\epsilon,\nu},~ D_{(a,b)}^{\epsilon,\mu}\in\mf_{\qp}^{\f}$$
the filtered $\varphi$-modules satisfying \hyperref[ssfilphimod]{$\mathbf{S1}$} determined by the following data
\begin{enumerate}
        \item[$(iso)$] $D^{\epsilon,iso}_{\epsilon'}=(\qp^4,\varphi^{\epsilon,iso}_{\epsilon'},\fil_\bullet)$, $\mat_{\mathcal B}\br{\varphi^{\epsilon,iso}_{\epsilon'}}=\br{\begin{smallmatrix}
            0 & 0 & \epsilon'p & 0\\
            0 & 0 & 0 & \epsilon'p\\
            1 & 0 & 0 & -(\epsilon+2\epsilon')p\\
            0 & 1 & 1 & 0
        \end{smallmatrix}}$, $\fil_1D^{\epsilon,iso}_{\epsilon'}=\Span(e_1,e_2)$,
        \item[$(\nu)$] $D^{\epsilon,\nu}_{a'}=(\qp^4,\varphi^{\epsilon,\nu}_{a'},\fil_\bullet)$, $\mat_{\mathcal B}\br{\varphi^{\epsilon,\nu}_{a'}}=\br{\begin{smallmatrix}
            0 & 0 & 0 & 1\\
            -p^2 & 0 & 0 & 0\\
            0 & 1 & 0 & -a'\\
            a'-\epsilon p & 0 & 1 & 0
        \end{smallmatrix}}$, $\fil_1D^{\epsilon,\nu}_{a'}=\Span(e_1,e_2)$,
        \item[$(\mu)$] $D^{\epsilon,\mu}_{(a,b)}=(\qp^4,\varphi^{\epsilon,\mu}_{(a,b)},\fil_\bullet)$, $\mat_{\mathcal B}\br{\varphi^{\epsilon,\mu}_{(a,b)}}=\br{\begin{smallmatrix}
            0 & 0 & 0 & -p^2\\
            0 & 0 & 1 & -\frac{a^2+\epsilon p+b^2p^2}{ab+1}\\
            1 & 0 & a & -\frac{a^3+\epsilon ap-bp^2}{ab+1}\\
            0 & 1 & b & -a
        \end{smallmatrix}}$, $\fil_1D^{\epsilon,\mu}_{(a,b)}=\Span(e_1,e_2)$,
\end{enumerate}
where $\mathcal B=(e_i)_{1\le i\le 4}$ is the canonical basis of $\qp^4$.
\end{defn}
\begin{rem}
    The filtered $\varphi$-modules $D_{\epsilon'}^{\epsilon,iso}$ should be thought of as "isolated" objects with respect to the natural families under consideration (see Proposition \ref{munu}).
\end{rem}

\begin{thm}\label{mainclass}
    Let $p\ge 7$ and $\epsilon\in\st{0,\pm1}$. Let $D=(D,\varphi,\fil_\bullet)\in\mf_{\qp}^{\f}$ be an admissible filtered $\varphi$-module with $\chi_{\varphi}(X)=X^4+\epsilon pX^2+p^2$, the following are equivalent.
    \begin{enumerate}
        \item[$\bullet$] $D\in\mfs$; 
        \item[$\bullet$] $D$ satisfies \hyperref[ssfilphimod]{$\mathbf{S1}$};
        \item[$\bullet$] $D$ is isomorphic to $D^{\epsilon,iso}_{\epsilon'}$, $D_{a'}^{\epsilon,\nu}$ or $D_{(a,b)}^{\epsilon,\mu}$ for some parameters $\epsilon',a,b,a'$ satisfy the following conditions:
        \begin{center}
            \begin{tabular}{|c|c|}
            \hline
                $\epsilon'$ & $\epsilon'\in\st{\pm 1}$\\
                \hline
                $a'$ & $a'\in\epsilon p+p^2\zp$ \\
                \hline
                $a, b$ & $v_p(a)\ge 1,v_p(b)\ge 0\mbox{ or}\atop v_p(a)=v_p(b)+1$.\\
            \hline
            \end{tabular}
        \end{center}
    \end{enumerate}
    Moreover, for each $D$ above, the Hodge decomposition is $D=D_1\oplus D_0$ where $D_1=\Span(e_1,e_2)$ and $D_0=\Span(e_3,e_4)$
\end{thm}

\begin{rem}
The families in Definition \ref{canfam} are always objects in $\mfs$. The additional conditions in Theorem \ref{mainclass} have two roles:
\begin{enumerate}
\item they are precisely the conditions under which the canonical basis used in Definition \ref{canfam} is adapted for the corresponding object, and
\item the objects satisfying these conditions still cover all isomorphism classes in $\mfs$.
\end{enumerate}
\end{rem}
\subsection{Proof of Theorem \ref{mainclass}}
The rest of the section is devoted to proving Theorem \ref{mainclass}. It suffices to classify admissible filtered $\varphi$-modules over $\qp$ satisfying \hyperref[ssfilphimod]{$\mathbf{S1}$} with $\chi_{\varphi}(X)=X^4+\epsilon pX^2+p^2$. For simplicity, we adopt the following definition. 
\begin{defn}\label{adaptedbasis}
    For $D\in\mf_{\qp}^{\f}$ satisfying \hyperref[ssfilphimod]{$\mathbf{S1}$}, a $\qp$-basis $(e_1,e_2,e_3,e_4)$ of $D$ is called \textit{adapted} if $D_1=\Span(e_1,e_2)$ and $D_0=\Span(e_3,e_4)$, where $D=D_0\oplus D_1$ is the Hodge decomposition.
\end{defn}
Using this terminology, for the filtered $\varphi$-module structures defined on $\qp^4$ listed in Theorem \ref{mainclass}, we will see that the canonical basis of $\qp^4$ is always adapted.
\begin{lem}\label{xdneqc}
    Let $D$ be an admissible filtered $\varphi$-modules over $\qp$ satisfying Condition \hyperref[ssfilphimod]{$\mathbf{S1}$} with $\chi_{\varphi}(X)=X^4+\epsilon pX^2+p^2$ for $\epsilon\in\st{0,\pm1}$, then $X_D\neq(C)$ or $(D)$ in Lemma \ref{listxd}.
\end{lem}
\begin{proof}
    By contradiction, suppose $X_D=(C)$. It follows that $D_0=D_{[0]}$ so $u\varphi(D_0)=D_0$.
    Take an adapted basis $\mb$ of $D$. Under this basis, the matrix of $u$, $\F^{\el}$ and $\mu(p)$ have the following forms, where $I_2=\br{\begin{smallmatrix}
        1 & 0\\
        0 & 1
    \end{smallmatrix}}$, $S,M_1,M_2\in\mat_{2\times 2}(\qp)$.
    $$\mat_{\mb}(u)=\br{\begin{smallmatrix}
        I_2 & 0\\
        S & I_2
    \end{smallmatrix}},~\mat_{\mb}(\F^{\el})=\br{\begin{smallmatrix}
        M_1 & 0\\
        0 & M_0
    \end{smallmatrix}}\qaq\mat_{\mb}(\mu(p))=\br{\begin{smallmatrix}
        pI_2 & 0\\
        0 & I_2
    \end{smallmatrix}}.$$
    In this case, we have $\chi_{\F^{\el}}=\chi_{\varphi}$ and hence $\F^{\el}$ has no eigenvector in $D$. It follows that for any nonzero $x\in D_1$, if $\varphi(x)=y+z$ where $y\in D_1$ and $z\in D_0$, then $y$ is not proportional to $x$. For the same reason, $\varphi(z)\in D_0$ is not proportional to $z$. So $\mathcal B:=(x,y,z,\varphi(z))$ is an adapted basis of $D$, under which the matrix of $\varphi$ has the form
    $$\mat_{\mathcal B}(\varphi)=\br{\begin{smallmatrix}
            0 & a & 0 & 0\\
            1 & b & 0 & 0\\
            1 & c & 0 & e\\
            0 & d & 1 & f
        \end{smallmatrix}}$$
        with $a,b,c,d,e,f\in\qp$. A direct calculation shows
        $$\chi_{\varphi}(X)=X^4-(b+f)X^3+(bf-a-e)X^2+(af+be)X+ae=X^4+\epsilon pX^2+p^2.$$
        Comparing the coefficients, we have
        $$b+f=0,~bf-a-e=\epsilon p,~af+be=0\qaq ae=p^2.$$
        If $b\neq0$, we have $a=e=\pm p$ and $b^2=-(a+e+\epsilon p)\in\st{\pm p,\pm 2p,\pm 3p}$, which is impossible for $p\ge 5$. Therefore, we have $b=f=0$ and $a+e=-\epsilon p$. In other words, $\st{a,e}$ is the set of roots of $X^2+\epsilon pX+p^2$. In particular, we have $v_p(a)=1$. Let $M$ be an adapted lattice of $D$. By the definition of adapted lattices and Theorem \ref{hodgedecomp} (d), we have $\pr_1\circ\frac{\varphi}{p}(M\cap D_1)\subset D_1$. However, $\det(\pr_1\circ\frac{\varphi}{p}|_{D_1})=-\frac{a}{p^2}\notin\zp$, contradiction. $X_D\neq(C)$.

        Suppose that $X_D=(D)$. Take $x,y\in D$ such that $D_{[1]}=\qp x$ and $D_{[100]}=\qp y$. One can check that $\mb=(x,y,\F^{\el}(y),(\F^{\el})^2(y))$ is an adapted basis of $D$. Under this basis, the matrix of $u$, $\F^{\el}$ and $\mu(p)$ have the following forms, where $t\in\qp^\times$ and $a,b,c,d\in\qp$.
        $$\mat_{\mb}(u)=\br{\begin{smallmatrix}
            1 & 0 & 0 & 0\\
            0 & 1 & 0 & 0\\
            a & b & 1 & 0\\
            c & d & 0 & 1
        \end{smallmatrix}},~\mat_{\mb}(\F^{\el})=\br{\begin{smallmatrix}
            t & 0 & 0 & 0\\
            0 & 0 & 0 & \frac{1}{t}\\
            0 & 1 & 0 & 0\\
            0 & 0 & 1 & 0
        \end{smallmatrix}}\qaq\mat_{\mb}(\mu(p))=\br{\begin{smallmatrix}
            p & 0 & 0 & 0\\
            0 & p & 0 & 0\\
            0 & 0 & 1 & 0\\
            0 & 0 & 0 & 1
        \end{smallmatrix}}.$$
        We deduce from $\varphi=u\circ \F^{\el}\circ\mu(p)$ that
        $$\mat_{\mb}(\varphi)=\br{\begin{smallmatrix}
            pt & 0 & 0 & 0\\
            0 & 0 & 0 & \frac{1}{t}\\
            apt & p & 0 & \frac{b}{t}\\
            cpt & 0 & 1 & \frac{d}{t}
        \end{smallmatrix}}$$
        and by taking the characteristic polynomial, we have
        $$\chi_{\varphi}(X)=X^4-\Big(\frac{d}{t}+pt\Big)X^3+\Big(pd-\frac{b}{t}\Big)X^2+\Big(pb-\frac{p}{t}\Big)X+p^2=X^4+\epsilon pX^2+p^2.$$
        Comparing the coefficients, we have $d=-pt^2$, $b=\frac{1}{t}$ and $pd-\frac{b}{t}=\epsilon p$. It follows that $p^2t^2+\frac{1}{t^2}+\epsilon p=0$ and hence $pt^2$ is a root of $X^2+\epsilon X+1$. Note that for $\epsilon\in\st{0,\pm 1}$, roots of $X^2+\epsilon X+1$ are roots of unity, whereas $2\nmid v_p(pt^2)$, contradiction.
\end{proof}
The following lemma is rather technical and will be used in the subsequent proofs.
\begin{lem}\label{rathertechnical}
    Let $a,b\in\qp$ with $ab\neq-1$. Then, the following are equivalent.
    \begin{enumerate}
        \item $c=-\frac{a^2+\epsilon p+b^2p^2}{ab+1}\in p\zp$.
        \item $v_p(a)\ge 1$ and $v_p(b)\ge 0$, or $v_p(a)=v_p(b)+1$.
    \end{enumerate}
    In both cases, there exists a $\zp$-lattice $N$ in $\qp^2$ satisfying the following conditions.
        \begin{enumerate}[label=(\roman*)]
            \item $\br{\begin{smallmatrix}
            0 & -1\\
            1 & \frac{c}{p}
        \end{smallmatrix}}N=N$,
            \item $\br{\begin{smallmatrix}
            -pb & a\\
            a+bc & pb
        \end{smallmatrix}}N\subset N$.
        \end{enumerate} 
\end{lem}
\begin{proof}
We divided $\mz^2$ into four regions $D_1,D_2,D_3$ and $L$ as following:
$$D_1=\st{x\le y,~x\le 0},~D_2=\st{x\ge y+2,~y\le -1},~D_3=\st{x\ge 1,~y\ge 0},~L=\st{x-y=1,~x\le 0}.$$
This is shown in the following figure.

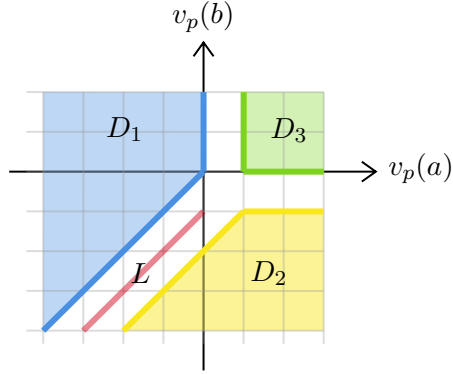
\centerline{\tikzset{every picture/.style={line width=0.75pt}} %set default line width to 0.75pt        

\begin{tikzpicture}[x=0.75pt,y=0.75pt,yscale=-1,xscale=1]
%uncomment if require: \path (0,310); %set diagram left start at 0, and has height of 310

%Shape: Axis 2D [id:dp7509711527982947] 
\draw  (260.11,117.39) -- (443.11,117.39)(357.11,52.39) -- (357.11,217.39) (436.11,112.39) -- (443.11,117.39) -- (436.11,122.39) (352.11,59.39) -- (357.11,52.39) -- (362.11,59.39)  ;
%Shape: Rectangle [id:dp08206958545127985] 
\draw  [color={rgb, 255:red, 0; green, 0; blue, 0 }  ,draw opacity=0 ][fill={rgb, 255:red, 184; green, 233; blue, 134 }  ,fill opacity=0.62 ] (377.12,77.4) -- (417.11,77.4) -- (417.11,117.4) -- (377.12,117.4) -- cycle ;
%Shape: Polygon [id:ds7507780999447653] 
\draw  [color={rgb, 255:red, 0; green, 0; blue, 0 }  ,draw opacity=0 ][fill={rgb, 255:red, 74; green, 144; blue, 226 }  ,fill opacity=0.36 ] (357.12,77.39) -- (357.11,117.39) -- (311.06,163.4) -- (277.08,197.36) -- (277.1,137.36) -- (277.12,77.36) -- cycle ;
%Shape: Polygon [id:ds20810705237694915] 
\draw  [color={rgb, 255:red, 0; green, 0; blue, 0 }  ,draw opacity=0 ][fill={rgb, 255:red, 248; green, 231; blue, 28 }  ,fill opacity=0.47 ] (377.1,137.4) -- (417.1,137.41) -- (417.08,197.41) -- (377.08,197.4) -- (317.08,197.38) -- cycle ;
%Shape: Grid [id:dp35074601466734157] 
\draw  [draw opacity=0] (417.12,77.41) -- (417.08,203.97) -- (268.64,203.92) -- (268.69,77.36) -- cycle ; \draw  [color={rgb, 255:red, 155; green, 155; blue, 155 }  ,draw opacity=0.35 ] (417.12,77.41) -- (268.69,77.36)(417.11,97.41) -- (268.68,97.36)(417.11,117.41) -- (268.67,117.36)(417.1,137.41) -- (268.67,137.36)(417.09,157.41) -- (268.66,157.36)(417.09,177.41) -- (268.65,177.36)(417.08,197.41) -- (268.65,197.36) ; \draw  [color={rgb, 255:red, 155; green, 155; blue, 155 }  ,draw opacity=0.35 ] (417.12,77.41) -- (417.08,203.97)(397.12,77.4) -- (397.08,203.96)(377.12,77.4) -- (377.08,203.96)(357.12,77.39) -- (357.08,203.95)(337.12,77.38) -- (337.08,203.94)(317.12,77.38) -- (317.08,203.94)(297.12,77.37) -- (297.08,203.93)(277.12,77.36) -- (277.08,203.92) ; \draw  [color={rgb, 255:red, 155; green, 155; blue, 155 }  ,draw opacity=0.35 ]  ;
%Straight Lines [id:da25753191568622735] 
\draw [color={rgb, 255:red, 74; green, 144; blue, 226 }  ,draw opacity=1 ][line width=2.25]    (357.11,117.39) -- (277.08,197.36) ;
%Straight Lines [id:da17498096059986823] 
\draw [color={rgb, 255:red, 248; green, 231; blue, 28 }  ,draw opacity=1 ][line width=2.25]    (317.08,197.38) -- (377.1,137.4) ;
%Straight Lines [id:da5569304206683098] 
\draw [color={rgb, 255:red, 248; green, 231; blue, 28 }  ,draw opacity=1 ][line width=2.25]    (377.1,137.4) -- (417.1,137.41) ;
%Straight Lines [id:da4292793225004158] 
\draw [color={rgb, 255:red, 126; green, 211; blue, 33 }  ,draw opacity=1 ][line width=2.25]    (377.12,77.4) -- (377.11,102) -- (377.11,117.4) ;
%Straight Lines [id:da33360410233041426] 
\draw [color={rgb, 255:red, 208; green, 2; blue, 27 }  ,draw opacity=0.48 ][line width=2.25]    (357.1,137.39) -- (297.08,197.37) ;
%Straight Lines [id:da9451715028559298] 
\draw [color={rgb, 255:red, 126; green, 211; blue, 33 }  ,draw opacity=1 ][line width=2.25]    (377.11,117.4) -- (417.11,117.4) ;
%Straight Lines [id:da9291966315461341] 
\draw [color={rgb, 255:red, 74; green, 144; blue, 226 }  ,draw opacity=1 ][line width=2.25]    (357.11,117.39) -- (357.12,77.39) ;

% Text Node
\draw (307,88.4) node [anchor=north west][inner sep=0.75pt]    {$D_{1}$};
% Text Node
\draw (379.09,160.8) node [anchor=north west][inner sep=0.75pt]    {$D_{2}$};
% Text Node
\draw (389.09,88.8) node [anchor=north west][inner sep=0.75pt]    {$D_{3}$};
% Text Node
\draw (319.09,161.78) node [anchor=north west][inner sep=0.75pt]    {$L$};
% Text Node
\draw (448,106.4) node [anchor=north west][inner sep=0.75pt]    {$v_{p}( a)$};
% Text Node
\draw (340,30.4) node [anchor=north west][inner sep=0.75pt]    {$v_{p}( b)$};

\end{tikzpicture}
}
    
Then, (2) is equivalent to $(v_p(a),v_p(b))\in L\sqcup D_3$. We verify that (1) implies (2) by contradiction. If $(v_p(a),v_p(b))\in D_1$, we have $v_p(c)\le v_p(a^2)-\min\st{v_p(ab),1}\le 0$. Similarly, if $(v_p(a),v_p(b))\in D_2$, we have $v_p(c)\le v_p(b^2p^2)-\min\st{v_p(ab),1}\le 0$. 

Conversely, if $(v_p(a),v_p(b))\in L$, we have $v_p(c)\ge v_p(a^2)-v_p(ab)=1$. If $(v_p(a),v_p(b))\in D_3$, we have $v_p(c)=\min\st{v_p(a^2),v_p(b^2p^2),1}\ge1$.

It remains to prove the existence of $N$. Suppose that $(v_p(a),v_p(b))\in L$. Let $(e_1,e_2)$ be the canonical basis of $\qp^2$. Take $u=\frac{bp}{a}\in\zp^\times$ and $n=-v_p(a)\ge 1$. Consider $N=\zp e_1\oplus\zp\frac{e_1+ue_2}{p^n}$, we claim that conditions (i), (ii) hold. In fact, under the basis $(e_1,\frac{1}{p^n}(e_1+ue_2))$, conditions (i) and (ii) becomes
$$\br{\begin{smallmatrix}
    -u^{-1} & -\frac{1}{p^n}(u^{-1}+u+\frac{c}{p})\\
    \frac{p^n}{u} & u^{-1}(1+\frac{cu}{p})
    \end{smallmatrix}}
    \qaq \br{\begin{smallmatrix}
        -\frac{a+bc}{u}-bp & \frac{au-(a+bc)u^{-1}-2bp}{p^n}\\
        \frac{p^n(a+bc)}{u} & \frac{a+bc}{u}+bp
    \end{smallmatrix}}\in\GL_2(\zp).$$
Note that
$$\frac{c}{p}=-\frac{a^2+\epsilon p+b^2p^2}{p(ab+1)}=-\frac{a^2(1+u^2)+\epsilon p}{a^2u+p}.$$
We deduce that $v_p(u^{-1}+u+\frac{c}{p})=v_p(-\frac{c+\epsilon p}{a^2u})\ge 2n$ and thereby $-\frac{1}{p^n}(u^{-1}+u+\frac{c}{p})\in\zp$. On the other hand, we can verify the equality $a+bc=-\frac{c+\epsilon p+p^2b^2}{a}$. Using this equality repeatedly, we have immediately 
$$v_p\br{\frac{p^n(a+bc)}{u}}\ge n+v_p(p^2b^2)-v_p(a)=0$$
and
$$-\frac{a+bc}{u}-bp=\frac{c+\epsilon p}{au}\in\zp,~\frac{au-(a+bc)u^{-1}-2bp}{p^n}=\frac{c+\epsilon p}{ap^nu}\in\zp.$$
Suppose that $(v_p(a),v_p(b))\in D_3$. Let $(e_1,e_2)$ be the canonical basis of $\qp^2$. We can verify directly that conditions (i), (ii) hold for the standard lattice $N=\zp e_1\oplus\zp e_2$. This completes the proof.
\end{proof}

We now discuss the classification according to the cases $X_D=(A)$ or $(B)$. The following proposition is the main part of Theorem \ref{mainclass}.

\begin{prop}\label{mainthm1}
Let $p\ge 7$ be a prime and $\epsilon\in\st{0,\pm1}$. For $D\in\mf_{\qp}^{\f}$ satisfying \hyperref[ssfilphimod]{$\mathbf{S1}$} with $\chi_{\varphi}(X)=X^4+\epsilon pX^2+p^2$, exactly one of the following holds.
    \begin{enumerate}
        \item $X_D=(A)$, and $D\simeq D_{a}^{\epsilon,\nu}$ where $a\in\epsilon p+p^2\zp$,
        \item $X_D=(B)$, and $D\simeq D_{(a,b)}^{\epsilon,\mu}$ where $a,b\in\qp$ satisfy one of the following:
        \begin{enumerate}[label=(\roman*)]
            \item $v_p(a)\ge 1$ and $v_p(b)\ge 0$,
            \item $v_p(a)=v_p(b)+1$.
        \end{enumerate}
        \item $X_D=(B)$, and $D\simeq D_{\epsilon'}^{\epsilon,iso}$ where $\epsilon'=\pm1$.
    \end{enumerate}
\end{prop}
\begin{proof}
        First, suppose that $X_D=(B)$. We have 
        $$u^{-1}\varphi(D_1)=u^{-1}\varphi(p^{-1}D_1)=u^{-1}\circ\varphi\circ\mu(p^{-1})(D_1)=\F^{\el}(D_1)=D_0.$$
        Therefore, we have $\varphi(D_1)=D_0$ by the unipotence of $u$. If there exists a nonzero $x\in D_1$ such that $y:=u^{-1}\varphi^2(x)$ is not proportional to $x$, we deduce that $\mathcal B=(x,y,\varphi(x),\varphi(y))$ is an adapted basis of $D$. The matrix of $\varphi$ under $\mathcal B$ has the form
        $$\mat_{\mathcal B}(\varphi)=\br{\begin{smallmatrix}
            0 & 0 & 0 & -p^2\\
            0 & 0 & 1 & c\\
            1 & 0 & a & d\\
            0 & 1 & b & e
        \end{smallmatrix}}$$
        with $a,b,c,d,e\in\qp$. A direct calculation shows
        $$\chi_{\varphi}(X)=X^4-(a+e)X^3+(ae-bd-c)X^2+(ac-d+bp^2)X+p^2=X^4+\epsilon pX^2+p^2.$$
        Comparing the coefficients, we have
        $$a+e=0,~ae-bd-c=\epsilon p\qaq ac-d+bp^2=0.$$
        We deduce that $-c(ab+1)=a^2+\epsilon p+b^2p^2$. If $ab=-1$ and $a^2+\epsilon p+b^2p^2=0$, we have $a^2-\epsilon abp+b^2p^2=0$. Since in this case $b\neq 0$, we have $-\frac{a}{bp}$ is a primitive $3$rd root of unity. It follows that $\frac{a^2}{p}=\frac{a}{bp}\cdot ab$ has normalized valuation $0$, which is impossible for $a\in\qp$. Therefore, $ab\neq-1$ and 
        $$c=-\frac{a^2+\epsilon p+b^2p^2}{ab+1},~d=-\frac{a^3+\epsilon ap-bp^2}{ab+1}\qaq e=-a.$$
        Therefore, we have $D\simeq D_{(a,b)}^{\epsilon,\mu}$. Denote $S=\br{\begin{smallmatrix}
            -b & a\\
            \frac{a+bc}{p^2} & b
        \end{smallmatrix}}$ and $M=\br{\begin{smallmatrix}
            0 & -p^2\\
            1 & c
        \end{smallmatrix}}$. We have $\mat_{\mb}(\varphi)=\br{\begin{smallmatrix}
            0 & M\\
            I_2 & SM
        \end{smallmatrix}}$
        and it follows that
        $$\mat_{\mb}(u)=\br{\begin{smallmatrix}
            I_2 & 0\\
            S & I_2
        \end{smallmatrix}},~\mat_{\mb}(\F^{\el})=\br{\begin{smallmatrix}
            0 & M\\
            \frac{1}{p}I_2 & 0
        \end{smallmatrix}}\qaq\mat_{\mb}(\mu(p))=\br{\begin{smallmatrix}
            pI_2 & 0\\
            0 & I_2
        \end{smallmatrix}}.$$ 
        Let $H$ be an adapted lattice of $D$ and $H_i=H\cap D_i$ for $i=0,1$. By Theorem \ref{hodgedecomp} (d), we have $H=H_1\oplus H_0$. For $i=0,1$, we denote $\pr_i\mathcal B=\mathcal B\cap D_i$. From now on, we identify $D_1$ and $D_0$ with lattices in $\qp^2$ canonically by identifying the bases $\pr_i\mathcal B$ with the canonical basis of $\qp^2$. Under this identification, the definition of an adapted lattice and Theorem \ref{hodgedecomp} (d) imply the containment of $\zp$-submodules of $\qp^2$
        $$\frac{1}{p}H_1\subset H_0,~MH_0\subset H_1\qaq SH_1\subset H_0.$$
        From the first two inclusions, we deduce that $\frac{1}{p}MH_1\subset H_1$. Note that $\det(\frac{1}{p}M)=1$, so the first two inclusions are indeed equalities. The third inclusion turns into $pSH_1\subset H_1$, which implies that $c+\epsilon p=\det(pS)\in\zp$. Moreover, if $c+\epsilon p$ is invertible in $\zp$, we have $(u-\id)H_1=H_0$ in $D$.  Using the notations in Theorem \ref{hodgedecomp} (d), on $\bar H$ we have $(\bar u-\id_{\bar H})\bar H=\bar H_0$ hence $\bar H_0\subset L_{r-1}$. Since $L_{r-1}$ is $\bar\F^{\el}$-stable, we have $\bar H_1=\bar\F^{\el}(\bar H_0)\subset L_{r-1}$. It follows that $L_{r-1}=\bar H$, contradiction. In conclusion, we have $c+\epsilon p\in p\zp$, hence $c\in p\zp$. The condition (d) of Theorem \ref{hodgedecomp} can be summarized to the existence of a lattice $N$ in $\qp^2$ satisfying the following conditions.
        \begin{enumerate}[label=(\roman*)]
            \item $\br{\begin{smallmatrix}
            0 & -1\\
            1 & \frac{c}{p}
        \end{smallmatrix}}N=N$,
            \item $\br{\begin{smallmatrix}
            -pb & a\\
            a+bc & pb
        \end{smallmatrix}}N\subset N$,
            \item on $N/pN$, the image of the $\fp$-endomorphism defined by $\br{\begin{smallmatrix}
            -pb & a\\
            a+bc & pb
        \end{smallmatrix}}$ is stable under the action of $\br{\begin{smallmatrix}
            0 & -1\\
            1 & \frac{c}{p}
        \end{smallmatrix}}$.
        \end{enumerate}  
        Note that condition (iii) always holds. Indeed, denote the endomorphisms of $N/pN$ defined by the matrices $\br{\begin{smallmatrix}
            0 & -1\\
            1 & \frac{c}{p}
        \end{smallmatrix}}$ and $\br{\begin{smallmatrix}
            -pb & a\\
            a+bc & pb
        \end{smallmatrix}}$ by $\bar M$ and $\bar S$ respectively, we have $\bar M\circ\bar S+\bar S\circ\bar M=\frac{c}{p}\bar S$. Apply both sides on $N/pN$ and note that $\bar M(N/pN)=N/pN$, we have $\bar M(\bar S(N/pN))\subset\bar S(N/pN)$ as claimed. Finally, Lemma \ref{rathertechnical} shows that the canonical basis of $\qp^4$ is adapted for the module $D^{\epsilon,\mu}_{(a,b)}$ if and only if $v_p(a)\ge 1$ and $v_p(b)\ge 0$, or $v_p(a)=v_p(b)+1$.

        Suppose $X_D=(B)$ and for all $x\in D_1$, $u^{-1}\varphi^2(x)$ is proportional to $x$. It follows that there exists a fix $\epsilon'\in\st{\pm 1}$ such that $u^{-1}\varphi^2(x)=\epsilon'px$. Let $(x,y)$ be a basis for $D_1$. We deduce that $\mathcal B=(x,y,\varphi(x),\varphi(y))$ is an adapted basis of $D$ and under which the matrix of $\varphi$ has the form
        $$\mat_{\mathcal B}(\varphi)=\br{\begin{smallmatrix}
            0 & 0 & \epsilon'p & 0\\
            0 & 0 & 0 & \epsilon'p\\
            1 & 0 & a & c\\
            0 & 1 & b & d
        \end{smallmatrix}}$$ 
        for some $a,b,c,d\in\qp$. A direct calculation shows
        $$\chi_{\varphi}(X)=X^4-(a+d)X^3+(ad-bc-2\epsilon'p)X^2+(a+d)\epsilon' pX+p^2=X^4+\epsilon pX^2+p^2.$$
        Comparing the coefficients, we get $a+d=0$ and $ad-bc=(\epsilon+2\epsilon')p$. If $b=0$, we have $a^2=-(\epsilon+2\epsilon')p\in\st{\pm p,\pm 2p,\pm 3p}$, which is impossible for $p\ge 5$. As before, we deduce that
        $$\mat_{\pr_1\mathcal B,\pr_0\mathcal B}(u|_{D_0})=\frac{\epsilon'}{p}\br{\begin{smallmatrix}
            a & -\frac{a^2+(\epsilon+2\epsilon')p}{b}\\
            b & -a
        \end{smallmatrix}}$$
        and that for any adapted lattice $H\subset D$, after identifying $D_1$ and $D_0$, we have $H_0=\frac{1}{p}H_1$ and  
        $$(\epsilon+2\epsilon')p=\det\br{\begin{smallmatrix}
            a & -\frac{a^2+(\epsilon+2\epsilon')p}{b}\\
            b & -a
        \end{smallmatrix}}\in p\zp.$$
        Therefore, we have $(\bar u-\id_{\bar H})\bar H_1\subsetneq\bar H_0$ is a $1$-dimensional subspace. For the sake of Theorem \ref{hodgedecomp} (d), take a nonzero element $t\in(\bar u-\id_{\bar H})\bar H_1$, let $r=2$ and $L_1=\Span(t,\F^{\el}(t))$. Since $\br{\begin{smallmatrix}
            a & -\frac{a^2+(\epsilon+2\epsilon')p}{b}\\
            b & -a
        \end{smallmatrix}}^2=-(\epsilon+2\epsilon')p$, the chain $0=L_0\subsetneq L_1\subsetneq L_2=\bar H$ satisfies the requirements.

        Finally, we should verify that $D\simeq D^{\epsilon,iso}_{\epsilon'}$. We can take $P\in\GL_2(\qp)$ such that 
        $$P\br{\begin{smallmatrix}
            a & -\frac{a^2+(\epsilon+2\epsilon')p}{b}\\
            b & -a
        \end{smallmatrix}}P^{-1}=\br{\begin{smallmatrix}
            0 & -(\epsilon+2\epsilon')p\\
            1 & 0
        \end{smallmatrix}},$$
        since these two matrices have the same characteristic polynomial which is separable. It follows that, under the canonical bases of $\qp^4$ as the underlying vector space of both $D$ and $D^{\epsilon,iso}_{\epsilon'}$, the matrix $\br{\begin{smallmatrix}
            P & 0\\
            0 & P
        \end{smallmatrix}}\in\GL_4(\qp)$ defines an isomorphism between these two modules, and our assertion follows.

        Suppose $X_D=(A)$. It's clear that $D$ is decomposed into $4$ subspaces of dimension $1$. 
        $$D=\underbrace{D_{[1100]}\oplus D_{[1001]}}_{D_1}\oplus \underbrace{D_{[0011]}\oplus D_{[0110]}}_{D_0}.$$
        Take a nonzero $x\in D_{[1100]}$, one can check that $\mb=(x,\F^{\el}(x),(\F^{\el})^2(x),(\F^{\el})^3(x))$ is an adapted basis of $D$. Under this basis, the matrix of $u$, $\F^{\el}$ and $\mu(p)$ have the following forms, where $a,b,c,d\in\qp$.
        $$\mat_{\mb}(u)=\br{\begin{smallmatrix}
            1 & 0 & 0 & 0\\
            0 & 1 & 0 & 0\\
            a & b & 1 & 0\\
            c & d & 0 & 1
        \end{smallmatrix}},~\mat_{\mb}(\F^{\el})=\br{\begin{smallmatrix}
            0 & 0 & 0 & -1\\
            1 & 0 & 0 & 0\\
            0 & 1 & 0 & 0\\
            0 & 0 & 1 & 0
        \end{smallmatrix}}\qaq\mat_{\mb}(\mu(p))=\br{\begin{smallmatrix}
            p & 0 & 0 & 0\\
            0 & p & 0 & 0\\
            0 & 0 & 1 & 0\\
            0 & 0 & 0 & 1
        \end{smallmatrix}}.$$
        We deduce from $\varphi=u\circ \F^{\el}\circ\mu(p)$ that
        $$\mat_{\mb}(\varphi)=\br{\begin{smallmatrix}
            0 & 0 & 0 & -1\\
            p & 0 & 0 & 0\\
            bp & p & 0 & -a\\
            dp & 0 & 1 & -c
        \end{smallmatrix}}\mbox{ and }\chi_{\varphi}(X)=X^4+cX^3+(pd+a)X^2+pbX+p^2=X^4+\epsilon pX^2+p^2.$$
        Comparing the coefficients, we have
        $$b=c=0\qaq dp=a-\epsilon p\qaq\mat_{\mathcal B}(\varphi)=\br{\begin{smallmatrix}
            0 & 0 & 0 & -1\\
            p & 0 & 0 & 0\\
            0 & p & 0 & -a\\
            a-\epsilon p & 0 & 1 & 0
        \end{smallmatrix}}.$$
        It follows that $D\simeq D_{a}^{\epsilon,\nu}$. Theorem \ref{hodgedecomp} (d) requires that any adapted lattice $H\subset D$ has a $\zp$-basis $\widetilde{\mathcal B}$ in which each element is proportional to one of elements in $\mathcal B=((f^{el})^{j}(x))_{0\le j\le 3}$. We deduce from the definition of adapted lattice, namely $H=\varphi(H)+\frac{1}{p}\varphi(\fil_1H)$, that $\widetilde{\mathcal B}=(\alpha_j\cdot(f^{el})^{j}(x))_{0\le j\le 3}$ for $\alpha_j\in\qp^\times$ such that $v_p(\alpha_i)=v_p(\alpha_j)$ for any $0\le i,j\le 3$. Without loss of generality, we can thereby suppose that $\alpha_i=\alpha_j$ for any $0\le i,j\le 3$. The matrices of $\F^{\el}$ and $u$ under $\widetilde{\mathcal B}$ are 
        $$\mat_{\widetilde{\mathcal B}}(\F^{\el})=\br{\begin{smallmatrix}
            0 & 0 & 0 & -1\\
            1 & 0 & 0 & 0\\
            0 & 1 & 0 & 0\\
            0 & 0 & 1 & 0
        \end{smallmatrix}}
        \qaq
        \mat_{\widetilde{\mathcal B}}(u)=\br{\begin{smallmatrix}
            1 & 0 & 0 & 0\\
            0 & 1 & 0 & 0\\
            a & 0 & 1 & 0\\
            0 & \frac{a-\epsilon p}{p} & 0 & 1
        \end{smallmatrix}}.$$ 
        We claim that $\bar u=\id_{\bar H}$. Indeed, we deduce from $\bar u(H)=H$ that the numbers $\frac{a-\epsilon p}{p},a$ belong to $\zp$, which implies $a\in p\zp$. Therefore, the image of $\bar u-\id_{\bar H}$ is contained in the $1$-dimensional linear subspace of $\bar H$ generated by $\overline{\varphi^2(x)}$. Theorem \ref{hodgedecomp} (d) requires that this image is contained in an $L_{r-1}\subset\bar H$ which is stable under $\F^{\el}$. However, there is no proper $\F^{\el}$-submodule of $\bar H$ containing $\overline{\varphi^2(x)}$, and our assertion follows. We deduce immediately that $\frac{a-\epsilon p}{p}\in p\zp$.

        %Finally, we should prove that the filtered $\varphi$-modules defined by the three cases are mutually non-isomorphic. We claim that, among these cases, $\fil_1D\cap\varphi(\fil_1D)\neq 0$ if and only if $D$ lives in case (1) with $a=\epsilon p$. Now we can focus on the cases when $\fil_1D\cap\varphi(\fil_1D)=0$. We use $\pr_1$ to denote the projection of $D$ onto $\fil_1D$ along $\varphi(\fil_1D)$. In case (1) the matrix of $\varphi$ under the basis $\mathcal B'=(y,x,\varphi(y),\varphi(x))$ is
        %$$\mat_{\mathcal B'}(\varphi)=\br{\begin{smallmatrix}
        %    0 & 0 & a-\epsilon p & 0\\
        %    0 & 0 & 0 & \frac{p^2}{a-\epsilon p}\\
        %    1 & 0 & 0 & \frac{1}{a-\epsilon p}\\
        %    0 & 1 & -p^2-a^2+\epsilon ap & 0
        %\end{smallmatrix}}.$$ 
        %Therefore, in this case $\tr(\pr_1\circ\varphi^2|_{D_1})=a-\epsilon p+\frac{p^2}{a-\epsilon p}\notin p\zp$ since $a-\epsilon p\in p^2\zp$. On the other hand, in case (2) and (3) we have $\tr(\pr_1\circ\varphi^2|_{D_1})=c$ and $2\epsilon' p$ respectively, both contained in $p\zp$. This shows the modules in case (1) won't occur in (2) or (3). On the other hand, in case (2) or (3), the automorphism $\pr_1\circ\varphi^2|_{D_1}$ of $D_1$ is a scalar-multiplication if and only if $D$ lives in case (3). This completes the proof.
\end{proof} 
\subsection{Construction of certain skew forms}
By Lemma \ref{ssfilphimod}, to verify that $D^{prod}_{\epsilon'}$ and the filtered $\varphi$-modules listed in Theorem \ref{mainclass} belong to $\mfs$, it suffices to verify the \hyperref[ssfilphimod]{$\mathbf{S3}$} for them.

On $D_{\epsilon'}^{prod}$, take a $\qp$-basis $(x,y)$ of $\fil_1D$ and consider the skew form $\delta:D_{\epsilon'}^{prod}\times D_{\epsilon'}^{prod}\to\qp$ whose matrix under the basis $\mb:=(x,y,\varphi(x),\varphi(y))$ is
    $$\mat_{\mb}(\delta)=\br{\begin{smallmatrix}
        0 & 0 & 0 & 1\\
        0 & 0 & -\epsilon' & 0\\
        0 & \epsilon' & 0 & 0\\
        -1 & 0 & 0 & 0
    \end{smallmatrix}}.$$
It's clear that $\delta$ is nondegenerate, and one can check that under $\delta$, $\varphi$ is a $p$-similitude and $\fil_1D$ is totally isotropic. 

\begin{lem}\label{existcyc}
    Let $p$ be any prime, $D\in\mf_{\qp}^{\f}$ be a filtered $\varphi$-modules satisfying Condition \hyperref[ssfilphimod]{$\mathbf{S1}$} with $\chi_{\varphi}$ supersingular and separable. Then, there exists a $\varphi$-cyclic vector in $\fil_1D$.
\end{lem}
\begin{proof}
    The irreducible decomposition $\chi_{\varphi}(X)=f_1\cdots f_r$ in $\qp[X]$ gives rise to a decomposition of $D$ into $\varphi$-stable subspaces
    $$D=D_1\oplus\cdots\oplus D_r.$$
    We claim that the projections of $\fil_1D$ to each $D_i$ are all nonzero. Otherwise, we have $\fil_1D\subseteq D'$ for some $\varphi$-stable proper subspace $D'\subsetneq D$ and the weak admissibility yields
    $$t_H(D')=\sum_{i\in\mz}i\dim(\fil_iD\cap D')\stackrel{\hyperref[ssfilphimod]{\mathbf{S1}}}=\dim(\fil_1D\cap D')=2\stackrel{\f}{\le} t_N(D').$$
    Therefore, $t_N(D')=t_N(D)=2$. Note that $t_N(D')=\frac{1}{2}\dim D'$ by the supersingularness of $\chi_{\varphi}(X)$, which implies that $D'=D$, contradiction. Finally, we can pick an $x\in\fil_1D$ whose projections to each $D_i$ are all nonzero\footnote{Indeed, we use that $\qp$ has infinitely many elements here.}. Such an $x$ is $\varphi$-cyclic since $\varphi$ is separable.
\end{proof}

For the case $\chi_{\varphi}(X)=X^4+\epsilon pX^2+p^2$ for $\epsilon\in\st{0,\pm1}$, we have the following Proposition, which completes the proof of Theorem \ref{mainclass}.

\begin{prop}\label{automat}
    Let $p\ge 7$ be a prime and $\epsilon\in\st{0,\pm1}$. Let $D\in\mf_{\qp}^{\f}$ be a filtered $\varphi$-modules satisfying Condition \hyperref[ssfilphimod]{$\mathbf{S1}$} with $\chi_{\varphi}(X)=X^4+\epsilon pX^2+p^2$. Then $D$ satisfies \hyperref[ssfilphimod]{$\mathbf{S2}$} and \hyperref[ssfilphimod]{$\mathbf{S3}$}, which therefore arises from an abelian surface over $\qp$.
\end{prop}
\begin{proof}
    Note that the polynomial $\chi_{\varphi}(X)=X^4+\epsilon pX^2+p^2$ is separable. Therefore, $D$ satisfies \hyperref[ssfilphimod]{$\mathbf{S2}$}. To verify that $D$ satisfies \hyperref[ssfilphimod]{$\mathbf{S3}$}, by Lemma \ref{existcyc} we can take a $\varphi$-cyclic vector $x\in\fil_1D$ which extends to a $\qp$-basis $(x,y)$ of $\fil_1D$. Denote by $\alpha_1,\alpha_2,\alpha_3\in\qp$ the coordinates of $y$ under the decomposition
    $$y\equiv\sum_{1\le j\le 3}\alpha_j\varphi^j(x)\mod\qp x.$$
    Take $x_1,x_2,x_3\in\qp$ such that 
    $$\begin{cases}
        (\alpha_1-(\epsilon-1)p\alpha_2)x_1+\alpha_2x_2=0\\
        (x_1,x_2)\neq(0,0)\\
        x_3=(1-\epsilon)px_1
    \end{cases}$$
    Consider the skew form $\delta:D\times D\to\qp$ whose matrix under the basis $\mb:=(\varphi^i(x))_{0\le i\le 3}$ is
    $$\mat_{\mb}(\delta)=\br{\begin{smallmatrix}
        0 & x_1 & x_2 & x_3\\
        -x_1 & 0 & px_1 & px_2\\
        -x_2 & -px_1 & 0 & p^2x_1\\
        -x_3 & -px_2 & -p^2x_1 & 0
    \end{smallmatrix}}.$$
    It follows that 
    $$\det(\mat_{\mb}(\delta))=(p^2x_1^2-px_2^2+px_1x_3)^2\xlongequal{x_3=(1-\epsilon)px_1}p^2((2-\epsilon)px_1^2-x_2^2)^2.$$
    The equation $(2-\epsilon)px_1^2=x_2^2$ has no solution for any $(0,0)\neq(x_1,x_2)\in\qp^2$ since $2\nmid v_p((2-\epsilon)px_1^2)$. It follows that $\delta$ is nondegenerate. One can check that under $\delta$, the endomorphism $\varphi$ is a $p$-similitude and the subspace $\fil_1D$ is totally isotropic. 
\end{proof}
\section{$p$-adic monodromy groups of supersingular abelian surfaces over $\qp$}\label{classmonodromy}
After establishing Theorem \ref{mainthm0} and Theorem \ref{mainclass}, we classify all filtered $\varphi$-modules $D$ arising from supersingular abelian surfaces over $\qp$. We proceed to compute the monodromy groups $H_D$ associated with these filtered $\varphi$-modules.

% A key advantage of our classification is that the Hodge cocharacter is simply $\mu(t)=\diag(t,t,1,1)\in\GL(D)$, while $\varphi_D$ is given by a concrete matrix. This allows us to compute $H_D$ via Proposition \ref{zardense}. Moreover, the supersingularity implies that $\varphi_D$ has finite order in $\mathrm{PGL}(D)$. Therefore, the neutral component $H_D^\circ$ is generated by the family of tori $\st{\varphi_D^i\circ\mu(t)\circ\varphi_D^{-i}}_{i\in\mz}$ and scalars.

For later use, we record the common strategy used in the computations of the section. 

\begin{lem}\label{monogpcheck}
    Let $D\in\mfs$ and $G\subset\GL(D)$ be a connected closed subgroup of $\GL(D)$. Denote by $\mu:\g_m\to\GL(D)$ the Hodge cocharacter of $D$. Suppose that 
    \begin{enumerate}
        \item $G$ contains the scalar torus and $\mu(\mathbb G_m)$, and is stable under conjugation by $\varphi$;
        \item $\lie(G)$ is the Lie subalgebra of $\gl(D)$ generated by $I_D$ and the elements
    $$
    \varphi^i\D\mu(1)\varphi^{-i},\qquad i\in\mathbb Z.
    $$
    \end{enumerate}
    Then $G=H^\circ_D$.
\end{lem}
\begin{proof}
By Proposition \ref{zardense}, the group $H_D^\circ$ is generated by the scalar torus and the tori
$$
\varphi^i\mu(\mathbb G_m)\varphi^{-i},\qquad i\in\mathbb Z.
$$
By the first assumption, these tori are contained in $G$. By the second assumption and Lemma \ref{Liealg}, their Lie algebras generate $\lie(G)$. Hence $G=H_D^\circ$.
\end{proof}

After fixing an adapted basis, let $A=\mat(\varphi)$ and $P=\D\mu(1)=\diag(1,1,0,0)\in\gl_4$. Let 
$$\mathfrak l(A,P)=\ang{I_4,A^iPA^{-i}\mid i\in\mz}_{\lie}\subset\gl_4$$
be the Lie subalgebra generated by the matrices $A^iPA^{-i}$ together with $I_4$. In the cases considered below, the Frobenius $\varphi$ is semisimple and its eigenvalues are supersingular $p$-Weil numbers. Hence $\varphi^m$ is scalar for some $m>0$. Thus, by Proposition \ref{zardense}, the neutral component $H_D^\circ$ is generated by the scalar torus together with the tori $\varphi^i\mu(\g_m)\varphi^{-i}$.

Consequently, in each of the following cases with a proposed connected algebraic subgroup $G\subset\GL_4$, it is enough to do the following two checks. First, one verifies that all the tori $\varphi^i\mu(\g_m)\varphi^{-i}$ are contained in $G$. In practice this will be verified by showing that $\mu(\g_m)\subset G$ and that $G$ is stable under conjugation by $A$. Second, one verifies that $\lie~G=\mathfrak l(A,P)$. Then Lemma \ref{Liealg} implies $H_D^\circ=G$.
\subsection{The case $\chi_{\varphi}(X)=(X^2\pm p)^2$}
When $D\simeq D_{\pm1}^{prod}$, we can compute the monodromy groups directly.
\begin{prop}\label{x2pmp}
    For $\epsilon'=\pm1$, we have $H_{D^{prod}_{\epsilon'}}^\circ\simeq\mathbf G_{m,\qp}^2$.
\end{prop}
\begin{proof}
    Under the canonical adapted basis $\mathcal B$ of $D_{\pm1}^{prod}$, $\varphi$ and the Hodge cocharacter $\mu$ becomes
    $$\mat_{\mathcal B}(\varphi)=\br{\begin{smallmatrix}
        0 & 0 & \epsilon' p & 0\\
        0 & 0 & 0 & \epsilon' p\\
        1 & 0 & 0 & 0\\
        0 & 1 & 0 & 0
    \end{smallmatrix}}\qaq\mu:t\mapsto\br{\begin{smallmatrix}
        t & 0 & 0 & 0\\
        0 & t & 0 & 0\\
        0 & 0 & 1 & 0\\
        0 & 0 & 0 & 1
    \end{smallmatrix}}.$$
    We can thereby verify that
    $$\st{\varphi^i\circ\mu(t)\circ\varphi^{-i}}_{i\in\mz}=\st{\diag(t,t,1,1),\diag(1,1,t,t)}_{t\in\qp^\times},$$
    and the closed subgroup of $\GL_D=\GL_{4}$ generated by these tori together with scalars is clearly $\g_{m}^2$ .
\end{proof}
\subsection{The case $\chi_{\varphi}(X)=X^4+\epsilon pX^2+p^2$}
To compute the monodromy groups of $D_a^{\epsilon,\nu}$, $D_{(a,b)}^{\epsilon,\mu}$ and $D_{\epsilon'}^{\epsilon,iso}$, we need to work on the Lie algebras side. Invoking Lemma \ref{Liealg}, we can compute the corresponding Lie algebras in a straightforward manner.
\begin{prop}\label{sit1}
    Let $D=D_{a}^{\epsilon,\nu}$. If $a=\epsilon=0$, we have $H_D^\circ\simeq\g_{m}^3$. Otherwise, we have
    $$H_D^\circ\simeq\GL_2\times_{\det}\GL_2:=\st{\br{\begin{smallmatrix}
            g & 0\\
            0 & g'
        \end{smallmatrix}}~\!\Big|~\!g,g'\in\GL_2(\qp),~\!\det(g)=\det(g')}.$$
\end{prop}
\begin{proof}
    Under the canonical adapted basis $\mathcal B=(e_i)_{1\le i\le 4}$ of $D_{a}^{\epsilon,\nu}$, $\varphi$ and the Hodge cocharacter $\mu$ becomes
    $$\mat_{\mathcal B}(\varphi)=\br{\begin{smallmatrix}
            0 & 0 & 0 & 1\\
            -p^2 & 0 & 0 & 0\\
            0 & 1 & 0 & -a\\
            a-\epsilon p & 0 & 1 & 0
        \end{smallmatrix}}\qaq\mu:t\mapsto\br{\begin{smallmatrix}
        t & 0 & 0 & 0\\
        0 & t & 0 & 0\\
        0 & 0 & 1 & 0\\
        0 & 0 & 0 & 1
    \end{smallmatrix}}$$
    where $a\in\epsilon p+p^2\zp$. Denote $\mathcal B'=(e_1,e_3,e_2,e_4)$, we have
    $$\mat_{\mathcal B'}(\varphi)=\br{\begin{smallmatrix}
            0 & 0 & 0 & 1\\
            0 & 0 & 1 & -a\\
            -p^2 & 0 & 0 & 0\\
            a-\epsilon p & 1 & 0 & 0
        \end{smallmatrix}}\qaq\mu':t\mapsto\br{\begin{smallmatrix}
        t & 0 & 0 & 0\\
        0 & 1 & 0 & 0\\
        0 & 0 & t & 0\\
        0 & 0 & 0 & 1
    \end{smallmatrix}}$$
We verify the two assumptions of Lemma \ref{monogpcheck}.

\noindent\textbf{Verification of Lemma \ref{monogpcheck}(1).} With respect to the decomposition determined by the basis $B'$, the group
$G=\GL_2\times_{\det}\GL_2$ consists of block diagonal matrices whose two diagonal blocks have the same determinant. Therefore, the scalar torus and the elements $\mu'(\mathbb G_m)\subset G$ are contained in $G$. 

It remains to check stability under conjugation by $\varphi$. In the basis $B'$, the matrix of $\varphi$ is block anti-diagonal. Hence conjugation by $\varphi$ preserves the determinants of the two blocks of any element in $\GL_2\times_{\det}\GL_2$. Therefore it preserves the condition defining $G$, and hence $G$ is stable under conjugation by $\varphi$.

    % \begin{lem}
    %     If $a\neq 0$ or $\epsilon\neq 0$, the Lie algebras $\mat_{\mathcal B'}(\varphi)^i\D\mu'(1)\mat_{\mathcal B'}(\varphi)^{-i}$ of the tori in $\mathcal F$ generate
    %     $$\st{\br{\begin{smallmatrix}
    %         g & 0\\
    %         0 & g'
    %     \end{smallmatrix}}~\!\Big|~\!g,g'\in\mathfrak{gl}_2(\mq_p),~\!\tr(g)=\tr(g')}=\lie(\GL_{2}\times_{\det}\GL_{2}).$$
    % \end{lem}
\noindent\textbf{Verification of Lemma \ref{monogpcheck}(2).} Note that $\D\mu'(1)=\diag(1,0,1,0)$. Denote
$$\mathfrak t:=\st{\mathrm{diag}(t_1,t_2,t_3,t_4)\in\gl_4(\qp)~\!|~\!t_1+t_2=t_3+t_4}$$
and
$$\mathfrak a:=\st{\br{\begin{smallmatrix}
            0 & \alpha_{1,2} & 0 & 0\\
            \alpha_{2,1} & 0 & 0 & 0\\
            0 & 0 & 0 & \beta_{1,2}\\
            0 & 0 & \beta_{2,1} & 0
        \end{smallmatrix}}\in\gl_4(\qp)~\!\Big|~\!\alpha_{i,j},\beta_{i,j}\in\qp}.$$
We have
        $$\mathfrak t+\mathfrak a=\st{\br{\begin{smallmatrix}
            g & 0\\
            0 & g'
        \end{smallmatrix}}~\!\Big|~\!g,g'\in\mathfrak{gl}_2(\mq_p),~\!\tr(g)=\tr(g')}=\lie(G).$$

\noindent\textbf{Case 1.} If $\epsilon=\pm1$, using $\epsilon^2=1$ iterately we have
    $$
    \mat_{\mathcal B'}(\varphi)^i\D\mu'(1)\mat_{\mathcal B'}(\varphi)^{-i}=
    \begin{cases}
        \br{\begin{smallmatrix}
            0 & 0 & 0 & 0\\
            a & 1 & 0 & 0\\
            0 & 0 & 1 & 0\\
            0 & 0 & \frac{-a+\epsilon p}{p^2} & 0
        \end{smallmatrix}}\quad &i=1\\
        \br{\begin{smallmatrix}
            \frac{-a^2+a\epsilon p}{p^2} & \frac{-a+\epsilon p}{p^2} & 0 & 0\\
            \frac{a^3+ap^2-a^2\epsilon p}{p^2} & \frac{a^2+p^2-a\epsilon p}{p^2} & 0 & 0\\
            0 & 0 & 0 & 0\\
            0 & 0 & \frac{-\epsilon}{p} & 1
        \end{smallmatrix}}&i=2\\
        \br{\begin{smallmatrix}
            \frac{p-\epsilon a}{p} & \frac{-\epsilon}{p} & 0 & 0\\
            \frac{-ap+\epsilon a^2}{p} & \frac{\epsilon a}{p} & 0 & 0\\
            0 & 0 & \frac{p-\epsilon a}{p} & a-\epsilon p\\
            0 & 0 & \frac{-a}{p^2} & \frac{\epsilon a}{p}
        \end{smallmatrix}}&i=3\\
        \br{\begin{smallmatrix}
            \frac{a\epsilon p-a^2}{p^2} & \frac{-a}{p^2} & 0 & 0\\
            \frac{a^3-2a^2\epsilon p+2ap^2-\epsilon p^3}{p^2} & \frac{a^2-a\epsilon p+p^2}{p^2} & 0 & 0\\
            0 & 0 & 0 & \epsilon p\\
            0 & 0 & 0 & 1
        \end{smallmatrix}}&i=4\\
        \br{\begin{smallmatrix}
            1 & 0 & 0 & 0\\
            \epsilon p-a & 0 & 0 & 0\\
            0 & 0 & 0 & a\\
            0 & 0 & 0 & 1
        \end{smallmatrix}}&i=5.
    \end{cases}
    $$
    Using Lemma \ref{Liegene} (ii), the first (resp. the last) matrix together with $\D\mu'(1)$ generates $\diag(0,1,1,0)$ (resp. $\diag(1,0,0,1)$). Therefore, the Lie algebra generated by the Lie algebras of the tori $\varphi^i\mu'(\g_m)\varphi^{-i}$ contains $\mathfrak t$. Using Lemma \ref{Liegene} (ii) again, these Lie algebras generate
    $$\begin{matrix}
        {\br{\begin{smallmatrix}
            0 & 0 & 0 & 0\\
            a & 0 & 0 & 0\\
            0 & 0 & 0 & 0\\
            0 & 0 & \frac{-a+\epsilon p}{p^2} & 0
        \end{smallmatrix}}},
        &{\br{\begin{smallmatrix}
            0 & 0 & 0 & 0\\
            \frac{a^3+ap^2-a^2\epsilon p}{p^2} & 0 & 0 & 0\\
            0 & 0 & 0 & 0\\
            0 & 0 & \frac{-\epsilon}{p} & 0
        \end{smallmatrix}}},
        &{\br{\begin{smallmatrix}
            0 & \frac{-a+\epsilon p}{p^2} & 0 & 0\\
            0 & 0 & 0 & 0\\
            0 & 0 & 0 & 0\\
            0 & 0 & 0 & 0
        \end{smallmatrix}}},
        &{\br{\begin{smallmatrix}
            0 & 0 & 0 & 0\\
            \frac{-ap+\epsilon a^2}{p} & 0 & 0 & 0\\
            0 & 0 & 0 & 0\\
            0 & 0 & \frac{-a}{p^2} & 0
        \end{smallmatrix}}},
        &{\br{\begin{smallmatrix}
            0 & \frac{-\epsilon}{p} & 0 & 0\\
            0 & 0 & 0 & 0\\
            0 & 0 & 0 & a-\epsilon p\\
            0 & 0 & 0 & 0
        \end{smallmatrix}}}\\
        m_1 & m_2 & m_3 & m_4 & m_5
    \end{matrix}$$
        and
        $$
        \begin{matrix}
        {\br{\begin{smallmatrix}
            0 & 0 & 0 & 0\\
            \frac{a^3-2a^2\epsilon p+2ap^2-\epsilon p^3}{p^2} & 0 & 0 & 0\\
            0 & 0 & 0 & 0\\
            0 & 0 & 0 & 0
        \end{smallmatrix}}},
        &{\br{\begin{smallmatrix}
            0 & \frac{-a}{p^2} & 0 & 0\\
            0 & 0 & 0 & 0\\
            0 & 0 & 0 & \epsilon p\\
            0 & 0 & 0 & 0
        \end{smallmatrix}}},
        &{\br{\begin{smallmatrix}
            0 & 0 & 0 & 0\\
            0 & 0 & 0 & 0\\
            0 & 0 & 0 & a\\
            0 & 0 & 0 & 0
        \end{smallmatrix}}},
        &{\br{\begin{smallmatrix}
            0 & 0 & 0 & 0\\
            \epsilon p-a & 0 & 0 & 0\\
            0 & 0 & 0 & 0\\
            0 & 0 & 0 & 0
        \end{smallmatrix}}}.\\
        m_6 & m_7 & m_8 & m_9
        \end{matrix}
        $$
        It remains to show that the $\qp$-subspace spanned by the matrices above is always $\mathfrak a$. We distinguish the two cases $a=\epsilon p$ and $a\neq\epsilon p$. If $a=\epsilon p$, the matrices $m_1,m_2,m_5,m_7$ form a $\qp$-basis of $\mathfrak a$. Otherwise, the matrices $m_1,m_3,m_5,m_9$ form a $\qp$-basis of $\mathfrak a$. In both cases, we have $H_D^\circ=\GL_{2}\times_{\det}\GL_{2}$ by Lemma \ref{Liealg}. 

    \noindent\textbf{Case 2.} If $\epsilon=0$ and $a\neq 0$, we have
    $$\mat_{\mathcal B'}(\varphi)^i\D\mu'(1)\mat_{\mathcal B'}(\varphi)^{-i}=\begin{cases}
        \br{\begin{smallmatrix}
            0 & 0 & 0 & 0\\
            a & 1 & 0 & 0\\
            0 & 0 & 1 & 0\\
            0 & 0 & \frac{-a}{p^2} & 0
        \end{smallmatrix}}\quad &i=1\\
        \br{\begin{smallmatrix}
            \frac{-a^2}{p^2} & \frac{-a}{p^2} & 0 & 0\\
            \frac{a^3+ap^2}{p^2} & \frac{a^2+p^2}{p^2} & 0 & 0\\
            0 & 0 & 0 & 0\\
            0 & 0 & 0 & 1
        \end{smallmatrix}}\quad &i=2\\
        \br{\begin{smallmatrix}
            1 & 0 & 0 & 0\\
            -a & 0 & 0 & 0\\
            0 & 0 & 0 & a\\
            0 & 0 & 0 & 1
        \end{smallmatrix}}\quad &i=3
    \end{cases}$$
    Again by Lemma \ref{Liegene} (ii), the first (resp. the last) matrix together with $\D\mu'(1)$ generates $\diag(0,1,1,0)$ (resp. $\diag(1,0,0,1)$). Therefore, the Lie algebra generated by the Lie algebras of the tori $\varphi^i\mu'(\g_m)\varphi^{-i}$ contains $\mathfrak t$. Using Lemma \ref{Liegene} (ii) again, these Lie algebras generate
    $$\begin{matrix}
        {\br{\begin{smallmatrix}
            0 & 0 & 0 & 0\\
            a & 0 & 0 & 0\\
            0 & 0 & 0 & 0\\
            0 & 0 & \frac{-a}{p^2} & 0
        \end{smallmatrix}}},
        &{\br{\begin{smallmatrix}
            0 & 0 & 0 & 0\\
            \frac{a^3+ap^2}{p^2} & 0 & 0 & 0\\
            0 & 0 & 0 & 0\\
            0 & 0 & 0 & 0
        \end{smallmatrix}}},
        &{\br{\begin{smallmatrix}
            0 & \frac{-a}{p^2} & 0 & 0\\
            0 & 0 & 0 & 0\\
            0 & 0 & 0 & 0\\
            0 & 0 & 0 & 0
        \end{smallmatrix}}},
        &{\br{\begin{smallmatrix}
            0 & 0 & 0 & 0\\
            -a & 0 & 0 & 0\\
            0 & 0 & 0 & 0\\
            0 & 0 & 0 & 0
        \end{smallmatrix}}},
        &{\br{\begin{smallmatrix}
            0 & 0 & 0 & 0\\
            0 & 0 & 0 & 0\\
            0 & 0 & 0 & a\\
            0 & 0 & 0 & 0
        \end{smallmatrix}}}\\
        n_1 & n_2 & n_3 & n_4 & n_5
    \end{matrix}$$
    By assumption $a\neq 0$, and the matrices $n_1,n_3,n_4,n_5$ form a $\qp$-basis of $\mathfrak a$. By Lemma \ref{Liealg}, we have $H_D^\circ=\GL_{2}\times_{\det}\GL_{2}$. 
        
    \noindent\textbf{Case 3.} If $a=\epsilon=0$, we have $\varphi^4=-p^2$ and
        $$\mat_{\mathcal B'}(\varphi)^i\D\mu'(1)\mat_{\mathcal B'}(\varphi)^{-i}=\begin{cases}
            \diag(0,1,1,0)\quad &i=1\\
            \diag(0,1,0,1)\quad &i=2\\
            \diag(1,0,0,1)\quad &i=3.
        \end{cases}$$
        It follows that the Lie algebras of the tori $\varphi^i\mu'(\g_m)\varphi^{-i}$ generate
    $$\mathfrak t:=\st{\mathrm{diag}(t_1,t_2,t_3,t_4)\in\gl_4(\qp)~\!|~\!t_1+t_2=t_3+t_4}.$$ 
    By Lemma \ref{Liealg}, we have $H_D^\circ=\st{\mathrm{diag}(t_1,t_2,t_3,t_4)\in\GL_4(\qp)~\!|~\!t_1t_2=t_3t_4}\simeq\g_{m}^3$.
\end{proof}
\begin{prop}\label{sit2}
    For any $a,b\in\qp$ such that $ab\neq-1$, let $D=D_{(a,b)}^{\epsilon,\mu}$ and $c=-\frac{a^2+\epsilon p+b^2p^2}{ab+1}$. Denote $S=\br{\begin{smallmatrix}
            -b & a\\
            \frac{a+bc}{p^2} & b
        \end{smallmatrix}}$ and $M=\br{\begin{smallmatrix}
            0 & -p^2\\
            1 & c
        \end{smallmatrix}}$.
    \begin{enumerate}[label=(\roman*)]
        \item If $c\neq-\epsilon p$ and $p\ge 5$, 
    $$H_D^\circ=\st{\left(\begin{smallmatrix}
            \alpha_{1,1}I_2+\alpha_{1,1}'M & \alpha_{1,2}S+\alpha_{1,2}'MS\\
            \alpha_{2,1}S+\alpha_{2,1}'MS & \alpha_{2,2}I_2+\alpha_{2,2}'M
        \end{smallmatrix}\right)~\!\Big|~\!\begin{smallmatrix}
            \alpha_{i,j},\alpha_{i,j}'\in\qp,~\det\not=0\\
            \alpha_{1,1}\alpha_{2,2}'-\alpha_{2,2}\alpha_{1,1}'=
            -\frac{\epsilon p+c}{p^2}(\alpha_{1,2}\alpha_{2,1}'-\alpha_{2,1}\alpha_{1,2}')
        \end{smallmatrix}}.$$
    In this case, we have $H_{D,\bar{\mq}_p}^\circ\simeq_{\bqp}\GL_{2}\times_{\det}\GL_{2}$.
        \item If $c=-\epsilon p$, we have
    $$H_D^\circ=\st{\br{\begin{smallmatrix}
            \alpha_{1,1}I_2 & \alpha_{1,2}S\\
            \alpha_{2,1}S & \alpha_{2,2}I_2
        \end{smallmatrix}}~\!\Big|~\!\alpha_{i,j}\in\qp,~\det\neq 0}.$$
    \begin{enumerate}
        \item If $ab\neq 0$, we have $H_D^\circ\simeq\g_a^2\rtimes_g\g_m^2$ with the canonical action $g(s,t)=\br{\begin{smallmatrix}
            s & 0\\
            0 & t
        \end{smallmatrix}}$, which is non-reductive; 
        \item If $ab=0$, we have $a=b=0$ and $H_D^\circ\simeq\mathbf G_m^2$.
    \end{enumerate}
    \end{enumerate}
\end{prop}
\begin{proof}
    Denote $S:=\br{\begin{smallmatrix}
            -b & a\\
            \frac{a+bc}{p^2} & b
        \end{smallmatrix}}$, $M:=\br{\begin{smallmatrix}
            0 & -p^2\\
            1 & c
        \end{smallmatrix}}\in\mat_{2\times 2}(\qp)$ as in the statement. The matrices $S,M\in\mat_{2\times 2}(\qp)$ satisfy relations 
        \begin{equation}\label{sm}
            S^2=-\frac{\epsilon p+c}{p^2}I_2,~M^2-cM+p^2I_2=0\qaq MS+SM=cS.
        \end{equation}
    Denote $q=-\frac{\epsilon p+c}{p^2}=S^2$, $E=\qp[M]\subset R=\mat_{2\times 2}(\qp)$ and $t\mapsto\bar t$ the involution of $E$ induced by $M\mapsto c-M$. By (\ref{sm}) we have 
    \begin{equation}\label{conjcomm}
        St=\bar tS,~\forall t\in E.
    \end{equation}
    As in the statement, denote by 
        $$G:=\st{\left(\begin{smallmatrix}
            \alpha_{1,1}I_2+\alpha_{1,1}'M & \alpha_{1,2}S+\alpha_{1,2}'MS\\
            \alpha_{2,1}S+\alpha_{2,1}'MS & \alpha_{2,2}I_2+\alpha_{2,2}'M
        \end{smallmatrix}\right)~\!\Big|~\!\begin{smallmatrix}
            \alpha_{i,j},\alpha_{i,j}'\in\qp,~\det\not=0\\
            \alpha_{1,1}\alpha_{2,2}'-\alpha_{2,2}\alpha_{1,1}'=-\frac{\epsilon p+c}{p^2}(\alpha_{1,2}\alpha_{2,1}'-\alpha_{2,1}\alpha_{1,2}')
        \end{smallmatrix}}$$
        and
        $$G':=\st{\br{\begin{smallmatrix}
            \alpha_{1,1}I_2 & \alpha_{1,2}S\\
            \alpha_{2,1}S & \alpha_{2,2}I_2
        \end{smallmatrix}}~\!\Big|~\!\alpha_{i,j}\in\qp,~\det\neq 0}$$
        the subsets in $\mat_{2\times 2}(R)$
        
    \noindent\textbf{Verify that $G$ and $G'$ are subgroups.}     
    For $x=x_0+x_1 M\in E$, denote $\alpha_E(x)=x_0$ and $\beta_E(x)=x_1$. For $u,v,w,t\in E$, denote
    $$X(u,v,w,t)=\br{\begin{smallmatrix}
        u & vS\\
        wS & t
    \end{smallmatrix}}\in\mat_{2\times 2}(R).$$
    Let $P=\br{\begin{smallmatrix}
        1 & 0\\
        0 & S^{-1}
    \end{smallmatrix}}\in\mat_{2\times 2}(R)$. Using (\ref{conjcomm}) we have 
    $$P^{-1}X(u,v,w,t)P=P^{-1}\br{\begin{smallmatrix}
        u & vS\\
        wS & t
    \end{smallmatrix}}P
    =\br{\begin{smallmatrix}
        u & v\\
        SwS & StS^{-1}
    \end{smallmatrix}}
    =\br{\begin{smallmatrix}
        u & v\\
        q\bar w & \bar t
    \end{smallmatrix}}.$$
    Therefore, we have $G\subset P\mat_{2\times 2}(E)P^{-1}$. Since $E$ is a field, the determinant map 
    $$\det_E:\mat_{2\times 2}(E)\to E$$ 
    is well-defined. Denote
    $$\Delta(u,v,w,t)=\det_{E}(P^{-1}X(u,v,w,t)P)\in E.$$
    We have 
    $$\Delta(u,v,w,t)=\det_{E}\br{\begin{smallmatrix}
        u & v\\
        q\bar w & \bar t
    \end{smallmatrix}}=u\bar t-qv\bar w$$
    and
    $$\beta_E(\Delta(u,v,w,t))=\alpha_E(t)\beta_E(u)-\alpha_E(u)\beta_E(t)-q\alpha_E(w)\beta_E(v)+q\alpha_E(v)\beta_E(w).$$
    Substituting $X(u,v,w,t)=\br{\begin{smallmatrix}
            \alpha_{1,1}I_2+\alpha_{1,1}'M & \alpha_{1,2}S+\alpha_{1,2}'MS\\
            \alpha_{2,1}S+\alpha_{2,1}'MS & \alpha_{2,2}I_2+\alpha_{2,2}'M
        \end{smallmatrix}}$, we have
    $$\beta_E(\Delta(u,v,w,t))=\alpha_{1,1}\alpha_{2,2}'-\alpha_{2,2}\alpha_{1,1}'+\frac{\epsilon p+c}{p^2}(\alpha_{1,2}\alpha_{2,1}'-\alpha_{2,1}\alpha_{1,2}')$$
    and it follows that 
    $$G=\st{A\in\mat_{2\times 2}(R)\mid \det_E(P^{-1}AP)\in\qp^\times}.$$
    Since $A\mapsto \det_E(P^{-1}AP)$ is a group homomorphism, $G$ is a subgroup.

    It remains to show that $G'$ is a subgroup when $q=0$. Note that
    $$\mathcal A:=\st{\br{\begin{smallmatrix}
        u & vS\\
        wS & t
    \end{smallmatrix}}\big|u,v,w,t\in\qp}$$
    is a $\qp$-subalgebra of $\mat_{2\times 2}(R)$. It follows that $G'=\mathcal{A}\cap\GL_4(\qp)=\mathcal{A}^\times$ is a subgroup.
    
    Differentiating the defining equations of $G$ and $G'$ at the identity, we have
        $$\lie(G)=\st{\br{\begin{smallmatrix}
            \alpha_{1,1}I_2+\alpha_{1,1}'M & \alpha_{1,2}S+\alpha_{1,2}'MS\\
            \alpha_{2,1}S+\alpha_{2,1}'MS & \alpha_{2,2}I_2+\alpha_{2,2}'M
        \end{smallmatrix}}~\!\Big|~\!
            \alpha_{i,j},\alpha_{i,j}'\in\qp,~
            \alpha_{1,1}'=\alpha_{2,2}'}$$
        and
        $$\lie(G')=\st{\br{\begin{smallmatrix}
            \alpha_{1,1}I_2 & \alpha_{1,2}S\\
            \alpha_{2,1}S & \alpha_{2,2}I_2
        \end{smallmatrix}}~\!\Big|~\!\alpha_{i,j}\in\qp}.$$
    Under the canonical adapted basis $\mathcal B=(e_i)_{1\le i\le 4}$ of $D_{(a,b)}^{\epsilon,\mu}$, the matrices of $\varphi$ and the Hodge cocharacter are
    $$\mat_{\mathcal B}(\varphi)=\br{\begin{smallmatrix}
            0 & 0 & 0 & -p^2\\
            0 & 0 & 1 & -\frac{a^2+\epsilon p+b^2p^2}{ab+1}\\
            1 & 0 & a & -\frac{a^3+\epsilon ap-bp^2}{ab+1}\\
            0 & 1 & b & -a
        \end{smallmatrix}}=\br{\begin{smallmatrix}
            0 & M\\
            I_2 & SM 
        \end{smallmatrix}}\qaq\mu:t\mapsto\br{\begin{smallmatrix}
        t & 0 & 0 & 0\\
        0 & t & 0 & 0\\
        0 & 0 & 1 & 0\\
        0 & 0 & 0 & 1
    \end{smallmatrix}}=\br{\begin{smallmatrix}
            tI_2 & 0\\
            0 & I_2 
        \end{smallmatrix}}.$$
    Note that $\br{\begin{smallmatrix}
            0 & M\\
            I_2 & SM 
        \end{smallmatrix}}^{-1}=\br{\begin{smallmatrix}
            -S & I_2\\
            M^{-1} & 0
        \end{smallmatrix}}$. We will verify the two assumptions of Lemma \ref{monogpcheck} case by case.
    
    \noindent\textbf{Verification of Lemma \ref{monogpcheck}(1).} 
    
    \noindent\textbf{Case 1.} If $c\neq-\epsilon p$, let $u,v,w,t\in E$ and we have
    $$\mat_{\mathcal B}(\varphi)X(u,v,w,t)\mat_{\mathcal B}(\varphi)^{-1}=\br{\begin{smallmatrix}
            0 & M\\
            I_2 & SM 
        \end{smallmatrix}}\br{\begin{smallmatrix}
        u & vS\\
        wS & t
    \end{smallmatrix}}\br{\begin{smallmatrix}
            -S & I_2\\
            M^{-1} & 0
        \end{smallmatrix}}=\br{\begin{smallmatrix}
        t-qwM & wMS\\
        (\bar t-u+v\bar M^{-1}-q\bar w\bar M)S & u+q\bar w\bar M
    \end{smallmatrix}}=X(u',v',w',t')$$
    where 
    $$u'=t-qwM,~v'=wM,~w'=\bar t-u+v\bar M^{-1}-q\bar w\bar M,~t'=u+q\bar w\bar M.$$
    It follows that
    \begin{align*}
        \Delta(u',v',w',t')&=u'\overline{t'}-qv'\overline{w'}\\
        &=(t-qwM)(\bar u+qwM)-qwM(t-\bar u+\bar vM^{-1}-qwM)\\
        &=t\bar u-qw\bar v=\overline{\Delta(u,v,w,t)}.
    \end{align*}
    Therefore, for $A\in\mat_{2\times 2}(R)$, conjugation by $\mat_{\mb}(\varphi)$ preserves the condition $\det_E(P^{-1}AP)\in\qp^\times$ and hence preserves $G$.

    \noindent\textbf{Case 2.} If $c=-\epsilon p$ and $ab\neq 0$, we have $q=0$. Note that $a+bc=-\frac{c+\epsilon p+p^2b^2}{a}=-\frac{p^2b^2}{a}$ and hence
    \begin{equation}\label{mss}
        MS=\br{\begin{smallmatrix}
        \frac{b^2p^2}{a}&-bp^2\\
        \frac{b^3p^2}{a^2}&-\frac{b^2p^2}{a}
    \end{smallmatrix}}=-\frac{p^2b}{a}S.
    \end{equation}
    For $u,v,w,t\in \qp$, we deduce from the same computation above together that
    $$\mat_{\mathcal B}(\varphi)X(u,v,w,t)\mat_{\mathcal B}(\varphi)^{-1}=\br{\begin{smallmatrix}
        t & wMS\\
        (\bar t-u+v\bar M^{-1})S & u
    \end{smallmatrix}}.$$
    By (\ref{mss}), $ES=\qp S\subset R$. Therefore, $(\bar t-u+v\bar M^{-1})S,wMS\in\qp S$ and conjugation by $\mat_{\mb}(\varphi)$ preserves $G'$.

    \noindent\textbf{Case 3.} If $c=-\epsilon p$ and $ab=0$, we deduce from $c=-\frac{a^2+\epsilon p+b^2p^2}{ab+1}$ that $a^2+b^2p^2=0$. It follows that $a=b=0$ and $S=0$. We have
    $$\mat_{\mathcal B}(\varphi)=\br{\begin{smallmatrix}
        0 & M\\
        1 & 0
    \end{smallmatrix}}\qaq G'=\st{\br{\begin{smallmatrix}
            uI_2 & 0\\
            0 & tI_2
        \end{smallmatrix}}~\!\Big|~\!u,t\in\qp^\times}.$$
    We deduce that $\mat_{\mathcal B}(\varphi)G'\mat_{\mathcal B}(\varphi)^{-1}=G'$.

        \noindent\textbf{Verification of Lemma \ref{monogpcheck}(2).} 
        
        \noindent\textbf{Case 1.} If $c\neq-\epsilon p$, we have
        $$
        \mat_{\mathcal B}(\varphi)^i\D\mu(1)\mat_{\mathcal B}(\varphi)^{-i}=
        \begin{cases}
            \br{\begin{smallmatrix}
            0 & 0\\
            -S & I_2
        \end{smallmatrix}}\quad &i=1\\
        \br{\begin{smallmatrix}
            I_2-\frac{\epsilon p+c}{p^2}M & -MS\\
            -\frac{\epsilon p+c}{p^2}SM+S & -\frac{\epsilon p+c}{p^2}(M-cI_2)
        \end{smallmatrix}}&i=2\\
        \br{\begin{smallmatrix}
            -\frac{\epsilon p+c}{p}I_2 & -(\epsilon p+c)S+MS\\
            -S-\frac{1}{p}SM& \br{1+\frac{\epsilon p+c}{p}}I_2
        \end{smallmatrix}}&i=3.
        \end{cases}
        $$
        Using Lemma \ref{Liegene}, the Lie algebra generated by the Lie algebras of the tori $\varphi^i\mu(\g_m)\varphi^{-i}$ contains
        $$\br{\begin{smallmatrix}
            I_2 & 0\\
            0 & 0
        \end{smallmatrix}},~\br{\begin{smallmatrix}
            0 & 0\\
            0 & I_2
        \end{smallmatrix}},~\br{\begin{smallmatrix}
            0 & 0\\
            -S & 0
        \end{smallmatrix}},~\br{\begin{smallmatrix}
            I_2-\frac{\epsilon p+c}{p^2}M & 0\\
            0 & -\frac{\epsilon p+c}{p^2}(M-cI_2)
        \end{smallmatrix}},~\br{\begin{smallmatrix}
            0 & -MS\\
            0 & 0
        \end{smallmatrix}},~\br{\begin{smallmatrix}
            0 & 0\\
            -\frac{\epsilon p+c}{p^2}SM+S & 0
        \end{smallmatrix}},~\br{\begin{smallmatrix}
            0 & -(\epsilon p+c)S+MS\\
            0 & 0
        \end{smallmatrix}}$$
        which span $\lie(G)$ over $\qp$. Therefore, We have $H_D^\circ=G$ by Lemma \ref{Liealg}.
        
        It remains to prove that $G\simeq_{\bqp}\GL_{2}\times_{\det}\GL_{2}$. Note that over $\bqp$, the matrix $M$ is conjugated to $\widetilde M=\br{\begin{smallmatrix}
            \mu_1 & 0\\
            0 & \mu_2
        \end{smallmatrix}}$ with $\mu_1\neq\mu_2\in\bar\mq_p^\times$. Indeed, $M$ is of rank $2$ with characteristic polynomial $X^2-cX+p^2$. By Jordan decomposition, if $M$ is not of this form, we have $c=\pm 2p$. It follows that $-\frac{a+\epsilon p+b^2p^2}{ab+1}=\pm 2p$ and $(a\pm bp)^2=-(\epsilon\pm 2)p\in\st{\pm p,\pm 3p}$, impossible when $p\ge 5$. On the other hand, using the relations $S^2=-\frac{\epsilon p+c}{p^2}I_2$ and $MS+SM=cS$, we can show that the same conjugation makes $S$ conjugate to $\widetilde S=\br{\begin{smallmatrix}
            0 & \lambda_2\\
            \lambda_1 & 0
        \end{smallmatrix}}$ with $\lambda_1,\lambda_2\in\bqp^\times$. It follows that $G$ is conjugated to the group
        $$\st{\left(\begin{smallmatrix}
            \alpha_{1,1}I_2+\alpha_{1,1}'\widetilde M & \alpha_{1,2}\widetilde S+\alpha_{1,2}'\widetilde M\widetilde S\\
            \alpha_{2,1}\widetilde S+\alpha_{2,1}'\widetilde M\widetilde S & \alpha_{2,2}I_2+\alpha_{2,2}'\widetilde M
        \end{smallmatrix}\right)~\!\Big|~\!\begin{smallmatrix}
            \alpha_{i,j},\alpha_{i,j}'\in\qp,~\det\not=0\\
            \alpha_{1,1}\alpha_{2,2}'-\alpha_{2,2}\alpha_{1,1}'=\\
            -\frac{\epsilon p+c}{p^2}(\alpha_{1,2}\alpha_{2,1}'-\alpha_{2,1}\alpha_{1,2}')
        \end{smallmatrix}}=\st{\left(\begin{smallmatrix}
        \beta_{1,1} & 0 & 0 & \beta_{1,2}\\
        0 & \beta_{2,2}' & \beta_{2,1}' & 0\\
        0 & \beta_{1,2}' & \beta_{1,1}' & 0\\
        \beta_{2,1} & 0 & 0 & \beta_{2,2}\end{smallmatrix}\right)~\!\Big|~\!\begin{smallmatrix}
            \beta_{i,j},\beta_{i,j}'\in\bqp,~\det\neq 0\\ \beta_{1,1}\beta{2,2}-\beta{1,2}\beta{2,1}=\\
            \beta_{1,1}'\beta{2,2}'-\beta{1,2}'\beta{2,1}'
        \end{smallmatrix}}$$
        which, after switching the second and the fourth vectors in the basis, becomes
        $$\st{\left(\begin{smallmatrix}
        \beta_{1,1} & \beta_{1,2} & 0 & 0\\
        \beta_{2,1} & \beta{2,2} & 0 & 0\\
        0 & 0 & \beta_{1,1}' & \beta_{1,2}'\\
        0 & 0 & \beta_{2,1}' & \beta_{2,2}'\end{smallmatrix}\right)~\!\Big|~\!\begin{smallmatrix}
            \beta_{i,j},\beta_{i,j}'\in\bqp,~\det\neq 0\\  \beta_{1,1}\beta{2,2}-\beta{1,2}\beta{2,1}=\\
            \beta_{1,1}'\beta{2,2}'-\beta{1,2}'\beta{2,1}'
        \end{smallmatrix}}=\GL_{2}\times_{\det}\GL_{2}.$$

    \noindent\textbf{Case 2.} If $c=-\epsilon p$ and $ab\neq 0$, we have
    $$\mat_{\mathcal B}(\varphi)\D\mu(1)\mat_{\mathcal B}(\varphi)^{-1}=\br{\begin{smallmatrix}
            0 & M\\
            I_2 & SM
        \end{smallmatrix}}\br{\begin{smallmatrix}
            I_2 & 0\\
            0 & 0
        \end{smallmatrix}}\br{\begin{smallmatrix}
            -S & I_2\\
            M^{-1} & 0
        \end{smallmatrix}}=\br{\begin{smallmatrix}
            0 & 0\\
            -S & I_2
        \end{smallmatrix}}
    $$
    and
    $$
        \mat_{\mathcal B}(\varphi)^2\D\mu(1)\mat_{\mathcal B}(\varphi)^{-2}=\br{\begin{smallmatrix}
            0 & M\\
            I_2 & SM
        \end{smallmatrix}}^2\br{\begin{smallmatrix}
            I_2 & 0\\
            0 & 0
        \end{smallmatrix}}\br{\begin{smallmatrix}
            -S & I_2\\
            M^{-1} & 0
        \end{smallmatrix}}^2=\br{\begin{smallmatrix}
            I_2 & -MS\\
            S & 0
        \end{smallmatrix}}.$$
 
    By Lemma \ref{Liegene} and (\ref{mss}), the Lie algebra generated by the Lie algebras of the tori $\varphi^i\mu(\g_m)\varphi^{-i}$ contains
    $$\br{\begin{smallmatrix}
            I_2 & 0\\
            0 & 0
        \end{smallmatrix}},~\br{\begin{smallmatrix}
            0 & 0\\
            0 & I_2
        \end{smallmatrix}},~\br{\begin{smallmatrix}
            0 & 0\\
            -S & 0
        \end{smallmatrix}}\qaq\br{\begin{smallmatrix}
            0 & \frac{p^2b}{a}S\\
            0 & 0
        \end{smallmatrix}}.$$
    In this case, we have $G'\simeq\g_{a}^2\rtimes_g\g_{m}^2$ and the subgroup $\st{\br{\begin{smallmatrix}
            1 & \alpha S\\
            0 & 1
        \end{smallmatrix}}~\!\big|~\!\alpha\in\qp}$ of $G'$ is nontrivial, normal and unipotent.
        
    \noindent\textbf{Case 3.} If $c=-\epsilon p$ and $ab\neq 0$, as above we have $a=b=0$. It follows that $S=0$. The subgroup $G'$ becomes $\st{\br{\begin{smallmatrix}
            \alpha I_2 & 0\\
            0 & \alpha'I_2
        \end{smallmatrix}}~\!\big|~\alpha\alpha'\neq 0}\simeq\g_{m,\qp}^2$.
\end{proof}

\begin{example}\label{majaexp2}
    Volkov \cite[\S3]{volkov2023abelian} gave an example of a nonsemisimple filtered $\varphi$-module arising from supersingular abelian surface: let $p\equiv 1~(\mathrm{mod}~3)$ be a prime and $\zeta_3\in\qp$ be a primitive third root of unity, consider a $4$-dimensional $\qp$-vector space $D$ with a $\qp$-basis $\mathcal B=(x_1,y_1,x_2,y_2)$ equipped with a filtered $\varphi$-module structure given by
    $$\mat_{\mathcal B}(\varphi)=\br{\begin{smallmatrix}
        0 & \zeta_3p & 0 & 0 \\
        1 & 0 & 0 & 0\\
        0 & 0 & 0 &\zeta_3^{-1}p\\
        0 & 0 & 1 & 0
    \end{smallmatrix}}\qaq\fil_0D=D,~\fil_1D=\qp x_1\oplus\qp(y_1+x_2),~\fil_2D=0.$$
    We can show that $D\simeq D^{1,\mu}_{(-\zeta_3^{-1}p,1)}$ and in this case, $c:=-\frac{a^2+\epsilon p+b^2p^2}{ab+1}=-\epsilon p$. Theorem \ref{sit2} implies that $H_D^\circ\simeq\g_a^2\rtimes_g\g_m^2$ which is non-reductive, which reproves that $D$ is non-semisimple \cite[Proposition 3.1(b)]{volkov2023abelian}.
\end{example}

\begin{prop}\label{sit3}
    Denote $D=D_{\epsilon'}^{\epsilon,iso}$, we have 
    $$H_D^\circ=\st{\br{\begin{smallmatrix}
            \alpha_{1,1}I_2 & \alpha_{1,2}S\\
            \alpha_{2,1}S & \alpha_{2,2}I_2
        \end{smallmatrix}}~\!\Big|~\!\alpha_{i,j}\in\qp,~\det\not=0}~\mbox{where}~S=\br{\begin{smallmatrix}
            0 & -(\epsilon+2\epsilon')p\\
            1 & 0
        \end{smallmatrix}}.$$
    In this case, we have $H^\circ_{D_{\epsilon'}^{\epsilon,iso}}\simeq_{\bqp}\GL_{2}$.
\end{prop}
\begin{proof}
    Under the canonical adapted basis $\mathcal B=(e_i)_{1\le i\le 4}$ of $D_{\epsilon'}^{\epsilon,iso}$, the matrices of $\varphi$ and the Hodge cocharacter are
    $$\mat_{\mathcal B}(\varphi)=\br{\begin{smallmatrix}
            0 & 0 & \epsilon'p & 0\\
            0 & 0 & 0 & \epsilon'p\\
            1 & 0 & 0 & -(\epsilon+2\epsilon')p\\
            0 & 1 & 1 & 0
        \end{smallmatrix}}=\br{\begin{smallmatrix}
            0 & \epsilon'pI_2\\
            I_2 & S 
        \end{smallmatrix}}\qaq\mu:t\mapsto\br{\begin{smallmatrix}
        t & 0 & 0 & 0\\
        0 & t & 0 & 0\\
        0 & 0 & 1 & 0\\
        0 & 0 & 0 & 1
    \end{smallmatrix}}.$$
    As in the statement, denote $S:=\br{\begin{smallmatrix}
            0 & -(\epsilon+2\epsilon')p\\
            1 & 0
     \end{smallmatrix}}$ and
     $$G:=\st{\br{\begin{smallmatrix}
            \alpha_{1,1} & \alpha_{1,2}S\\
            \alpha_{2,1}S & \alpha_{2,2}
        \end{smallmatrix}}~\!\Big|~\!\alpha_{i,j}\in\qp,~\det\not=0}.$$
     Denote $q=-(\epsilon+2\epsilon')p$. We have $S^2=qI_2$. It follows that
     $$\mathcal A:=\st{\br{\begin{smallmatrix}
            \alpha_{1,1} & \alpha_{1,2}S\\
            \alpha_{2,1}S & \alpha_{2,2}
        \end{smallmatrix}}~\!\Big|~\!\alpha_{i,j}\in\qp}$$
    is a $\qp$-subalgebra of $\mat_{4\times 4}(\qp)$ and that $G=\mathcal A\cap\GL_4(\qp)=\mathcal A^\times$ is a group.
     %By Proposition \ref{zardense}, $H_D^\circ$ is generated by the family of tori
        % $$\mathcal F=\st{\mat_{\mathcal B}(\varphi)^i\circ\mu(t)\circ\mat_{\mathcal B}(\varphi)^{-i}}_{i\in\mz}=\st{\br{\begin{smallmatrix}
        %     0 & \epsilon' p\\
        %     I_2 & S
        % \end{smallmatrix}}^i\br{\begin{smallmatrix}
        %     tI_2 & 0\\
        %     0 & I_2
        % \end{smallmatrix}}\br{\begin{smallmatrix}
        %     0 & \epsilon' p\\
        %     I_2 & S
        % \end{smallmatrix}}^{-i}}_{i\in\mz}.$$
        % Using Lemma \ref{Liealg}, it suffices to show that $\st{\br{\begin{smallmatrix}
        %     0 & \epsilon' p\\
        %     I_2 & S
        % \end{smallmatrix}}^i\br{\begin{smallmatrix}
        %     I_2 & 0\\
        %     0 & 0
        % \end{smallmatrix}}\br{\begin{smallmatrix}
        %     0 & \epsilon' p\\
        %     I_2 & S
        % \end{smallmatrix}}^{-i}}_{i\in\mz}$ generates $\st{\br{\begin{smallmatrix}
        %     \alpha_{1,1} & \alpha_{1,2}S\\
        %     \alpha_{2,1}S & \alpha_{2,2}
        % \end{smallmatrix}}}_{\alpha_{i,j}\in\qp}$ as a Lie subalgebra of $\gl_4(\qp)$.
    It remains to verify the two assumptions of Lemma \ref{monogpcheck}.
    
    \noindent\textbf{Verification of Lemma \ref{monogpcheck}(1).} Let $\alpha_{i,j}\in\qp$. We have
    $$\br{\begin{smallmatrix}
            0 & \epsilon' pI_2\\
            I_2 & S
        \end{smallmatrix}}\br{\begin{smallmatrix}
            \alpha_{1,1}I_2 & \alpha_{1,2}S\\
            \alpha_{2,1}S & \alpha_{2,2}I_2
        \end{smallmatrix}}\br{\begin{smallmatrix}
            0 & \epsilon' pI_2\\
            I_2 & S
        \end{smallmatrix}}^{-1}=\br{\begin{smallmatrix}
            (-\epsilon'p\alpha_{2,2}+\epsilon'pqa_{2,1})I_2 & -\alpha_{2,1}p^2S\\
            (\alpha_{1,1}-\alpha_{1,2}-\alpha_{2,2}+q\alpha_{2,1})S & -\epsilon'p(\alpha_{1,1}+q\alpha_{2,1})I_2
        \end{smallmatrix}}\in\mathcal A.$$
        Using $G=\mathcal A\cap\GL_4(\qp)$, we deduce that $\mat_{\mathcal B}(\varphi)G\mat_{\mathcal B}(\varphi)^{-1}=G$.
    
    \noindent\textbf{Verification of Lemma \ref{monogpcheck}(2).} We have
        $$
            \br{\begin{smallmatrix}
            0 & \epsilon' pI_2\\
            I_2 & S
        \end{smallmatrix}}\br{\begin{smallmatrix}
            I_2 & 0\\
            0 & 0
        \end{smallmatrix}}\br{\begin{smallmatrix}
            0 & \epsilon' pI_2\\
            I_2 & S
        \end{smallmatrix}}^{-1}=\br{\begin{smallmatrix}
            0 & 0\\
            -\frac{\epsilon'}{p}S & I_2
        \end{smallmatrix}}
        \mbox{ and }
        \br{\begin{smallmatrix}
            0 & \epsilon' pI_2\\
            I_2 & S
        \end{smallmatrix}}^2\br{\begin{smallmatrix}
            I_2 & 0\\
            0 & 0
        \end{smallmatrix}}\br{\begin{smallmatrix}
            0 & \epsilon' pI_2\\
            I_2 & S
        \end{smallmatrix}}^{-2}=\br{\begin{smallmatrix}
            -\frac{\epsilon+\epsilon'}{\epsilon'}I_2 & -S\\
            -\frac{\epsilon+\epsilon'}{p}S & \frac{\epsilon+2\epsilon'}{\epsilon'}I_2
        \end{smallmatrix}}.$$
    By Lemma \ref{Liegene}, the Lie algebra generated by the Lie algebras of the tori $\varphi^i\mu(\g_m)\varphi^{-i}$ contains 
    $$\br{\begin{smallmatrix}
            I_2 & 0\\
            0 & 0
        \end{smallmatrix}},~\br{\begin{smallmatrix}
            0 & 0\\
            0 & I_2
        \end{smallmatrix}},~\br{\begin{smallmatrix}
            0 & 0\\
            -\frac{\epsilon'}{p}S & 0
        \end{smallmatrix}},~\br{\begin{smallmatrix}
            0 & -S\\
            0 & 0
        \end{smallmatrix}}.$$
    These matrices form a $\qp$-basis of $\lie(G)$.
    
    It remains to prove that $G$ is isomorphic to $\GL_2$ over $\bqp$. Note that the matrix $S$ is conjugated to $\br{\begin{smallmatrix}
            \sigma& 0\\
            0 & -\sigma
        \end{smallmatrix}}$ for some $\sigma\in\bar\mq_p^\times$. It follows that
    $$H_{D,\bar\mq_p}^\circ\simeq\st{\br{\begin{smallmatrix}
            \alpha_{1,1} & 0 & \alpha_{1,2}\sigma & 0\\
            0 & \alpha_{1,1} & 0 & -\alpha_{1,2}\sigma\\
            \alpha_{2,1}\sigma & 0 & \alpha_{2,2} & 0\\
            0 & -\alpha_{2,1}\sigma & 0 & \alpha_{2,2}
        \end{smallmatrix}}~\!\Big|~\!\alpha_{i,j}\in\qp,~\det\not=0}.$$
    The map
    $$\br{\begin{smallmatrix}
            \alpha_{1,1} & 0 & \alpha_{1,2}\sigma & 0\\
            0 & \alpha_{1,1} & 0 & -\alpha_{1,2}\sigma\\
            \alpha_{2,1}\sigma & 0 & \alpha_{2,2} & 0\\
            0 & -\alpha_{2,1}\sigma & 0 & \alpha_{2,2}
        \end{smallmatrix}}\mapsto\br{\begin{smallmatrix}
            \alpha_{1,1} & \alpha_{1,2}\sigma\\
            \alpha_{2,1}\sigma & \alpha_{2,2}
        \end{smallmatrix}}$$
    defines an isomorphism between $H_{D,\bar\mq_p}^\circ$ and $\GL_{2,\bar\mq_p}$.
\end{proof}
Combining all the cases above, we have the following conclusion.
\begin{thm}\label{mainthm}
    Fix a prime $p\ge 7$, we have the following isomorphisms over $\bqp$.
    \begin{enumerate}
        \item $H_{D^{prod}_{\epsilon'}}^\circ\simeq\g_{m}^2$,
        \item $H_{D^{\epsilon,\nu}_{a'}}^\circ\simeq\begin{cases}
        \g_m^3 &\mbox{if }a'=\epsilon=0\\
        \GL_2\times_{\det}\GL_2&\mbox{otherwise},
    \end{cases}$
        \item $H_{D^{\epsilon,\mu}_{(a,b)}}^\circ\simeq\begin{cases}
        \g_m^2&\mbox{if } a=b=0\\
        \g_a^2\rtimes_g\g_m^2 &\mbox{if }c(a,b)=-\epsilon p,\mbox{ and }a\neq 0\mbox{ or }b\neq 0\\
        \GL_2\times_{\det}\GL_2&\mbox{otherwise},
    \end{cases}$
        \item $H_{D^{\epsilon,iso}_{\epsilon'}}^\circ\simeq\GL_2$
    \end{enumerate}
    
    where $c(a,b)=-\frac{a^2+\epsilon p+b^2p^2}{ab+1}$ and $g(s,t)=\br{\begin{smallmatrix}
            s & 0\\
            0 & t
        \end{smallmatrix}}$. In particular, the neutral connected component $G_{V_p(A),\bar{\mathbb Q}_p}^\circ$ of the $p$-adic monodromy group associated with a supersingular abelian surface $A$ over $\qp$ is isomorphic to one of the following.
    $$\mathbf G_{m}^2,~\mathbf G_{m}^3,~\mathbf G_{a}^2\rtimes_g\mathbf G_{m}^2,~\GL_{2},~\GL_{2}\times_{\det}\GL_{2}.$$
\end{thm}

\begin{cor}\label{semisimplicity}
    Let $D\in\mfs$. Then $D$ is non-semisimple if and only if $D\simeq D^{\epsilon,\mu}_{(a,b)}$ for some $\epsilon\in\st{0,\pm1}$, $ab\neq-1$ and $(0,0)\neq(a,b)\in\qp^2$ satisfying $c(a,b)=-\epsilon p$. Equivalently,
    $$\epsilon=0,p\equiv 1~\mathrm{mod}~4\qaq D\simeq D^{0,\mu}_{(ipb,b)}$$
    for $b\in\qp^\times$ and some primitive fourth root of unity $i\in \qp$, or
    $$\epsilon=\pm1,p\equiv 1~\mathrm{mod}~3\qaq D\simeq D^{\epsilon,\mu}_{(-\epsilon\omega pb,b)}$$
    for $b\in\qp^\times$ and some primitive third root of unity $\omega\in\qp$.
\end{cor}
\begin{proof}
    By Tannakian formalism, the object $D$ corresponds to the faithful standard representation of its monodromy group $H_D$. Since we work in characteristic zero, this representation is semisimple if and only if $H_D^\circ$ is reductive (see \cite[Proposition 2.23, Remark 2.28]{deligne2012tannakian}). 
    
    Let $D\in\mfs$. By Theorem \ref{mainthm}, $H_D^\circ$ is non-reductive if and only if $D\simeq D^{\epsilon,\mu}_{(a,b)}$ for some $\epsilon\in\st{0,\pm1}$, $ab\neq-1$ and $(0,0)\neq(a,b)\in\qp^2$ satisfying $c(a,b)=-\epsilon p$. Equivalently, we have
    $$a^2-\epsilon abp+b^2p^2=0.$$
    Hence $ab=0$ implies $a=b=0$ and we can assume that $ab\neq 0$. There are two cases:
    
    If $\epsilon=0$, $t:=\frac{a}{bp}$ satisfies $t^2=-1$. So we have $p\equiv 1 ~\mathrm{mod}~4$ and $t$ is a primitive fourth root of unity. 
    
    If $\epsilon=\pm1$, $s:=-\frac{a}{\epsilon pb}$ satisfies $s^2+s+1=0$. So we have $p\equiv 1~\mathrm{mod}~3$ and $s$ is a primitive third root of unity.
\end{proof}
\begin{rem}\label{normalization}
    In the statement above, the parameter $b\in\mathbb Q_p^\times$ can be removed after passing to isomorphism classes. Indeed, using Theorem \ref{propmain}, for any $b\in\qp^\times$, we have
    $$D^{0,\mu}_{(ipb,b)}\simeq D^{0,\mu}_{(ip,1)}\text{ and for }\epsilon=\pm1,~ D^{\epsilon,\mu}_{(-\epsilon\omega pb,b)}\simeq D^{\epsilon,\mu}_{(-\epsilon\omega p,1)}.$$
\end{rem}

\section{A moduli space description}\label{moduli}
Fix a prime $p\ge 7$ and a parameter $\epsilon\in\st{0,\pm 1}$. In this section, we study the set $\mwa$ of isomorphism classes of objects in $\mfs$ with Weil polynomial $\chi_{\varphi}(X)=X^4+\epsilon pX^2+p^2$. By Proposition \ref{automat}, the elements of $\mwa$ are precisely admissible filtered $\varphi$-modules over $\qp$ satisfying \hyperref[ssfilphimod]{$\mathbf{S1}$} with $\chi_{\varphi}(X)=X^4+\epsilon pX^2+p^2$. 

We first treat $\mwa$ as a set and study the corresponding monodromy groups. We then construct a surjective Zariski quotient presentation $\pi:\mbp^2(\qp)\twoheadrightarrow\mwa$, whose fibers are Zariski locally closed subsets of $\mbp^2(\qp)$. We endow $\mwa$ with the quotient topology induced by the Zariski topology via $\pi$. The resulting topology records specialization relations among the isomorphism classes and is used below to study the loci with fixed monodromy group.

A main tool to describe $\mwa$ is the Grassmannian $Y=\Gr(2,4)$. Indeed, for any such $D$, we can take a certain $\qp$-basis $\mb$ of $D$ such that 
\begin{equation}\label{matrixofphi}
    \mat_{\mb}(\varphi)=\br{\begin{smallmatrix}
    0 & 0 & 0 & -p^2\\
    1 & 0 & 0 & 0\\
    0 & 1 & 0 & -\epsilon p\\
    0 & 0 & 1 & 0
\end{smallmatrix}}.
\end{equation}
Since $D$ satisfies \hyperref[ssfilphimod]{$\mathbf{S1}$}, the only undetermined data is the choice of $\fil_1D\subset\qp^4$, described by $Y(\qp)$.
\begin{defn}\label{dp}
    We denote by $D_P$ the filtered $\varphi$-module represented by $P\in Y(\qp)$, and by $\adlc(Y(\qp))$ the subset of $Y(\qp)$ consisting of the points representing admissible filtered $\varphi$-modules. By Lemma \ref{adm}, $\adlc(Y(\qp))$ is the complement of the $\varphi$-stable points in $Y(\qp)$.
\end{defn}
Note that $\GL_4$ acts algebraically on $Y$, under which $\adlc(Y(\qp))$ is stable. After identifying $\varphi$ with the matrix (\ref{matrixofphi}), we have the following observation.
\begin{prop}
    Denote $Z_{\varphi}(\qp)=\st{g\in\GL_4(\qp)|g\varphi=\varphi g}$. The quotient set $\adlc(Y(\qp))/Z_{\varphi}(\qp)$ is identified with $\mwa$ via $[P]\mapsto D_P$. 
\end{prop}

\begin{nota}
    Let $K$ be any field and $P\in Y(K)$. Denote by $V_P\subset K^4$ the linear subspace defined by $P$. Denote by $U(K)\subset Y(K)$ the open subvariety of $Y$ defined by the condition 
    $$K^4=V_P\oplus\varphi(V_P).$$ For $P\in U(K)$, we denote by $\pr_P$ the projection from $K^4$ onto $V_P$ along $\varphi(V_P)$. 
\end{nota}
\begin{defn}\label{func}
    Fix the canonical embedding $i:\mba^1\to\mbp^1$, $a\mapsto[a:1]$. We denote by $c$ the morphism from $U$ to $\mbp^1$ defined by $P\mapsto i(\tr(\pr_P\circ\varphi^2|_{V_P}))$. We denote by $\dm_c$ the domain of definition of $c$ (the maximal locus where $c$ can be extended), which is an open subvariety of $Y$.
\end{defn}

\subsection{Parametrization by $\mbp^2(\qp)$}
Let $K$ be any field, we can decompose $K^4$ as
$$K^4=Ke_1\oplus\bigoplus_{2\le i\le 4}Ke_i,$$ 
where $(e_i)_{1\le i\le 4}$ is the canonical basis of $K^4$. By considering only the $2$-dimensional subspaces of $K^4$ containing $e_1$, we can define a closed immersion 
$$\iota:\mbp^2=\Gr(1,3)\hookrightarrow\Gr(2,4)=Y.$$
Moreover, we have $\iota(\mbp^2(\qp))\subset \adlc(Y(\qp))$ (since there is no $2$-dimensional $\varphi$-stable subspace of $\qp^4$ containing $e_1$).
\begin{prop}\label{surjpi}
    The composition $\pi:\mbp^2(\qp)\stackrel{\iota}\to \adlc(Y(\qp))\twoheadrightarrow\mwa$ is surjective.
\end{prop}
\begin{proof}
    For any $P\in \adlc(Y(\qp))$, by Lemma \ref{existcyc} we can take $x,y\in\fil_1D_P$ such that $x$ is $\varphi$-cyclic and that $\fil_1D_P=\qp x\oplus\qp y$. If $y=\sum_{0\le j\le 3}c_j\varphi^j(x)$, let $Q=\iota([c_1:c_2:c_3])$ and we have an isomorphism of filtered $\varphi$-modules $D_Q\simeq D_P$ via taking the canonical basis of $V=\qp^4$ to the basis $\st{\varphi^j(x)}_{0\le j\le 3}$ of $D_P$. Therefore, we have $\pi(Q)=\pi(P)$ and $\pi$ is surjective.
\end{proof}
For convenience, for $P\in\mbp^2(\qp)$, we denote by $D_P$ the filtered $\varphi$-module $D_{\iota(P)}$.

\begin{prop}\label{munu}
    Consider the points $P_{\pm}=[(\epsilon\pm1)p:0:1]\in\mbp^2(\qp)$ and morphisms
    $$\mba^1\stackrel{\nu}\to\mbp^2\qaq\mba^2\backslash\st{xy=-1}\stackrel{\mu}\to\mbp^2$$
    defined by
    $$a'\stackrel{\nu}\mapsto[a':0:1]\qaq (a,b)\stackrel{\mu}\mapsto\sq{-\frac{a+\epsilon pab^2-b^3p^2}{ab+1}:1:-b}.$$
    For $\epsilon'\in\st{\pm1}$ and $a,b,a'\in\qp$ such that $ab\neq-1$, we have
    $$D_{P_{\epsilon'}}\simeq D_{\epsilon'}^{\epsilon,iso},~D_{\nu(a')}\simeq D_{a'}^{\epsilon,\nu}\qaq D_{\mu(a,b)}\simeq D_{(a,b)}^{\epsilon,\mu}.$$
\end{prop}
\begin{proof}
    Recall that for $P=[t_1:t_2:t_3]\in\mbp^2(\qp)$, the isomorphism class of $D_P$ is determined by the condition that
    \begin{equation}\label{co}
        \mbox{there exists } x\in D_P\mbox{ such that }\fil_1D_P=\Span\br{x,\textstyle{\sum_{i=1}^3t_i\varphi^i(x)}}.
    \end{equation}
    The canonical adapted basis of $D_{a'}^{\epsilon,\nu}$ has the form $(y,x,\varphi(x),\varphi^2(x))$ such that $\varphi^3(x)=y-a'\varphi(x)$. It follows that $\fil_1D_{a'}^{\epsilon,\nu}$ is generated by $x$ and $a'\varphi(x)+\varphi^3(x)$, and the assertion $D_{\nu(a')}\simeq D_{a'}^{\epsilon,\nu}$ follows from (\ref{co}). 
    
    Similarly, the canonical adapted basis of $D_{\epsilon'}^{\epsilon,iso}$ has the form $(x,y,\varphi(x),\varphi(y))$ such that
    $$\begin{cases}
        \varphi^2(x)=\epsilon'px+\varphi(y)\\
        \varphi^2(y)=\epsilon'py-(\epsilon+2\epsilon')p\varphi(x).
    \end{cases}$$
    It follows that 
    $$\varphi^3(x)=\epsilon'p\varphi(x)+\varphi^2(y)=\epsilon'py-(\epsilon+\epsilon')p\varphi(x).$$
    Therefore, $\fil_1D_{\epsilon'}^{\epsilon,iso}$ is generated by $x$ and $(\epsilon+\epsilon')p\varphi(x)+\varphi^3(x)$ and our assertion $D_{P_{\epsilon'}}\simeq D_{\epsilon'}^{\epsilon,iso}$ follows from (\ref{co}).

    Finally, the canonical adapted basis of $D_{(a,b)}^{\epsilon,\mu}$ has the form $(x,y,\varphi(x),\varphi(y))$ such that
    \begin{equation}\label{ast}
        \begin{cases}
        \varphi^2(x)=y+a\varphi(x)+b\varphi(y)\\
        \varphi^2(y)=p^2x+cy+d\varphi(x)-a\varphi(y)
    \end{cases}
    \end{equation}
    where $c=-\frac{a^2+\epsilon p+b^2p^2}{ab+1}$ and $d=ac+bp^2$. Since $\fil_1D_{(a,b)}^{\epsilon,\mu}=\Span(x,y)$, it suffices to compute the coordinates of $y$ under the basis $(\varphi^i(x))_{0\le i\le 3}$. If $b\neq0$, using (\ref{ast}) repeatedly, we have
    \begin{align*}
        \varphi^3(x)=\varphi(y)+a\varphi^2(x)+b\varphi^2(y)&=bp^2x+bcy+bd\varphi(x)+a\varphi^2(x)+(1-ab)\varphi(y)\\
        %&=\varphi(y)+a\varphi^2(x)+b(p^2x+cy+d\varphi(x)-a\varphi(y))\\
        %&=bp^2x+bcy+bd\varphi(x)+a\varphi^2(x)+(1-ab)\varphi(y)\\
        %&=bp^2x+bcy+bd\varphi(x)+a\varphi^2(x)+\frac{1-ab}{b}(\varphi^2(x)-y-a\varphi(x))\\
        &=bp^2x+(a+bc-b^{-1})y+(bd-ab^{-1}+a^2)\varphi(x)+b^{-1}\varphi^2(x).
    \end{align*}
    A direct computation shows that $bd-ab^{-1}+a^2=-\frac{a+\epsilon pab^2-b^3p^2}{ab+1}$. Therefore, we have
    $$(a+bc-b^{-1})y\equiv\frac{a+\epsilon pab^2-b^3p^2}{ab+1}\varphi(x)+\varphi^2(x)-b\varphi^3(x)\mod x.$$
    We deduce that $D_{\mu(a,b)}\simeq D_{(a,b)}^{\epsilon,\mu}$ from (\ref{co}). If $b=0$, by (\ref{ast}) we have $y=-a\varphi(x)+\varphi^2(x)$. Since $\mu(a,0)=[-a:1:0]$, we have $D_{\mu(a,0)}\simeq D_{(a,0)}^{\epsilon,\mu}$.
\end{proof}
The following lemma is technical, and we will prove it at the end of next subsection.
\begin{lem}\label{dfwa}
    We have $\adlc(Y(\qp))=\dm_c(\qp)$. In particular, the rational map $c$ is defined at every point of the admissible subset $\adlc(Y(\qp))\subseteq Y(\qp)$. 
\end{lem}
The first main theorem of this section is the following.
\begin{thm}\label{propmain}
    The map (of sets) $c:\adlc(Y(\qp))\rightarrow\mbp^1(\qp)$ factors through $\mwa$. Equivalently, there exists a unique map (of sets) $\bar c:\mwa\to\mbp^1(\qp)$ such that the diagram
    $$% https://tikzcd.yichuanshen.de/#N4Igdg9gJgpgziAXAbVABwnAlgFyxMJZABgBpiBdUkANwEMAbAVxiRAE0A9YAHR6gBmAXwAUfAI5oAlCCGl0mXPkIoAjOSq1GLNnwC2AIzSdVYnpJlyF2PASJlVm+s1aIQ+gO51ZmmFADm8ESgAgBOEHpIZCA4EEgATNQAFjB0UGw4HhApaQhWIGERUdSxSOpaLmwAxrLyBeGRiIkxcYjlDFhgriBQdHAp6dTOOm58BnShAAQ1QhRCQA
\begin{tikzcd}
\adlc(Y(\qp)) \arrow[d, two heads] \arrow[r, "c"] & \mbp^1(\qp) \\
\mwa \arrow[ru, "\bar c", dashed]                &            
\end{tikzcd}$$
commutes. Moreover, $\bar c$ is injective outside $\bar c^{-1}(i(-\epsilon p))$ (where $i$ is the embedding in Definition \ref{func}). 
\end{thm}  
\begin{proof}
By Proposition \ref{surjpi}, $\mwa$ is a quotient set of $\mbp^2(\qp)$ under $\pi$. Two points $P,P'\in\mbp^2(\qp)$ are in the same fiber of $\pi$ if and only if $D_{P}\simeq D_{P'}$ as filtered $\varphi$-modules.
    $$
% https://tikzcd.yichuanshen.de/#N4Igdg9gJgpgziAXAbVABwnAlgFyxMJZARgBoAGAXVJADcBDAGwFcYkQAdDgWwCM0AegCYAFFwCOaAJQgAvqXSZc+QiiGli1Ok1bsAmgOBcoAM1liOkmfMXY8BIgGYNWhizaJOPfgOIWrcgogGHYqROpCrjoeXtwA7vSBtsoOKOQuNG66nlx89ML+0ly89ADGANZwjPRwABZccDjA9LwAvAC0xLJJwUr2qiQZ2u7suSW+hdZaMFAA5vBEoCYAThDcSGQgOBBIziC1MPRQ7DhxEAdHCDYgK2sbNNtI6sPZIKU9t+uIe4+Iz4xYMAxKA1A7HTLRUYcErLAAE72unyQ6S2O0QmyyMS4+BwiURqy+KN+ewux08p3OhygCAhIxyHDQWBANGqvBgjAACn0wp5llhZrUcB8CciHmjnpj2KUuKUsMtpRwcXigkjEAAWMWil5YnjMYV3dWa76014AEQA+sB2rJDFwYGhsIwCKRcsxuiyWuyuaFUiA+QKhfiDQBWI0oyX0sB6oNfUOo3YmmIWq02owce2O51cKPukCsr3c32A7CwOSUWRAA
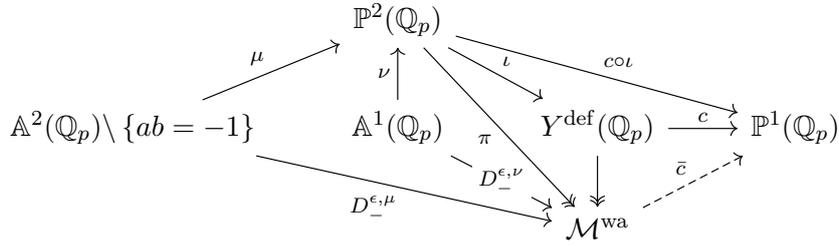
\begin{tikzcd}
                                                                                         & \mbp^2(\qp) \arrow[rd, "\iota"] \arrow[rdd, "\pi"', two heads] \arrow[rrd, "c\circ\iota"] &                                                  &             \\
\mba^2(\qp)\backslash\st{ab=-1} \arrow[ru, "\mu"] \arrow[rrd, "{D_{-}^{\epsilon,\mu}}"'] & \mba^1(\qp) \arrow[u, "\nu"] \arrow[rd, "{D_{-}^{\epsilon,\nu}}" description]             & \adlc(Y(\qp)) \arrow[d, two heads] \arrow[r, "c"] & \mbp^1(\qp) \\
                                                                                         &                                                                                           & \mwa \arrow[ru, "\bar c", dashed]                &            
\end{tikzcd}
$$
Denote $\{\infty\}=\mbp^2(\qp)\backslash\mathrm{im}~\! i$. We can verify that $c\circ\iota$ is given by
\begin{equation}\label{c}
    [x:y:z]\mapsto-\frac{x^2-\epsilon pxz+p^2z^2}{\epsilon pz^2+y^2-xz}-\epsilon p.
\end{equation}
Consider the decomposition
\begin{equation}\label{decompmbp}
    \mbp^2(\qp)=U_{\infty}\sqcup U_{\mu}\sqcup U_{\nu}\sqcup\st{P_{\pm}}
\end{equation}
        where $P_{\pm}=[(\epsilon\pm1)p:0:1]\in\mbp^2(\qp)$ and the subsets $U_{\infty}$, $U_{\mu}$ and $U_{\nu}$ are defined as
\begin{align*}
    U_{\infty}&=\st{P\in\mbp^2(\qp)~\!|~\!c(\iota(P))=\infty},\\
    U_{\mu}&=\st{[x:y:z]\in\mbp^2(\qp)~\!|~\!y\neq 0}\backslash U_{\infty},\\
    U_{\nu}&=\st{[a:0:1]~\!|~\!a\in\qp\backslash\st{\epsilon p,(\epsilon\pm1)p}}.
\end{align*}
    If $c(\iota(P))=\infty$ for some $P=[x:y:z]\in\mbp^2(\qp)$, we claim that $\fil_1(D_P)\cap\varphi(\fil_1(D_P))\neq 0$. Indeed, let $v:=(xy,xz,yz,z^2)\in\fil_1(D_P)\subset\qp^4$, we have
    $$\varphi(v)=(-z^2,xy,xz-\epsilon pz^2,yz)\in\qp^4.$$
    Since $c(\iota(P))=\infty$, by (\ref{c}) we have $xz-\epsilon pz^2=y^2$ and it's clear that $0\neq\varphi(v)\in\fil_1(D_P)$ and our assertion follows. Invoking Theorem \ref{mainclass}, $D=D_{\epsilon p}^{\epsilon,\nu}$ represents the only isomorphism class such that $\fil_1(D)\cap\varphi(\fil_1(D))\neq 0$. Therefore, $D_P$ is isomorphic to $D_{\epsilon p}^{\epsilon,\nu}$.

    Let $P,P'\in\mbp^2(\qp)\backslash U_{\infty}$ such that $D_P\simeq D_{P'}$. Then, $c(P)=c(P')$ by the intrinsic Definition \ref{func}. Therefore, the map $c$ factors through $\pi$, which gives the desired map $\bar c:\mwa\to\mbp^1(\qp)$.

    To verify that $\bar c$ is injective outside $\bar c^{-1}(i(-\epsilon p))$, we analyse its behavior on each part of (\ref{decompmbp}) separately. Denote by $V=\st{(a,b)\in\qp^2~\!|~\! ab\neq -1}\subset\qp^2$. We can verify directly that $\mu$ gives an isomorphism between $V$ and $U_{\mu}$ 
    %whose inverse is given by $$[x:y:z]\mapsto\br{-\frac{xy^2+p^2z^3}{y(\epsilon pz^2+y^2-xz)},-\frac{z}{y}}.$$
    and that $c(\iota(\mu(a,b)))=-\frac{a^2+\epsilon p+b^2p^2}{ab+1}$, which coincides with the notation $c$ in Theorem \ref{mainclass}. For $(a_1,b_1),(a_2,b_2)\in V$, denote $S_i=\br{\begin{smallmatrix}
            -b_i & a_i\\
            \frac{a_i+b_ic_i}{p^2} & b_i
        \end{smallmatrix}}$ and $M_i=\br{\begin{smallmatrix}
            0 & -p^2\\
            1 & c_i
        \end{smallmatrix}}\in\mat_{2\times 2}(\qp)$ where $c_i=c(a_i,b_i)$. We have the relations 
        \begin{equation}\label{rel}
            M_iS_i+S_iM_i=c_iS_i\qaq(S_i)^2=-\frac{\epsilon p+c_i}{p^2}I_2.
        \end{equation} 
        If $c_1=c_2\neq-\epsilon p,\pm2p$, we claim that $D_{\mu(a_1,b_1)}\simeq D_{\mu(a_2,b_2)}$. In this case we have $M_1=M_2=:M$.

        \begin{lem}
            There exists $H\in Z(M)\cap\GL_2(\qp)$ such that $S_2=HS_1H^{-1}$.
        \end{lem}
        \begin{proof}
            Let $E=\qp[M]$. By assumption, $E/\qp$ is \'etale and denote by $t\mapsto\bar t$ the nontrivial involution of $E$ induced by $M\mapsto c-M$. By (\ref{rel}), for $i=1,2$ we have $S_iM=\bar MS_i$. Therefore, we have $S_it=\bar tS_i$ for any $t\in E$. Denote $u=S_2S_1^{-1}\in \GL_2(\qp)$. For any $t\in E$, we have 
            $$ut=S_2S_1^{-1}t=S_2\bar tS_1^{-1}=tS_2S_1^{-1}=tu.$$ 
            It follows that $u\in E^\times$ since $E^\times$ is a maximal commutative subalgebra in $\mat_2(\qp)$. Moreover, we have
            $$S_2^2=uS_1uS_1=u\bar uS_1^2.$$
            Since $S_1^2=S_2^2=-\frac{\epsilon p+c}{p^2}\neq 0$, we have $u\bar u=1$. By Hilbert 90 we have $u=H\bar H^{-1}$ for some $H\in E^\times=Z(M)\cap\GL_2(\qp)$, which yields
            $$S_2=H\bar H^{-1}S_1=HS_1H^{-1}.$$
        \end{proof}
        % By (\ref{rel}), when $M_1=M_2$ is diagonalized over $\bqp$, the same conjugation makes $S_i$ into $\widetilde S_i:=\br{\begin{smallmatrix}
        %     0 & \alpha_i\\
        %     \beta_i & 0
        % \end{smallmatrix}}$ for some $\alpha_i,\beta_i\in\bqp^\times$, $i=1,2$. We have $\alpha_1\beta_1=\alpha_2\beta_2=\frac{\epsilon p+c(a,b)}{p^2}\neq 0$ by our assumption. It follows that $\widetilde S_1$ is conjugated to $\widetilde S_2$ by $\br{\begin{smallmatrix}
        %     \alpha_1 & 0\\
        %     0 & \beta_2
        % \end{smallmatrix}}$, which commutes with $\widetilde M$.
        By the lemma above, we have $S_2=HS_1H^{-1}$ for some $H\in Z(M_1)(\qp)$. It follows that the matrices $\mat_{\mathcal B}(\varphi_{D_{(a_i,b_i)}^{\epsilon,\mu}})=\br{\begin{smallmatrix}
            0 & M_i\\
            I_2 & S_iM_i
        \end{smallmatrix}}$ $(i=1,2)$ are conjugate via $\br{\begin{smallmatrix}
            H & 0\\
            0 & H
        \end{smallmatrix}}$, which proves that $D_{(a_1,b_1)}^{\epsilon,\mu}\simeq D_{(a_2,b_2)}^{\epsilon,\mu}$. 
        
        If $c_1=c_2=-\epsilon p$. We can conjugate simultaneously the matrices $M_1=M_2$, $S_1$ and $S_2$ into $\widetilde M=\br{\begin{smallmatrix}
            \mu_1 & 0\\
            0 & \mu_2
        \end{smallmatrix}}$, $\widetilde S_1=\br{\begin{smallmatrix}
            0 & \lambda_2\\
            \lambda_1 & 0
        \end{smallmatrix}}$ and $\widetilde S_2=\br{\begin{smallmatrix}
            0 & \lambda_2'\\
            \lambda_1' & 0
        \end{smallmatrix}}$ where $\mu_i$ are roots of $X^2+\epsilon pX+p^2$ and $\lambda_i,\lambda_i'\in\qp$. By assumption, $\lambda_1\lambda_2=\lambda_1'\lambda_2'=0$. In this case, the isomorphism class of $D_{(a_i,b_i)}^{\epsilon,\mu}$ depends on the $Z(\widetilde M)$-conjugation class of $\widetilde S_i$. Since $Z(\widetilde M)$ is the canonical maximal torus in $\GL_2$, there are three such classes, represented by $\br{\begin{smallmatrix}
            0 & 0\\
            1 & 0
        \end{smallmatrix}}$, $\br{\begin{smallmatrix}
            0 & 1\\
            0 & 0
        \end{smallmatrix}}$ and $\br{\begin{smallmatrix}
            0 & 0\\
            0 & 0
        \end{smallmatrix}}$. Using (\ref{c}), the three conjugation classes above are respectively corresponding to the cases 
        \begin{equation}\label{cases}
            \frac{x}{pz}=-\mu_1,~\frac{x}{pz}=-\mu_2\qaq x=z=0
        \end{equation}
        which gives the three preimages under $\bar c$ of $i(-\epsilon p)$.

        For any $P\in U_{\nu}$, we claim that there exists $P'\in U_{\mu}$ such that $\pi(P)=\pi(P')$. Note that for $P=[x:0:z]$, $c(P)\neq\infty$ if and only if $z\neq 0$ and $x\neq\epsilon pz$. The function $a\mapsto c([a:0:1])$ is clearly identified with $a\mapsto a-\epsilon p+\frac{p^2}{a-\epsilon p}$. It suffices to prove that, for any $a'\in\qp\backslash\st{\epsilon p,(\epsilon\pm1)p}$, there exists $(a,b)\in\qp^2$ with $ab+1\neq 0$ such that $D_{a'}^{\epsilon,\nu}\simeq D_{(a,b)}^{\epsilon,\mu}$. Let $(e_i)_{1\le i\le 4}$ be the canonical adapted basis of $D_{a'}^{\epsilon,\nu}$. Denote $\mathcal B'=(e_2,e_1,\varphi(e_2),\varphi(e_1))$, the matrix of $\varphi_{D_{a'}^{\epsilon,\nu}}$ under $\mathcal B'$ is
        $$\mat_{\mathcal B'}(\varphi_{D_{a'}^{\epsilon,\nu}})=\br{\begin{smallmatrix}
            0 & 0 & a'-\epsilon p & 0\\
            0 & 0 & 0 & \frac{p^2}{a'-\epsilon p}\\
            1 & 0 & 0 & \frac{1}{a'-\epsilon p}\\
            0 & 1 & -p^2-a'^2+\epsilon a'p & 0
        \end{smallmatrix}}.$$ 
        Denote $H=\br{\begin{smallmatrix}
            1 & a'-\epsilon p\\
            1 & \frac{p^2}{a'-\epsilon p}
        \end{smallmatrix}}$. By the assumption that $a'\in\qp\backslash\st{\epsilon p,(\epsilon\pm1)p}$, the matrix $H$ is well-defined and invertible. The conjugation of $\mat_{\mathcal B'}(\varphi_{D_{a'}^{\epsilon,\nu}})$ by $\br{\begin{smallmatrix}
            H & 0\\
            0 & H
        \end{smallmatrix}}$ has the form of $\varphi_{D_{(a,b)}^{\epsilon,\mu}}$ given in Definition \ref{canfam}, which gives an isomorphism between $D_{a'}^{\epsilon,\nu}$ and certain $D_{(a,b)}^{\epsilon,\mu}$ and verifies our assertion.

        Finally, for $\epsilon'=\pm 1$, $c^{-1}(2\epsilon'p)$ is the single point $P_{\pm}=[(\epsilon\pm1)p:0:1]$, corresponding to the only isomorphism class $D_{\epsilon'}^{\epsilon,iso}$. This completes our proof.
\end{proof}
\begin{cor}\label{distwintype}
    Let $P\in \adlc(Y(\qp))$. The Wintenberger type of $D_P$ satisfies
    $$X_{D_P}=\begin{cases}
        (A)\quad &v_p(\bar c(P))\le 0\\
        (B) &v_p(\bar c(P))\ge 1.
    \end{cases}$$
\end{cor}
\begin{proof}
    Using Theorem \ref{propmain} and Proposition \ref{munu}, we can compute the value of $\bar c$ corresponding to the three cases in Proposition \ref{mainthm1}.

   In the above proof we see that the function $a\mapsto c([a:0:1])$ is identified with $a\mapsto a-\epsilon p+\frac{p^2}{a-\epsilon p}$. The case (1) in Proposition \ref{mainthm1} implies that $a-\epsilon p\in p^2\zp$, which yields $v_p(\bar c)\le 0$. In the proof of Proposition \ref{mainthm1} (2), we have shown that $c(a,b)\in p\zp$, which implies that $v_p(\bar c)\ge 1$ in this case. Finally, case (3) corresponds to $\bar c=\pm2p$, which completes our proof.
\end{proof}
\begin{rem}\label{diagramp2}
For illustrating the decomposition (\ref{decompmbp}), we represent $\mbp^2(\qp)$ as the disjoint union of 
$$\mathbb A^2(\qp)=\st{[x:y:z]\in\mbp^2(\qp)~|~y\neq 0}\qaq\mbp^1(\qp)=\st{[x:0:z]\in\mbp^2(\qp)}$$
in the following diagram, where the open disk corresponds to $\mathbb A^2(\qp)$ and half of its boundary corresponds to $\mbp^1(\qp)$. This is merely a topological model (recall that $\mbp^2(\mathbb R)$ can be obtained from the closed disk $\mathbb D^2$ by identifying antipodal points on the boundary), intended to make the subsequent results more transparent.
\begin{center}
\tikzset{every picture/.style={line width=0.75pt}} %set default line width to 0.75pt        

\begin{tikzpicture}[x=0.75pt,y=0.75pt,yscale=-1,xscale=1]

%Shape: Arc [id:dp7451794816713618] 
\draw  [draw opacity=0] (244.01,229.64) .. controls (192.32,228.74) and (150.69,186.78) .. (150.69,135.15) .. controls (150.69,83.28) and (192.7,41.18) .. (244.72,40.66) -- (245.69,135.15) -- cycle ; \draw   (244.01,229.64) .. controls (192.32,228.74) and (150.69,186.78) .. (150.69,135.15) .. controls (150.69,83.28) and (192.7,41.18) .. (244.72,40.66) ;  
%Shape: Arc [id:dp07325736239625458] 
\draw  [draw opacity=0][dash pattern={on 4.5pt off 4.5pt}] (244.68,229.65) .. controls (245.02,229.65) and (245.35,229.65) .. (245.69,229.65) .. controls (298.16,229.65) and (340.69,187.34) .. (340.69,135.15) .. controls (340.69,82.96) and (298.16,40.65) .. (245.69,40.65) .. controls (245.02,40.65) and (244.36,40.66) .. (243.69,40.67) -- (245.69,135.15) -- cycle ; \draw  [dash pattern={on 4.5pt off 4.5pt}] (244.68,229.65) .. controls (245.02,229.65) and (245.35,229.65) .. (245.69,229.65) .. controls (298.16,229.65) and (340.69,187.34) .. (340.69,135.15) .. controls (340.69,82.96) and (298.16,40.65) .. (245.69,40.65) .. controls (245.02,40.65) and (244.36,40.66) .. (243.69,40.67) ;  
%Shape: Circle [id:dp3702488128058149] 
\draw  [color={rgb, 255:red, 0; green, 0; blue, 0 }  ,draw opacity=1 ][fill={rgb, 255:red, 0; green, 0; blue, 0 }  ,fill opacity=1 ] (177.68,67.3) .. controls (177.68,65.93) and (178.8,64.81) .. (180.17,64.81) .. controls (181.55,64.81) and (182.67,65.93) .. (182.67,67.3) .. controls (182.67,68.68) and (181.55,69.8) .. (180.17,69.8) .. controls (178.8,69.8) and (177.68,68.68) .. (177.68,67.3) -- cycle ;
%Shape: Circle [id:dp5785433549488034] 
\draw  [color={rgb, 255:red, 0; green, 0; blue, 0 }  ,draw opacity=1 ][fill={rgb, 255:red, 0; green, 0; blue, 0 }  ,fill opacity=1 ] (177.68,203.3) .. controls (177.68,201.93) and (178.8,200.81) .. (180.17,200.81) .. controls (181.55,200.81) and (182.67,201.93) .. (182.67,203.3) .. controls (182.67,204.68) and (181.55,205.8) .. (180.17,205.8) .. controls (178.8,205.8) and (177.68,204.68) .. (177.68,203.3) -- cycle ;
%Straight Lines [id:da8708218309340889] 
\draw    (215.17,45.8) -- (275.67,224.67) ;
%Straight Lines [id:da883687590345868] 
\draw    (273.67,45.67) -- (217.67,224.67) ;
%Shape: Circle [id:dp042207187829697124] 
\draw  [color={rgb, 255:red, 0; green, 0; blue, 0 }  ,draw opacity=1 ][fill={rgb, 255:red, 255; green, 255; blue, 255 }  ,fill opacity=1 ] (241.67,136.39) .. controls (241.67,134.03) and (243.57,132.13) .. (245.92,132.13) .. controls (248.27,132.13) and (250.18,134.03) .. (250.18,136.39) .. controls (250.18,138.74) and (248.27,140.64) .. (245.92,140.64) .. controls (243.57,140.64) and (241.67,138.74) .. (241.67,136.39) -- cycle ;
%Shape: Circle [id:dp8194329829205944] 
\draw  [color={rgb, 255:red, 0; green, 0; blue, 0 }  ,draw opacity=1 ][fill={rgb, 255:red, 0; green, 0; blue, 0 }  ,fill opacity=1 ] (243.43,136.39) .. controls (243.43,135.01) and (244.55,133.89) .. (245.92,133.89) .. controls (247.3,133.89) and (248.42,135.01) .. (248.42,136.39) .. controls (248.42,137.76) and (247.3,138.88) .. (245.92,138.88) .. controls (244.55,138.88) and (243.43,137.76) .. (243.43,136.39) -- cycle ;
%Curve Lines [id:da02846888593990382] 
\draw    (154,110) .. controls (194,80) and (271.67,233) .. (311.67,203) ;

% Text Node
\draw (215,80.4) node [anchor=north west][inner sep=0.75pt]    {$l_{1}$};
% Text Node
\draw (232,131.4) node [anchor=north west][inner sep=0.75pt]    {$o$};
% Text Node
\draw (263,80.4) node [anchor=north west][inner sep=0.75pt]    {$l_{2}$};
% Text Node
\draw (154,54.4) node [anchor=north west][inner sep=0.75pt]    {$P_{+}$};
% Text Node
\draw (154,198.4) node [anchor=north west][inner sep=0.75pt]    {$P_{-}$};
% Text Node
\draw (172,136.4) node [anchor=north west][inner sep=0.75pt]    {$c=\infty $};
% Text Node
\draw (140,241.4) node [anchor=north west][inner sep=0.75pt]    {$\mbp^2(\qp)=U_{\infty}\sqcup U_{\mu}\sqcup U_{\nu}\sqcup\st{P_{\pm}}$};

\end{tikzpicture}
\end{center}
\end{rem}

\begin{rem}\label{intersection}
    In the proof, we can see that if $X^2+\epsilon pX+p^2$ has no solution in $\qp$, the set $c^{-1}(-\epsilon p)$ is a single point $o=[0:1:0]$ and thereby $\bar c$ is injective on $\mwa$. On the other hand, if $\mu_{i}\in\qp$ $(i=1,2)$ are the two different solutions of $X^2+\epsilon pX+p^2$, then $c^{-1}(-\epsilon p)$ consists of two projective lines $l_1,l_2$ in $\mbp^2(\qp)$ which pass $o$ and are of slopes $-\mu_1p,-\mu_2p$ respectively. Moreover, $l_1\backslash\st{o},l_2\backslash\st{o}$ and $\st{o}$ belong to three different $Z_{\varphi}(\qp)$-orbits respectively. Therefore, under the quotient topology of $\mwa$, the points $\pi(l_1)$ and $\pi(l_2)$ specialize to $\pi(o)$.
\end{rem}

\subsection{The GIT viewpoint}
Throughout this subsection, we fix an algebraically closed field $K$ over $\qp$ and work over $K$. The goal of this subsection is to prove that the rational map $c:Y_K\dashrightarrow\mbp^1_K$ defined by Definition \ref{func} becomes a projective GIT quotient. We can suppose that $\varphi$ is diagonal
$$\varphi=\br{\begin{smallmatrix}
    \lambda_1 & 0 & 0 & 0\\
    0 & \lambda_2 & 0 & 0\\
    0 & 0 & \lambda_3 & 0\\
    0 & 0 & 0 & \lambda_4
\end{smallmatrix}}\in\GL_4(K)$$ 
where $(\lambda_i)_{1\le i\le 4}\in(K^\times)^4$ are pairwise distinct nonzero elements. Denote by $G$ the algebraic group such that
$$G(R)=\st{g\in\SL_4(R)~|~g\varphi=\varphi g}$$
for any commutative ring $R$. It's clear that $G\simeq\g_m^3$ is a subtorus of the canonical maximal subtorus in $\GL_4$, which inherits the canonical action of $\GL_4$ on $Y$. 

Generally, the machinery of GIT asks for a nice (very ample) linearisation of the action, which is equivalent to a $G$-equivariant projective embedding of $X$. In our case, we should consider the Pl\"ucker embedding 
\begin{equation}\label{pluckerembedding}
    Y=\Gr(2,4)\to\mbp^5=\pj K[x_{ij}]_{1\le i<j\le 4}
\end{equation}
identify $Y$ with the closed subvariety of $\mbp^5$ defined by the Pl\"ucker relation 
$$x_{12}x_{34}-x_{13}x_{24}+x_{14}x_{23}=0.$$
Equivalently, (\ref{pluckerembedding}) induces an isomorphism $Y\simeq\pj R(Y)$ where $$R(Y)=K[x_{ij}]_{1\le i<j\le 4}/(x_{12}x_{34}-x_{13}x_{24}+x_{14}x_{23}).$$
The embedding (\ref{pluckerembedding}) becomes $G$-equivariant under the homomorphism $G\to\GL(\bigoplus_{1\le i<j\le 4} Kx_{ij})$ defined by $\diag(a_1,a_2,a_3,a_4)x_{ij}=a_ia_jx_{ij}$. It turns out that in this case, we can determine the invariant subalgebra.
\begin{prop}\label{ginvsubalg}
    The canonical morphism $K[x,y]\to R(Y)$ defined by $x\mapsto x_{12}x_{34}$ and $y\mapsto x_{14}x_{23}$ induces an isomorphism $K[x,y]\simeq R(Y)^G$, where $R(Y)^G$ is the $G$-invariant subalgebra.
\end{prop}
\begin{proof}
    Denote $S=K[x_{12}x_{34},x_{14}x_{23}]\subset R(Y)$ and $\underline{d}=(d_{ij})_{1\le i<j\le 4}$. Let
    $$D=\st{\underline{d}\in(\mathbb Z_{\ge 0})^6~\!|~\!d_{12}d_{34}=d_{13}d_{24}=d_{14}d_{23}=0}.$$
    It follows that $S\subset R(Y)^G$ and that $R(Y)$ is a free $S$-module with basis 
    $$\st{\prod_{1\le i<j\le 4}x_{ij}^{d_{ij}}~\!|~\!\underline{d}\in D}.$$ 
    The monomial $f=\prod_{1\le i<j\le 4}x_{ij}^{d_{ij}}$ is a weight vector of weight $\alpha(f)=\prod_{1\le i<j\le 4}(\alpha_i\alpha_j)^{d_{ij}}$, where $\alpha_i$ is the cocharacter of $G$ defined by
    $$\alpha_i(\diag(a_1,a_2,a_3,a_4))=a_i.$$
    Therefore, $f$ is $G$-stable if and only if $\alpha(f)$ is trivial. Consider the map
    $$\Phi:(\mathbb Z_{\ge 0})^6\to(\mathbb Z_{\ge 0})^4,~
        \underline{d}\mapsto\br{\Phi_k(\underline{d})}_{k\in[\![1,4]\!]},\mbox{ where }\Phi_k(\underline{d})=\sum_{i=1}^{k-1}d_{ik}+\sum_{j=k+1}^4d_{kj}.$$
    For $n\in\mz_{\ge 0}$, denote
    $$E_n=\st{\underline{d}\in(\mathbb Z_{\ge 0})^6~\!|~\!\Phi_k(\underline{d})=n,~\forall~k\in[\![1,4]\!]}.$$
    Then, $\alpha(f)$ is trivial if and only if $\underline{d}\in\bigsqcup_{n\ge 0}E_n$.

    To verify that $R(Y)^G=S$, it suffices to show that $D\cap E_n=\varnothing$ for any $n\ge 1$. Denote 
    $$U=\st{\underline{d}\in(\mathbb Z_{\ge 0})^6~\!|~\!\Phi_k(\underline{d})>0,~\forall k\in[\![1,4]\!]}$$
    and for $k\in[\![1,4]\!]$,
    $$V_k=\st{\underline{d}\in(\mathbb Z_{\ge 0})^6~\!|~\!\Phi_k(\underline{d})>\Phi_{m}(\underline{d}),~\forall~0\le m\le 4,~m\neq k}.$$
    For any $n\ge 1$, by definition we have $E_n\subset U$ and $E_n\cap V_k=\varnothing$. We claim that $D\cap U\subset\bigsqcup_{k\in[\![1,4]\!]}V_k$, which thereby implies that $D\cap E_n=\varnothing$. 
    
    Denote by $\Sigma$ the set consisting of the subsets of $2$ elements in $[\![1,4]\!]$ and for $A=\st{i<j}\in\Sigma$, let $d_A=d_{ij}$. For any $\underline{d}\in D\cap U$, the set $\Sigma(\underline{d})=\st{A\in\Sigma~\!|~\!d_A\neq 0}$ satisfies
    \begin{enumerate}
        \item[$\bullet$] $A\in\Sigma(\underline{d})$ implies that $A^c\not\in\Sigma(\underline{d})$
        \item[$\bullet$] $\bigcup_{A\in\Sigma(\underline{d})}A=[\![1,4]\!]$
    \end{enumerate}
    which implies that $\#\Sigma(\underline{d})=3$ and hence $\#\bigcap_{A\in\Sigma(\underline{d})}A=1$. It follows that $\bigcap_{A\in\Sigma(\underline{d})}A=\st{k}$ for a certain $k\in[\![1,4]\!]$, which implies that $\underline{d}\in V_k$ as claimed.
\end{proof}
In particular, we see that the ideal $R(Y)_+^G$ is generated by $x_{12}x_{34},x_{14}x_{23}$. By definition, the semistable set $Y^{ss}$ is given by the complement to the closed subscheme of $Y$ defined by $R(Y)_+^G$, and the GIT quotient of this action is given by $Y^{ss}\to Y/\!\!/G:=\pj R(Y)^G\simeq\mbp^1$. This is a special case of Gelfand-MacPherson Correspondence (see \cite[Theorem 2.4.7]{kapranov1992chowquotientsgrassmanniani}). We have the following application.

\begin{prop}\label{cisgit}
    There exist nonzero $G$-invariant sections $s_1,s_2\in\Gamma(Y,\mo_{Y}(2))^G$ such that
    $$\Gamma(Y,\mo_{Y}(2))^G=K s_1\oplus K s_2\qaq c=[s_1:s_2].$$
    In particular, we have $\dm_{c}=Y^{ss}$ both being the complement to the common vanishing locus of $s_1$ and $s_2$. The rational map $c:Y\dashrightarrow\mbp^1_{K}$ is the projective GIT quotient with respect to $\mo_{Y}(1)$. 
\end{prop}
\begin{proof}
Over $K$, we can diagonalize the matrix of $\varphi$ into $\diag(\lambda_1,\lambda_2,\lambda_3,\lambda_4)$ and use the Pl\"ucker embedding (\ref{pluckerembedding}). Compute by definition, for $[\underline x]\in V(x_{12}x_{34}-x_{13}x_{24}+x_{14}x_{23})\subset\mbp^5(K)$, we have
$$
c([\underline x]) = \frac{(\lambda_1\lambda_3+\lambda_2\lambda_4)(\lambda_1-\lambda_3)(\lambda_2-\lambda_4)x_{14}x_{23} - (\lambda_1\lambda_4+\lambda_2\lambda_3)(\lambda_1-\lambda_4)(\lambda_2-\lambda_3)x_{13}x_{24}}{-(\lambda_1-\lambda_3)(\lambda_2-\lambda_4)x_{14}x_{23} + (\lambda_1-\lambda_4)(\lambda_2-\lambda_3)x_{13}x_{24}}.
$$
Note that by Proposition \ref{ginvsubalg}, $x_{14}x_{23}$ and $x_{13}x_{24}$ consist of a $K$-basis of $\Gamma(Y,\mo_{Y}(2))^G$. It follows that $c$ is induced by $G$-invariant sections. Moreover,
$$
\left|\begin{smallmatrix}
(\lambda_1\lambda_3+\lambda_2\lambda_4)(\lambda_1-\lambda_3)(\lambda_2-\lambda_4) & -(\lambda_1\lambda_4+\lambda_2\lambda_3)(\lambda_1-\lambda_4)(\lambda_2-\lambda_3)\\
-(\lambda_1-\lambda_3)(\lambda_2-\lambda_4) & (\lambda_1-\lambda_4)(\lambda_2-\lambda_3)
\end{smallmatrix}\right|
=\prod_{1\le i<j\le 4}(\lambda_i-\lambda_j)\neq 0.$$
Therefore, the morphism $c$ differs by an automorphism of $\mbp^1$ from the projective GIT quotient obtained from Proposition \ref{ginvsubalg}.
\end{proof}

\begin{proof}[Proof of Lemma \ref{dfwa}]
    From the proof of Proposition \ref{cisgit}, we see that for any $\bqp$-algebra $K$,
    $$
        \dm_c(K)=Y(K)\backslash\st{[\underline x]\in Y(K)~\!|~\! x_{12}x_{34}=x_{13}x_{24}=0}
    $$
    and that
    $$\st{x_{12}x_{34}=x_{13}x_{24}=0}=\bigcup_{1\le k\le 4}\underbrace{\st{P\in Y(K)~\!|~\!V_P\subset{\textstyle\bigoplus_{i\neq k}}Ke_i}}_{Z_k(K)}\cup\bigcup_{1\le k\le 4}\underbrace{\st{P\in Y(K)~\!|~\!e_k\in V_P}}_{W_k(K)}.$$
Note that this decomposition is defined over $\qp$. Let $P\in Y(\qp)$ and denote by $V_P\subset\qp^4$ the $2$-dimensional subspace defined by $P$. If $P\in Z_k(\qp)$, we have $\dim\sum_{i\ge 0}\varphi^i(V_P)\le 3$. Note that $\sum_{i\ge 0}\varphi^i(V_P)$ is $\varphi$-stable and we deduce that $\dim\sum_{i\ge 0}\varphi^i(V_P)=2$ since the polynomial $\chi_{\varphi}$ has no factor of degree $3$ over $\qp$. It follows that $V_P$ is $\varphi$-stable. If $P\in W_k(\qp)$, $V_P\otimes\bqp$ contains an eigenvector $v$ of $\varphi$. It follows that $V_P\otimes\bqp$ is spanned by $\st{\sigma(v)}_{\sigma\in\Gamma_{\qp}}$, which implies that $V_P$ is $\varphi$-stable. 

On the other hand, suppose that $P\in Y(\qp)$ corresponds to a $\varphi$-stable subspace $V_P\subset\qp^4$. We claim that $V_P\otimes\bqp$ has the form $\ang{e_m,e_n}$ for some $1\le m\neq n\le 4$. Indeed, if $v=\sum a_ie_i\in V_P\otimes\bqp$ with $a_k\in\bqp^\times$ for some $1\le k\le 4$, we have 
$$\prod_{i\neq k}(\varphi-\lambda_i)\cdot v=\prod_{i\neq k}(\lambda_k-\lambda_i)a_ke_k\in V_P\otimes\bqp.$$
Therefore, $V_P\otimes\bqp$ is spanned by basis vectors, and our assertion follows from the dimension counting. In particular, we have $P\in \st{x_{12}x_{34}=x_{13}x_{24}=0}$.

By definition, $\adlc(Y(\qp))$ is the complement of the $\varphi$-stable points, and Lemma \ref{dfwa} follows.
\end{proof}
\subsection{Distribution of monodromy groups} 
Invoking the results in Section \ref{classmonodromy}, we can determine the neutral components of the algebraic monodromy groups associated with the isomorphism class $D_P$ for any $P\in\mwa$. Combining Theorem \ref{mainthm}, Theorem \ref{propmain} and Remark \ref{intersection}, we have the following result.
\begin{thm}\label{distmono}
    We have $H_{D_{P}}^\circ\simeq_{\bqp}\GL_{2}\times_{\det}\GL_{2}$ for all but finitely many $P\in\mwa$. The exceptions are 
    $$H_{D_{P}}^\circ\simeq_{\bqp}\begin{cases}
        \GL_{2}&\bar c(P)=i(\pm2p)\\
        \g_{m}^2~\mbox{or}~\g_{a}^2\rtimes_g\g_{m}^2&\bar c(P)=i(-\epsilon p)\\
        \g_{m}^3 &\epsilon=0\mbox{ and }\bar c(P)=\infty
    \end{cases}$$
    where $g(s,t)=\br{\begin{smallmatrix}
        s & 0\\
        0 & t
    \end{smallmatrix}}$. When $\bar c(P)=i(-\epsilon p)$, in the context of Remark \ref{intersection}, the case $H_{D_{P}}^\circ\simeq_{\bqp}\g_{m}^2$ occurs when $P=\pi(o)$, and $H_{D_{P}}^\circ\simeq_{\bqp}\g_{a}^2\rtimes\g_{m}^2$ when $P=\pi(x_i)$ where $x_i\in l_i\backslash\st{o}$ for $i=1,2$.
\end{thm}
\begin{proof}
    This result can be viewed as a geometric presentation of Theorem \ref{mainthm}, and it follows directly from that theorem in combination with Proposition \ref{munu}.
\end{proof}

We can also consider the distribution of algebraic monodromy groups of the filtered $\varphi$-modules associated with the points on $\mbp^2(\qp)$. The distribution for the case $\epsilon=0$ is illustrated in the following diagram.
\begin{center}
\tikzset{every picture/.style={line width=0.75pt}} %set default line width to 0.75pt        

\begin{tikzpicture}[x=0.75pt,y=0.75pt,yscale=-1,xscale=1]

%Shape: Arc [id:dp359561006969801] 
\draw  [draw opacity=0] (264.01,249.64) .. controls (212.32,248.74) and (170.69,206.78) .. (170.69,155.15) .. controls (170.69,103.28) and (212.7,61.18) .. (264.72,60.66) -- (265.69,155.15) -- cycle ; \draw   (264.01,249.64) .. controls (212.32,248.74) and (170.69,206.78) .. (170.69,155.15) .. controls (170.69,103.28) and (212.7,61.18) .. (264.72,60.66) ;  
%Shape: Arc [id:dp37708628630220264] 
\draw  [draw opacity=0][dash pattern={on 4.5pt off 4.5pt}] (264.68,249.65) .. controls (265.02,249.65) and (265.35,249.65) .. (265.69,249.65) .. controls (318.16,249.65) and (360.69,207.34) .. (360.69,155.15) .. controls (360.69,102.96) and (318.16,60.65) .. (265.69,60.65) .. controls (265.02,60.65) and (264.36,60.66) .. (263.69,60.67) -- (265.69,155.15) -- cycle ; \draw  [dash pattern={on 4.5pt off 4.5pt}] (264.68,249.65) .. controls (265.02,249.65) and (265.35,249.65) .. (265.69,249.65) .. controls (318.16,249.65) and (360.69,207.34) .. (360.69,155.15) .. controls (360.69,102.96) and (318.16,60.65) .. (265.69,60.65) .. controls (265.02,60.65) and (264.36,60.66) .. (263.69,60.67) ;  
%Shape: Circle [id:dp34075134049622724] 
\draw  [color={rgb, 255:red, 0; green, 0; blue, 0 }  ,draw opacity=1 ][fill={rgb, 255:red, 0; green, 0; blue, 0 }  ,fill opacity=1 ] (197.68,87.3) .. controls (197.68,85.93) and (198.8,84.81) .. (200.17,84.81) .. controls (201.55,84.81) and (202.67,85.93) .. (202.67,87.3) .. controls (202.67,88.68) and (201.55,89.8) .. (200.17,89.8) .. controls (198.8,89.8) and (197.68,88.68) .. (197.68,87.3) -- cycle ;
%Shape: Circle [id:dp07626022594571147] 
\draw  [color={rgb, 255:red, 0; green, 0; blue, 0 }  ,draw opacity=1 ][fill={rgb, 255:red, 0; green, 0; blue, 0 }  ,fill opacity=1 ] (197.68,223.3) .. controls (197.68,221.93) and (198.8,220.81) .. (200.17,220.81) .. controls (201.55,220.81) and (202.67,221.93) .. (202.67,223.3) .. controls (202.67,224.68) and (201.55,225.8) .. (200.17,225.8) .. controls (198.8,225.8) and (197.68,224.68) .. (197.68,223.3) -- cycle ;
%Straight Lines [id:da6753982195377196] 
\draw    (235.17,65.8) -- (295.67,244.67) ;
%Straight Lines [id:da5109488059224139] 
\draw    (293.67,65.67) -- (237.67,244.67) ;
%Shape: Circle [id:dp6725561032214642] 
\draw  [color={rgb, 255:red, 0; green, 0; blue, 0 }  ,draw opacity=1 ][fill={rgb, 255:red, 255; green, 255; blue, 255 }  ,fill opacity=1 ] (261.67,156.39) .. controls (261.67,154.03) and (263.57,152.13) .. (265.92,152.13) .. controls (268.27,152.13) and (270.18,154.03) .. (270.18,156.39) .. controls (270.18,158.74) and (268.27,160.64) .. (265.92,160.64) .. controls (263.57,160.64) and (261.67,158.74) .. (261.67,156.39) -- cycle ;
%Shape: Circle [id:dp3983651749340189] 
\draw  [color={rgb, 255:red, 0; green, 0; blue, 0 }  ,draw opacity=1 ][fill={rgb, 255:red, 0; green, 0; blue, 0 }  ,fill opacity=1 ] (263.43,156.39) .. controls (263.43,155.01) and (264.55,153.89) .. (265.92,153.89) .. controls (267.3,153.89) and (268.42,155.01) .. (268.42,156.39) .. controls (268.42,157.76) and (267.3,158.88) .. (265.92,158.88) .. controls (264.55,158.88) and (263.43,157.76) .. (263.43,156.39) -- cycle ;
%Curve Lines [id:da828568591639107] 
\draw    (174,130) .. controls (214,100) and (291.67,253) .. (331.67,223) ;
%Straight Lines [id:da6374058719418172] 
\draw    (257.67,33) -- (246.67,90) ;
%Straight Lines [id:da4762422915028409] 
\draw    (270.67,33) -- (283.67,90) ;
%Straight Lines [id:da10877582916235773] 
\draw    (274,160) -- (370.67,206) ;

% Text Node
\draw (235,100.4) node [anchor=north west][inner sep=0.75pt]    {$l_{1}$};
% Text Node
\draw (252,151.4) node [anchor=north west][inner sep=0.75pt]    {$o$};
% Text Node
\draw (283,100.4) node [anchor=north west][inner sep=0.75pt]    {$l_{2}$};
% Text Node
\draw (158,72.4) node [anchor=north west][inner sep=0.75pt]    {$\GL_{2}$};
% Text Node
\draw (197,147.4) node [anchor=north west][inner sep=0.75pt]    {$\g_{m}^3$};
% Text Node
\draw (373,201.4) node [anchor=north west][inner sep=0.75pt]    {$\g_{m}^2$};
% Text Node
\draw (158,217.4) node [anchor=north west][inner sep=0.75pt]    {$\GL_{2}$};
% Text Node
\draw (233,11.4) node [anchor=north west][inner sep=0.75pt]    {$\g_{a}^2\rtimes\g_{m}^2$};
% Text Node
\draw (296,129.4) node [anchor=north west][inner sep=0.75pt]    {$\GL_{2}\times_{\det}\GL_{2}$};
\draw (102,257) node [anchor=north west][inner sep=0.75pt]   [align=left] {Distribution of monodromy groups on $\mbp^2(\qp)$ when $\epsilon=0$};

\end{tikzpicture}
\end{center}

%\nocite{*}
\bibliographystyle{amsalpha}
\bibliography{bib}

@misc{volkov2023abelian,
  title={Abelian surfaces with supersingular good reduction and non semisimple {Tate} module},
  author={Volkov, Maja},
  howpublished={arXiv:2301.05564},
  year={2023}
}

@article{pink1998ℓ,
  title={$\ell$-adic algebraic monodromy groups, cocharacters, and the {Mumford-Tate} conjecture},
  author={Pink, Richard},
  journal={J. Reine Angew. Math.},
  pages={197--237},
  year={1998}
}

@article{wintenberger1984scindage,
  title={Un scindage de la filtration de {Hodge} pour certaines vari{\'e}t{\'e}s alg{\'e}briques sur les corps locaux},
  author={Wintenberger, Jean-Pierre},
  journal={Annals of Math.},
  volume={119},
  number={3},
  pages={511--548},
  year={1984}
}

@book{milne2017algebraic,
  title={Algebraic groups: the theory of group schemes of finite type over a field},
  author={Milne, James S.},
  year={2017},
  volume={170},
  publisher={Cambridge University Press}
}

@article{volkov2005class,
  title={A class of $p$-adic {Galois} representations arising from abelian varieties over $\qp$},
  author={Volkov, Maja},
  journal={Math. Ann.},
  volume={331},
  pages={889--923},
  year={2005},
  publisher={Springer}
}

@misc{brinon2009cmi,
  title={{CMI} Summer School notes on $p$-adic {H}odge theory},
  author={Brinon, Olivier and Conrad, Brian},
  journal={course notes},
  year={2009}
}

@article{fontaine1979modules,
  title={Modules galoisiens, modules filtr{\'e}s et anneaux de {Barsotti-Tate}},
  author={Fontaine, Jean-Marc},
  journal={Ast{\'e}risque},
  volume={65},
  pages={3--80},
  year={1979}
}

@article{colmez2000construction,
  title={Construction des repr{\'e}sentations $p$-adiques semi-stables},
  author={Colmez, Pierre and Fontaine, Jean-Marc},
  journal={Invent. Math.},
  volume={140},
  number={1},
  pages={1--44},
  year={2000},
  publisher={Berlin, Springer-Verlag.}
}

@book{serre1997abelian,
  title={Abelian $\ell$-adic representations and elliptic curves},
  author={Serre, Jean-Pierre},
  year={1997},
  publisher={AK Peters/CRC Press}
}

@article{larsen1992independence,
  title={On $\ell$-independence of algebraic monodromy groups in compatible systems of representations},
  author={Larsen, Michael and Pink, Richard},
  journal={Invent. Math.},
  volume={107},
  number={3},
  pages={603--636},
  year={1992}
}

@article{serre2000collected,
  title={Lettres à {Ken Ribet} du 1/1/1981},
  journal={Springer Collected Works in Mathematics},
  author={Serre, J.-P.},
  year={2000}
}

@article{deligne2012tannakian,
  title={Tannakian categories},
  author={Deligne, Pierre and Milne, James S},
  journal={Lecture Notes in Mathematics},
  year={2012}
}

@misc{fu2025padicdensityconjecturehecke,
  title={Towards a p-adic nowhere density conjecture of Hecke Orbits}, 
  author={Yu Fu},
  year={2025},
  howpublished={arXiv:2506.06932v1}
}

@misc{d2022hecke,
  title={Hecke orbits on {Shimura} varieties of {Hodge} type},
  author={D'Addezio, Marco and van Hoften, Pol},
  howpublished={arXiv:2205.10344},
  year={2022}
}

@misc{kapranov1992chowquotientsgrassmanniani,
      title={Chow quotients of {G}rassmannian {I}}, 
      author={M. Kapranov},
      year={1992},
      eprint={alg-geom/9210002},
      archivePrefix={arXiv},
      primaryClass={alg-geom},
      url={https://arxiv.org/abs/alg-geom/9210002}, 
}
\end{document}